\documentclass[9pt]{article}
\usepackage{float}
\usepackage{footmisc}
\usepackage{palatino}
\usepackage{caption}
\usepackage[mathscr]{euscript}
\usepackage{comment}
\usepackage[a4paper, portrait, margin=1.1811in]{geometry}
\usepackage[english]{babel}
\usepackage{amssymb}
\usepackage{mathtools}

\usepackage{amsmath, bm}
\usepackage{amsthm}
\usepackage{tikz-cd}
\usepackage{tikz}
\usepackage[utf8]{inputenc}
\usepackage[T1]{fontenc}
\usepackage{helvet}
\usepackage{etoolbox}
\usepackage{amsmath}
\usepackage{graphicx}
\usepackage{titlesec}
\usepackage{caption}
\usepackage{booktabs}
\usepackage{natbib}
\usepackage{xcolor} 
\usepackage[colorlinks, citecolor=cyan]{hyperref}
\usepackage{caption}
\graphicspath{{./images/}}
\usepackage{scrextend}
\usepackage{fancyhdr}
\usepackage{graphicx}
\newtheorem{Lem}{Lemma}
\newtheorem{thm}{Theorem}
\newtheorem*{teo*}{Theorem}
\newtheorem{Ex}{Example}
\newtheorem*{Ex*}{Example}
\newtheorem{De}{Definition}

\newtheorem{Pro}{Proposition}
\newtheorem{Cor}{Corollary}
\newtheorem{Rem}{Remark}
\fancypagestyle{plain}{
	\fancyhf{}

}

\makeatletter
\patchcmd{\@maketitle}{\LARGE \@title}{\fontsize{16}{19.2}\selectfont\@title}{}{}
\makeatother
\usetikzlibrary{positioning}
\usetikzlibrary{arrows.meta}
\usetikzlibrary{bending}
\usetikzlibrary{arrows}
\usetikzlibrary{decorations.pathmorphing}
\usepackage{authblk}

\newsavebox\affbox
\author[1]{\textbf{Habib Benziadi}}
\author[2]{\textbf{Antonio López Almorox}}
\author[3]{\textbf{Carlos Tejero Prieto}}

\affil[1,2,3]{{Departamento de Matemáticas, $^\text{1,3}$Instituto de Física Fundamental y Matemáticas,}\newline \ {Universidad de Salamanca,  Plaza de la Merced 4, 37008, Salamanca, Spain}}

\titlespacing\section{0pt}{12pt plus 4pt minus 2pt}{0pt plus 2pt minus 2pt}
\titlespacing\subsection{12pt}{12pt plus 4pt minus 2pt}{0pt plus 2pt minus 2pt}
\titlespacing\subsubsection{12pt}{12pt plus 4pt minus 2pt}{0pt plus 2pt minus 2pt}

\titleformat{\section}{\normalfont\fontsize{10}{15}\bfseries}{\thesection.}{1em}{}
\titleformat{\subsection}{\normalfont\fontsize{10}{15}\bfseries}{\thesubsection.}{1em}{}
\titleformat{\subsubsection}{\normalfont\fontsize{10}{15}\bfseries}{\thesubsubsection.}{1em}{}

\titleformat{\author}{\normalfont\fontsize{10}{15}\bfseries}{\thesection}{1em}{}

\title{\textbf{\huge  Harmonic maps into principal bundles\\ and generalized magnetic maps }}\date{}  
\nocite{*}

\begin{document}
\pagestyle{headings}	
\newpage
\setcounter{page}{1}
\renewcommand{\thepage}{\arabic{page}}
\captionsetup[figure]{labelfont={bf},labelformat={default},labelsep=period,name={Figure }}	\captionsetup[table]{labelfont={bf},labelformat={default},labelsep=period,name={Table }}
\setlength{\parskip}{0.5em}
\newcommand{\Vol}{\operatorname{Vol}}
\newcommand{\Ad}{\operatorname{Ad}}
\newcommand{\Hom}{\operatorname{Hom}}
\newcommand{\Id}{\operatorname{Id}}
\newcommand{\Tr}{\operatorname{Tr}}
\newcommand{\Dim}{\operatorname{dim}}
\newcommand{\X}{\mathfrak{X}}
\newcommand{\Aut}{\mathrm{Aut}}
\newcommand{\cin}{{C^\infty}}
\newcommand{\R}{\mathbb{R}}
\newcommand{\Dif}{\mathrm{Diff}}
\newcommand{\Ima}{\mathrm{Im}}
\newcommand{\Tor}{\mathrm{Tor}}
\newcommand{\ad}{\mathrm{ad}}
\newcommand{\End}{\operatorname{End}}
\newcommand{\Ker}{\operatorname{Ker}}
\newcommand{\g}{\mathfrak{g}}
\maketitle
	
\noindent\rule{15cm}{0.5pt}
		\textbf{Abstract.} We study harmonic mappings from a Riemannian manifold $N$ into a principal $G$-bundle $P$ 
endowed with a $G$-invariant Riemannian  metric (i.e. a Kaluza-Klein metric).
These morphisms are called Kaluza-Klein harmonic maps and naturally lead to the notion of 
generalized magnetic maps for an arbitrary gauge group $G$, which are just their projections onto the base manifold of $P$ and might provide a geometric formulation for the magnetic interaction of extended objects modelled by $N$ under the action of a generalized Lorentz force. We provide a characterization of Kaluza-Klein harmonic maps and show that the space of generalized magnetic maps is a quotient of the space of Kaluza-Klein harmonic maps under an equivalence relation generated by an appropriate gauge group. We establish a necessary condition that they must satisfy, the gauge variation formula and the harmonic gauge fixing equation, also providing a main existence theorem for them.     After analyzing how they are influenced by the geometry of the fibers of the principal bundle, we construct several instances of generalized magnetic maps, including two non-trivial one-parameter families of examples based on $\alpha$-twisted spherical harmonic immersions with values in the complex $S^{3}\longrightarrow S^{2}$ and quaternionic $S^{7}\longrightarrow S^{4}$ Hopf fibrations, proving that among them the unique uncharged ones are the standard Clifford torus and the standard spherical harmonic immersion of $ S^3\times S^3$ into $S^7$.
        \\ \\
		\let\thefootnote\relax\footnotetext{
			\small\noindent Research partially supported by Grant PID2021-128665NB-I00 funded by MCIN/AEI/ 10.13039/501100011033 and by “ERDF A way of making Europe”; and by GIR 4139 of the Universidad de Salamanca. 			
		}\noindent
		\textbf{\textit{2020 Mathematics Subject Classification.}}      53C43, 53Z05.\\
		\textbf{\textit{Keywords}}: \textit{Kaluza-Klein harmonic maps, generalized magnetic maps, higher dimensional  Wong's equations, harmonic immersions into Hopf fibrations.}
	
\noindent\rule{15cm}{0.4pt}
\section{Introduction}
Harmonic maps  between Riemannian manifolds were  introduced by J. Eells and J.H. Sampson \cite{eells1964harmonic} and later thoroughly studied by many authors, see \cite{xin2012geometry} and the references cited therein. They are solutions to the variational problem defined by the Dirichlet energy functional. Harmonic maps play an important role both in differential geometry and in the theory of partial differential equations, see  \cite{Urakawa}. They are also used in theoretical physics to describe nonlinear sigma models as well as in string theory, see \cite{Percacci}. Riemannian geodesic curves are  particular cases of harmonic maps which arise when the initial manifold is one-dimensional. \\
On the other hand, the so-called magnetic curves or magnetic geodesics,  which represent the trajectories of classical charged particles moving under the influence of a magnetic field $F$ that interacts with the particle through the Lorentz force, have been also thoroughly considered in the literature, see for instance \cite{ARMT, ErIn} and the extensive bibliography used there. In geometric models, the space in which the particle moves is a Riemannian manifold $(M,g)$  and the magnetic field $F$ is described by a closed $2$-form on $M$. The Lorentz force is then just the antisymmetric endomorphism $\mathscr{F}$ associated to the Riemannian metric $g$ and the magnetic field $F$. With this language, magnetic curves are solutions to a system of second-order differential equations $\frac{D\dot{\gamma}}{dt}=\mathscr{F}(\dot{\gamma})$ that generalizes the well-known Lorentz equation. If Riemannian geodesic curves are the critical curves of the variational problem determined by the energy functional of the curve, then magnetic curves can also be characterized as solutions to the so-called Landau-Hall local variational problem in Riemannian manifold settings when the electromagnetic field is described as a closed $2$-form $F$, see \cite{Opava,CS}. The Landau-Hall variational problem is only local in nature, since its Lagrangian depends on the choice of a potential for $F$, and in general non-trivial topological situations, this only exists locally.

Fortunately, there is a global framework capable of dealing with general mag\-ne\-tic fields, regardless of whether they are globally exact or not.  Another key advantage of this formulation is that it can be applied to gauge fields for an arbitrary Lie group. This approach, originally introduced by Th. Kaluza \cite{kaluza2018unification} and later by O. Klein \cite{klein1926quantentheorie}  at the beginning of the last century, revolutionized unification theories in physics and is today called the Kaluza-Klein theory after the pioneering work of these authors who sought to unify gravity and electromagnetism. Since their seminal contributions, this unification approach has been followed by many authors, see for instance \cite{MKKT}. The original work of Kaluza-Klein, expressed in the framework of differential geometry, shows that a harmonic curve, that is a geodesic, in a five-dimensional principal $U(1)$-bundle endowed with the so-called Kaluza-Klein metric, when projected onto the four-dimensional spacetime base, satisfy the equations of motion of a particle in an external electromagnetic field.  Subsequently, this study was generalized locally to non-Abelian gauge groups along the spacetime directions by R. Kerner \cite{kerner1968generalization}. Later, S. K. Wong \cite{Wong} provided a complete set of equations, including also the fiber directions, for the structure group $SU(2)$. Several authors provided generalizations of Wong's equations for an arbitrary Lie group and R. Montgomery \cite{Montgomery} proved that these equations were indeed equivalent to the Euler-Lagrange equations for Kaluza-Klein\ geodesics. Therefore, Wong's equations are precisely the characterization of Kaluza-Klein harmonic curves naturally adapted to the geometry of the supporting principal bundle.  It seems that this theory has gained recently a renewed interest in connection with the Standard Model for particle physics, see for instance \cite{baptista2025test} and the reference cited therein.
The Kaluza-Klein approach inspires our work, as we seek to formulate it in the broadest class of higher-dimensional manifolds, thus providing a completely new framework in the literature for understanding and recognizing its applications in the context of harmonic maps, as well as for modelling the interaction of extended objects with arbitrary gauge fields. 

From a mathematical point of view our generalization of magnetic geodesics to higher dimensional initial manifolds, that we have termed generalized magnetic maps, will be based on the theory of harmonic maps from Riemannian manifolds into principal $G$-bundles endowed with $G$-invariant Riemannian metrics, that are precisely the so-called Kaluza-Klein metrics, see \cite{Coquereaux,Betounes}. On the other hand, we will see that the equations for Kaluza-Klein harmonic maps are the natural generalization of Wong's equations to higher dimensional initial manifolds. 

In this paper, we start by recalling that Kaluza-Klein metrics are completely determined by a principal connection (i.e., a gauge field) and Riemannian metrics on the adjoint bundle and base manifold. Moreover, we prove their unconditional existence for any principal bundle over a paracompact manifold. After this, we proceed to study the geometrical aspects of harmonic maps into Kaluza-Klein principal bundles. The main results that we prove for them  are contained in Theorems \ref{teo:vertical-tension-field}, \ref{teo:horizontal-hamonic-curvature-modified-metric} and can be summarized as follows:
    
    \begin{teo*}
    	Let $\pi: (P,\widehat{g})\longrightarrow (M,\overline{g})$  be a Kaluza-Klein principal $G$-bundle with  principal connection $\omega$. A map $\widetilde{\Phi}: (N,g)\longrightarrow (P,\widehat{g})$ is horizontally harmonic if and only if $\Phi=\pi\circ\widetilde \Phi\colon (N,g)\longrightarrow (M,\overline{g})$ satisfies
    \begin{align}\label{eq:horizontal-intro}
\tau(\Phi)&=-\pi_*\left(\Tr_{g}\big[\widetilde\Phi^{*}\widehat{g}_{\mathscr{F}^{\omega}}\big]\right)-\pi_{*}\left(\Tr_g\big[\widetilde{\Phi}^{*}\bm{T}\big]\right),
    \end{align}with $\tau(\Phi)$ the tension field of $\Phi$, $\widehat{g}_{\mathscr{F}^{\omega}}$ is the curvature modified metric obtained from the Lorentz endomorphism of $P$ and $\bm{T}$ is O'Neill's tensor describing the family  of second fundamental forms of the fibers of $\pi$.
    On the other hand, $\widetilde\Phi$ is vertically harmonic if and only if \begin{align}\label{eq:vertical-intro}
\delta^{(\nabla^{N},\widetilde\Phi^*\nabla^{P,\mathcal{V}})}(\widetilde{\Phi}^{*}\omega)=\omega\left(\Tr_g(\widetilde{\Phi}^{*}(\bm{T}))\right),
        \end{align} with $\delta^{(\nabla^{N},\widetilde\Phi^*\nabla^{P,\mathcal{V}})}$ the codifferential  defined on  $T^{*}N\otimes\widetilde{\Phi}^{*}(\mathcal{V}P)$ with respect to the Levi-Civita connection of $N$ and the pullback of the connection induced on the vertical bundle by the Levi-Civita connection of  $P$. Moreover, $\widetilde\Phi$ is harmonic if and only if (\ref{eq:horizontal-intro}) and (\ref{eq:vertical-intro}) hold.
    \end{teo*}  We also verify that equation (\ref{eq:horizontal-intro}), see Proposition \ref{pro:Lorentz-strength}, which describes the projection of Kaluza-Klein harmonic maps to the base manifold of the principal bundle, correspond to generalized magnetic maps, see Definition \ref{defi:generalized-magnetic-maps}. The results obtained by considering this principal bundle framework recover in Example \ref{ex:magnetic-curves}  the well-known equations for magnetic geodesics when the initial manifold is one-dimensional; see, for instance, \cite{del2020introduction}, \cite{kerner1968generalization}. Therefore, Kaluza-Klein harmonic maps provide a proper global geometric generalization to higher dimensional manifolds of the equations of motion of particles under the action of electromagnetic fields. That is, they might be useful for describing  the dynamics of extended objects, like strings or branes, interacting with  a gauge field.  On the other hand, equation (\ref{eq:vertical-intro}), which describes the vertical component of harmonic maps, can be seen as a natural generalization to the relative setting of harmonic maps into Lie groups, see the works by  N. Hitchin \cite{Hitchin1,Hitchin2}  and K. Uhlenbeck \cite{Uhlenbeck}.
    
The main results of the paper and its organization are as follows. In Section 2 we introduce some key concepts about harmonic maps; for more details, see \cite{baird2003harmonic}, \cite{eells1964harmonic}, \cite{xin2012geometry}. Section 3 is devoted to study Kaluza-Klein metrics on a principal bundle, showing in Theorem \ref{teo:structure-Kaluza-Klein-metrics} that they are parametrized by the space of principal connections and  the spaces of Riemannian metrics on the adjoint bundle and on the base manifold. We also prove in Corollary \ref{cor:existence-Kaluza-Klein-metrics} their unconditional existence for any principal bundle over a paracompact manifold. In Section 4, which is the backbone of the paper, we present our findings related to harmonic maps into  Kaluza-Klein principal bundles. In particular,  in Theorems \ref{teo:horizontal-component-tension-field}, \ref{teo:vertical-tension-field} we provide the explicit expressions of the Euler-Lagrange equations satisfied, respectively, by horizontally and vertically harmonic maps. We continue in Section 5 introducing in Definition \ref{defi:generalized-magnetic-maps} the generalized Lorentz equations and generalized magnetic maps. We show in Theorem \ref{teo:aplicaciones-magneticas-cociente} that the space of generalized magnetic maps is the quotient of the space of Kaluza-Klein harmonic maps under an equivalence relation  generated  the gauge group. We explain in Proposition \ref{pro:necessary-condition-generalized-magnetic-maps} a necessary condition that any generalized magnetic map has to satisfy. We present in Proposition \ref{pro:gauge-variation-harmonic-gauge-fixing} the gauge variation formula and the harmonic gauge fixing equation, whereas in Theorem \ref{teo:existence-generalized-magnetic-maps} we provide the main existence theorem for generalized magnetic maps.  In section 6 we collect in Proposition \ref{pro:local-expressions-Kaluza-Klein}
the local expressions for the main geometric objects related to Kaluza-Klein principal bundles, offering  in Proposition \ref{pro:local-expression-Wong-Lorentz} the local versions for the generalized Lorentz and Wong equations. In Section 7, we explore how the geometry of the fibers of a Kaluza-Klein principal bundle influences the equations for generalized magnetic maps. We show in Proposition \ref{pro:tensor-diferencia-conexiones} that the dependence is controlled by O'Neill's tensor $\bm{T}$, which models essentially the second fundamental forms of the fibers of the principal bundle, and by the tensor that measures the difference between the connections induced on the vertical bundle by the Levi-Civita connection of the Kaluza-Klein metric and by the principal connection. In Theorem \ref{teo:beta-cov-constant} and Corollary \ref{cor:subvariedad-metricas_KK-fibras-tot-geodesicas} we identify explicitly the subvariety of Kaluza-Klein metrics that lead to totally geodesic fibers. We use this result to give in Theorem \ref{teo:tension-vertical-modificada} and Corollary \ref{cor:equivalent-harmonic-equations} new versions of the Euler-Lagrange equations for vertical and horizontal harmonic maps.  Finally, in Section 6, we apply our results to explore several instances of generalized magnetic maps, including examples supported by the complex and quaternionic Hopf fibrations $S^{3}\longrightarrow S^{2}$ and $S^{7}\longrightarrow S^{4}$. 

\noindent
\emph{Conventions.} All manifolds considered in this paper are Hausdorff and paracompact. For any vector bundle $E\to M$ and smooth map $f\colon N\to M$ we denote by $f^{\#}\colon \Gamma(M,E)\to \Gamma(N,f^*E)$ the natural morphism of $\cin(M)$-modules that sends $s\in\Gamma(M,E)$ to $f^{\#}(s):=s\circ f\in\Gamma(N,f^*E)$.
\section{Dirichlet energy and harmonic maps}
\begin{De}
    Let $(V,g)$, $(W,\overline{g})$ be two euclidean vector spaces. The energy of a linear map $T: V\longrightarrow W$ is the real number
    \begin{align}
        \mathcal{E}_{g,\overline{g}}(T) &=\frac{1}{2}<g,T^{*}\overline{g}>=\frac{1}{2}\Tr_{g}(T^{*}\overline{g}),
    \end{align}
i.e., it is the trace of the endomorphism associated with the pair of euclidean metrics $(g,T^{*}\overline{g})$ of the real vector space $V$. The energy mapping of the euclidean vector spaces $(V,g)$, $(W,\overline{g})$:
\begin{align*}
    \mathcal{E}_{g,\overline{g}}: \text{Hom}_{\mathbb{R}}(V,W)\longrightarrow \mathbb{R}
\end{align*}
assigns to each linear map $T\in \text{Hom}_{\mathbb{R}}(V,W)$ its energy $\mathcal{E}_{g,\overline{g}}(T)$.
\end{De}
The following result follows straightforwardly. 
\begin{Lem}[Covariance of the energy mapping under automorphisms]\label{Lem:covarianza-energia} Let $(V,g)$, $(W,\overline{g})$ be two euclidean vector spaces. The natural representation of the group of automorphisms $\text{Aut}_{\mathbb{R}}(V)\times \text{Aut}_{\mathbb{R}}(W)$ on the vector space of linear maps
    \begin{align}
        \rho: \text{Aut}_{\mathbb{R}}(V)\times\text{Aut}_{\mathbb{R}}(W)\times\text{Hom}_{\mathbb{R}}(V,W)\longrightarrow\text{Hom}_{\mathbb{R}}(V,W)
    \end{align}
    is such that given $(\sigma,\tau)\in\text{Aut}_{\mathbb{R}}(V)\times\text{Aut}_{\mathbb{R}}(W), T\in\text{Hom}_{\mathbb{R}}(V,W)$ one has
    \begin{align*}
        \rho_{(\sigma,\tau)}(T):=\tau\circ T\circ\sigma^{-1},
    \end{align*}
    and satisfies the following covariance property $\mathcal{E}_{g,\overline{g}}(\tau\circ T\circ\sigma^{-1})=\mathcal{E}_{\sigma^{*}g,\tau^{*}\overline{g}}(T).$
    i.e., one has
    \begin{align*}
\mathcal{E}_{g,\overline{g}}\circ\rho_{(\sigma,\tau)}=\mathcal{E}_{\sigma^{*}g,\tau^{*}\overline{g}}.
    \end{align*}
\end{Lem}

\begin{De}
Let $(N,g)$ and $(M,\overline{g})$ be two Riemannian manifolds of dimensions $n$ and $m$, respectively. We assume that $(N,g)$ is oriented and denote its Riemannian volume element $\Vol_{g}$. A smooth map $\Phi: N\longrightarrow M$ is called harmonic if it is a critical point of the Dirichlet energy functional
$$E_{g,\bar g}\colon C^\infty(N,M)\to\mathbb R$$
defined by 
\begin{align}\label{ine}
    E_{g,\bar g}(\Phi) &= \frac{1}{2}\int_{N}\|d\Phi\|_{g,\overline g}^{2}\Vol_{g},
\end{align}
where  $\|d\Phi\|_{g,\overline g}$ is the Hilbert-Schmidt norm of the tangent map $d\Phi:=\Phi_{*}: TN\longrightarrow TM$. That is, $\frac{1}{2}\|d\Phi\|_{g,\overline g}^2\in C^\infty(N)$ is the function defined on $x\in N$  by the energy of the tangent linear map $d_x\Phi$:  $$\frac{1}{2}\|d\Phi\|_{g,\overline g}^2(x)=\mathcal{E}_{g,\overline{g}}(\Phi)(x):=\mathcal E_{g_x,\bar g_{\Phi(x)}}(d_x\Phi)=\frac{1}{2}\Tr_{g_x}((\Phi^*\overline g)_x)=\frac{1}{2}[\Tr_{g}(\Phi^*\overline g)](x).$$\
\end{De}

A straightforward computation bearing in mind Lemma (\ref{Lem:covarianza-energia}) proves the following:

\begin{Pro}[Covariance of the Dirichlet energy under diffeomorphisms]\label{pro:covariance-energy}
    Let $(N,g)$, $(M,\overline{g})$ be two Riemannian manifolds. The natural action of the group of diffeomorphisms $\text{Diff}(N)\times \text{Diff}(M)$ on the space of smooth maps 
    \begin{align*} \rho: \text{Diff}(N)\times \text{Diff}(M)\times C^{\infty}(N,M)\longrightarrow C^{\infty}(N,M)
    \end{align*}
 is defined in such a way that, given $(\sigma,\tau)\in \text{Diff}(N)\times\text{Diff}(M)$, $\Phi\in C^{\infty}(N,M)$, one has
    \begin{align*}
        \rho_{(\sigma,\tau)}(\Phi):=\tau\circ\Phi\circ\sigma^{-1},
    \end{align*}
and the Dirichlet energy functional $E_{g,\overline{g}}$ satisfies the following covariance  property with respect to it
    \begin{align*}
        E_{g,\overline{g}}(\tau\circ\Phi\circ\sigma^{-1})=E_{\sigma^{*}g,\tau^{*}\overline{g}}(\Phi);\quad \text{that is},\quad E_{g,\overline{g}}\circ\rho_{(\sigma,\tau)}=E_{\sigma^{*}g,\tau^{*}\overline{g}}. 
    \end{align*}
    
\end{Pro}

We denote by $\nabla^{N}$ and $\nabla^{M}$ the Levi-Civita connections of $(N,g)$ and $(M,\bar g)$, respectively. For a map $\Phi\colon N\to M$, let $\nabla$ be the connection  on  $T^{*}N\otimes\Phi^{*}TM$ induced by $\nabla^{N}$ and $\Phi^*\nabla^{M}$. The second fundamental form of $\Phi$ is the symmetric tensor $\mathrm{II}_{\Phi}\in\Gamma(N,S^2(T^*N)\otimes\Phi^*TM)$ given by
\begin{align}
    \mathrm{II}_{\Phi}(D_{1},D_{2}):=[\nabla (d\Phi)](D_{1},D_{2}) &=(\Phi^*\nabla^{M})_{D_{1}}\Phi_{*}(D_{2})-\Phi_{*}(\nabla^{N}_{D_{1}}D_{2}),\quad D_{1}, D_{2}\in\mathfrak{X}(N).
\end{align}
The tension field of $\Phi$ is the vector field $\tau(\Phi)\in\Gamma(N,\Phi^{*}TM)$ given by the $g$-trace of $\mathrm{II} _{\Phi}$ , that is
\begin{align}\label{mnn} 
    \tau(\Phi) &=\Tr_g \mathrm{II} _{\Phi}=\sum_{r=1}^{n}\left[(\Phi^*\nabla^{M})_{U_{r}}\Phi_{*}(U_{r})-\Phi_{*}(\nabla^{N}_{U_{r}}U_{r})\right],
\end{align}
where $\{U_{r}\}_{r=1,\ldots,n}$ is a local $g$-orthonormal frame of vector fields on $N$.
\begin{De}
     A smooth map $\Phi: (N,g)\longrightarrow (M,\overline{g})$  is said to be harmonic if 
     \begin{align*}
          \left.\frac{dE_{g,\overline g}(\Phi_{s})}{ds}\right|_{s=0}=0.
     \end{align*}
     for all compactly supported smooth variations $\{\Phi_{s}\}_{s\in I}$ of $\Phi$. We denote by $\mathbf{Har}((N,g),(M,\overline{g}))$ the set of all harmonic maps from $(N,g)$ to $(M,\overline g)$.
\end{De}
Considering the spaces of variations one immediately obtains the following  covariance result:
\begin{Pro}\label{pro:invarianza-funcional-energia-Dirichlet} If  $\Phi\in \mathbf{Har}((N,g),(M,\overline{g}))$, then for any $(\sigma,\tau)\in \text{Diff}(N)\times\text{Diff}(M)$  one has
$$\tau\circ\Phi\circ\sigma^{-1}\in \mathbf{Har}((N,(\sigma^{-1})^*g),(M,(\tau^{-1})^*\overline{g})).$$ In particular, 
	the action $\rho: \text{Diff}(N)\times \text{Diff}(M)\times C^{\infty}(N,M)\longrightarrow C^{\infty}(N,M)$  restricts to an action \begin{align*}
 	\rho\colon O(&N,g)\times O(M,\overline g)\times \mathbf{Har}((N,g),(M,\overline{g}))\to \mathbf{Har}((N,g),(M,\overline{g})).
 \end{align*}
\end{Pro}

One gets the following key results, see \cite{eells1964harmonic}.
\begin{thm}\label{thm:primera-variacion}
    Let $\Phi: (N,g)\longrightarrow (M,\overline{g})$ be a smooth map. For any compactly supported smooth variation $\{\Phi_{s}\}_{s\in I}$ of $\Phi$, the first variation of the energy functional $E_{g,\bar g}$ is given by
\begin{align*}
    \frac{d}{ds}E_{g,\bar g}(\Phi_{s}){\Big|_{s=0}} &= -\int_{N}\overline{g}(\tau(\Phi),\overline{\bm{\eta}})\Vol_{g},
\end{align*}
    where $\overline{\bm{\eta}}=\frac{d\Phi_{s}}{d s\ }{\Big|_{s=0}}\in \Gamma_c(N,\Phi^{*}TM)$ is the compactly supported variation vector field of $\Phi$. 
\end{thm}

\begin{thm}
 Let  $\Phi: N\longrightarrow M$ be a smooth map between two Riemannian manifolds $(N,g)$ and $(M,\overline{g})$ with N compact. Then $\Phi$ is harmonic if and only if $\tau(\Phi)=0$.
\end{thm}

Let $U$ be an open chart of $M$ with coordinate functions $\{x^{1},\ldots, x^{m}\}$ and let $V$ be an open subset of $N$, such that $V\subset \Phi^{-1}(U)$, coordinated by the functions $\{y^{1},\ldots, y^{n}\}$, we can locally write $\Phi^{\lambda}=x^{\lambda}\circ\Phi\in C^{\infty}(V)$. The tension field of $\Phi$ can be expressed in the above chosen local coordinates by $\tau(\Phi)
=\sum_{\lambda=1}^{m}\tau(\Phi)^{\lambda}\, \Phi^{\#}\!\!\left(\frac{\partial \ \ }{\partial x^{\lambda}}\right)$ with
\begin{align*}
\tau(\Phi)^{\lambda}=\sum_{\lambda=1}^{m}\sum_{s,t=1}^{n}g^{st}\Big[\frac{\partial^{2}\Phi^{\lambda}}{\partial y^{s}\partial y^{t}}-\sum_{l=1}^{n}\Gamma_{st}^{l}\frac{\partial\Phi^{\lambda}}{\partial y^{l}}+\sum_{\mu,\nu=1}^{m}\frac{\partial\Phi^{\mu}}{\partial y^{t}}\frac{\partial\Phi^{\nu}}{\partial y^{s}}\,\Phi^*(\overline{\Gamma}_{\mu\nu}^{\lambda})\Big],
\end{align*}
where $\{\Gamma_{st}^{l}\}$ and $\{\overline{\Gamma}_{\mu\nu}^{\lambda}\}$ are the Christoffel symbols of $g$ and $\overline{g}$, respectively, and $\Phi^{\#}\!\!\left(\frac{\partial \ \ }{\partial x^{\lambda}}\right)$ are local sections of the bundle $\Phi^{*}TM.$ The map $\Phi$ is harmonic if  and only if for every coordinate system
\begin{align*}
\tau(\Phi)^{\lambda}&=\sum_{s,t=1}^{n}g^{st}\Bigg[\frac{\partial^{2}\Phi^{\lambda}}{\partial y^{s}\partial y^{t}}-\sum_{l=1}^{n}\Gamma_{st}^{l}\frac{\partial\Phi^{\lambda}}{\partial y^{l}}+\sum_{\mu,\nu=1}^{m}\frac{\partial\Phi^{\mu}}{\partial y^{t}}\frac{\partial\Phi^{\nu}}{\partial y^{s}}\, \Phi^*(\overline{\Gamma}_{\mu\nu}^{\lambda})\Bigg] =0,\quad \lambda=1,\ldots,m.
\end{align*}

As a first instance of harmonic maps, see \cite[Example 3.3.10]{baird2003harmonic}, we have the classical:
\begin{Ex*}[Geodesic curves]
   A parametrized curve $\gamma:(-\epsilon,\epsilon)\subset\mathbb{R}\longrightarrow M$ into a Riemannian manifold $(M,\bar g)$ is a harmonic map if and only if it is a geodesic parametrized by an affine function of arc length.   That is, $\gamma$ is harmonic if and only if it verifies the Euler-Lagrange equations:
    \begin{align}
        \tau(\gamma) &=\nabla^{M}_{\dot{\gamma}}\dot{\gamma}=0,
    \end{align} where $\nabla^{M}$ is the Levi-Civita connection of $(M,\bar g)$.
In local coordinates such that $\gamma(t)=(\gamma^{1}(t),\ldots, \gamma^{m}(t))$ with $t\in(-\epsilon,\epsilon)$ this is equivalent, assuming Einstein index convention, to 
$$\frac{d^{2}\gamma^{\lambda}}{dt^{2}}+\frac{d\gamma^{\mu}}{dt}\frac{d\gamma^{\nu}}{dt}\,\gamma^*\left(\overline{\Gamma}^{\lambda}_{\mu\nu}\right) =0,\quad\quad 1\leq \lambda\leq m.$$
\end{Ex*}

\section{
Kaluza-Klein metrics on principal bundles}
In this section, we freely use many results from the theory of principal bundles and connections on them. As a general reference for these topics and other results in this area that we will use later in the article, mainly in section \ref{sec:Harmonic-Kaluza-Klein}, the reader can consult the classic text \cite{kobayashi1963foundations}  or the more recent book \cite{Joyce}.  We have also used the results of \cite{Coquereaux, Betounes}.

Let $G$ be a Lie group. Given a principal $G$-bundle $\pi: P\longrightarrow M$ whose right $G$-action is denoted $R\colon P\times G\to P$, its group of automorphisms is the group formed by its $G$-equivariant diffeomorphisms
    \begin{align*}
        \Aut(P):=\text{Diff}_{G}(P)=\{\varphi\in \text{Diff (P)}: R_{g}\circ\varphi=\varphi\circ R_{g}, \forall g\in G\}.
    \end{align*}
    There is a natural group morphism
    \begin{align*}
        \varpi &: \Aut(P)\longrightarrow \text{Diff (M)}.
    \end{align*}
    Given $\varphi\in \Aut(P)$, we write $\overline{\varphi}=\varpi(\varphi)\in \text{Diff (M)}$. The kernel of  $\varpi$ is  called the group of vertical automorphisms or the gauge group of the principal $G$-bundle and is denoted by  $\text{Gau(P)} = \text{Aut}^{\mathcal{V}}(P)$.
    Thus, one has an exact sequence of groups  $$1\longrightarrow \text{Gau}(P)\longrightarrow \text{Aut}(P)\longrightarrow \text{Diff (M)}.$$ 
    The group $G$ acts on itself on the left by internal automorphisms $c: G\times G\longrightarrow G$; that is, given $g,g'\in G$, one has $c_{g}(g')=g\cdot g'\cdot g^{-1}$. One also has the adjoint representation $\Ad: G\longrightarrow \text{Aut}(\mathfrak{g})$ of $G$ on its Lie algebra $\mathfrak g$,  that for $g\in G$ is given by $\Ad_{g}:=(c_{g})_{*e}$. The adjoint bundles associated with $P$ are defined by 
    \begin{align*}
        \pi_{Ad}: \mathrm{ad}P:=P\times_{c}G\longrightarrow M, \quad \quad \pi_{ad}:\mathrm{ad}P=P\times_{Ad}\mathfrak{g}\longrightarrow M,
    \end{align*}and they are, respectively, a bundle of Lie groups and a bundle of Lie algebras.

The following results are well known.
\begin{Pro}
    For every principal $G$-bundle $\pi: P\longrightarrow M$, the adjoint bundle $\pi_{Ad}: \mathrm{ad}P\longrightarrow M$ is a bundle of Lie groups with typical fiber $G$, whereas the adjoint bundle $\pi_{ad}:\mathrm{ad}P\longrightarrow M$ is a bundle of Lie algebras whose typical fiber is $\mathfrak{g}$. In addition, there is a natural identification
    \begin{align*}
        \gamma: \Gamma^{\infty}(M,\mathrm{ad}P)=C^{\infty}(P,G)^{G}\longrightarrow Gau(P),
    \end{align*}
    such that if $f\in C^{\infty}(P,G)^{G}$ is a G-equivariant map, (i.e., $f(p\cdot g)=c_{g^{-1}}(f(p)),\forall p\in P, g\in G)$, then the associated vertical automorphism is given by
    \begin{align*}
        \gamma_{f}(p):=p\cdot f(p),\quad p\in P.
    \end{align*}
\end{Pro}
\begin{De}
    Given a principal $G$-bundle $\pi: P\longrightarrow M$, the $G$-equivariant retracts of its exact tangent sequence
    \begin{align*}
        0\longrightarrow \mathcal{V}P \overset{i}{\longrightarrow} TP \overset{\pi_*}{\longrightarrow}  \pi^{*}TM\longrightarrow 0,
    \end{align*}
    are called principal $G$-connections. We denote by $\mathcal{A}(P)$ the space of the principal $G$-connections of $P$. Thus, $\omega\in \mathcal A(P)$ if and only if $\omega\in\Hom_P(TP,\mathcal{V}P)$ is a $G$-equivariant vector bundle morphism such that $\omega\circ i=\Id_{\mathcal{V}P}$.
\end{De}
\begin{Pro}\label{pro:principal-connections}
    Let $\pi: P\longrightarrow M$ be a principal $G$-bundle. The automorphism group of $P$ acts naturally on the left in the space of principal $G$-connections
    \begin{align*}
        (-)_{\mathcal{A}}: \Aut(P)\times \mathcal{A}(P)\longrightarrow\mathcal{A}(P) 
    \end{align*}
    so that given $\varphi\in \Aut(P)$, $\omega\in\mathcal{A}(P)$, the connection $\varphi_{\mathcal{A}}(\omega)=\varphi\cdot\omega$ is defined for each $p\in P$ by
    \begin{align*}
(\varphi\cdot\omega)_{p}:=\varphi_{*p}\circ\omega_{\varphi^{-1}(p)}\circ\varphi^{-1}_{*p}.
    \end{align*}
    The infinitesimal right $G$-action on TP establishes an isomorphism of vector bundles 
    \begin{align*}
        R_\bullet\colon \mathfrak{g}_{P}:=P\times \mathfrak{g}\overset{\sim}\longrightarrow \mathcal{V}P
    \end{align*}
    such that for any  $(p,\xi)\in P\times\mathfrak{g}$ one has that $R_\bullet(p,\xi)=\xi^*_p$ is the value at $p\in P$ of the fundamental vector field $\xi^*\in\mathfrak{X}(P)$ associated to $\xi\in \mathfrak{g}$. By means of this isomorphism each $G$-connection $\omega\in\mathcal{A}(P)$ is identified with a $\mathfrak{g}$-valued $1$-form $\widehat{\omega}\in\Omega^{1}(P,\mathfrak{g})$ that is called the connection $1$-form associated with $\omega$. The $G$-equivariance of $\omega$ translates into the $Ad$-invariance of $\widehat{\omega}$, while the retract condition is equivalent to the fact that for all $\xi\in\mathfrak{g}$ one has $i_{\xi^{*}}\widehat{\omega}=\xi$. We denote by $\widehat{\mathcal{A}}(P)$ the space formed by all connection $1$-forms  of the principal $G$-bundle $P$. The above description establishes a bijection $\mathcal{A}(P)\overset{\sim}{\longrightarrow} \widehat{\mathcal{A}}(P)$.
    Hence, given any $G$-connection $\omega\in\mathcal A(P)$ there is a commutative diagram of exact sequences of vector bundles 
\[\begin{tikzcd}
	&&&&&&  \\
	0 && \mathcal{V}P && TP && {\pi^{*}TM} && 0 \\
	\\
	0 && {\mathfrak{g}_{P}} && TP && {\pi^{*}TM} && 0
	\arrow[from=2-1, to=2-3]
	\arrow["i", from=2-3, to=2-5]
	\arrow["\omega"', shift left, bend right=30, from=2-5, to=2-3]
	\arrow["{\pi_{*}}", from=2-5, to=2-7]
	\arrow[from=2-7, to=2-9]
	\arrow[from=4-1, to=4-3]
	\arrow["{R_\bullet}","\wr"', from=4-3, to=2-3]
	\arrow["i", from=4-3, to=4-5]
	\arrow["{Id_{TP}}"', from=4-5, to=2-5]
	\arrow["{\widehat{\omega}}"', shift left, bend right=30, from=4-5, to=4-3]
	\arrow["{\pi_{*}}", from=4-5, to=4-7]
	\arrow["{Id_{\pi^{*}TM}}"', from=4-7, to=2-7]
	\arrow[from=4-7, to=4-9]
\end{tikzcd}\]
In this way, given $\omega\in \mathcal{A}(P)$, one has an isomorphism of vector bundles  $\varphi_\omega: TP\overset{\sim}{\longrightarrow} \pi^{*}TM\oplus\mathcal{V}P$ such that $\varphi_{\omega}(D_{p})=(\pi_{*}D_{p},\omega_{p}(D_{p}))$. The inverse isomorphism   $\varphi^{-1}_{\omega}:\pi^{*}TM\oplus\mathcal{V}P\overset{\sim}{\longrightarrow} TP$
    is given by $\varphi^{-1}_{\omega}(\pi_{*}D_{p},V_{p})=D_{p}-\omega_{p}(D_{p})+V_{p}$. One says the $\mathcal HP:=\mathrm{Ker}\omega$ is the $\omega$-horizontal subbundle of $TP$.\\
    Equivalently, one has the isomorphism
    $\varphi_{\widehat{\omega}}: TP\overset{\sim}{\longrightarrow} \pi^{*}TM\oplus\mathfrak{g}_{P}$, 
     defined by $\varphi_{\widehat{\omega}}(D_{p})=(\pi_{*}D_{p},\widehat{\omega}_{p}(D_{p}))$. The inverse isomorphism  $\varphi^{-1}_{\widehat{\omega}}: \pi^{*}TM\oplus \mathfrak{g}_{P} \overset{\sim}{\longrightarrow} TP$
     is given by $$\varphi^{-1}_{\widehat{\omega}}(\pi_{*}D_{p},(p,\xi))=D_{p}-[\widehat{\omega}_{p}(D_{p})]^{*}_{p}+R_{\bullet}(p,\xi)=D_{p}-[\widehat{\omega}_{p}(D_{p})]^{*}_{p}+\xi^{*}_{p}.$$
\end{Pro}
\begin{De}Let $\pi: P\longrightarrow M$ be a principal G-bundle. A Kaluza-Klein metric on $P$ is a Riemannian metric $\widehat{g}$ which is $G$-invariant. We denote by $\mathrm{Met}_{KK}(P)$ the space of Kaluza-Klein metrics on the principal $G$-bundle $\pi: P\longrightarrow M$.  
\end{De}
\begin{Pro}
    Let $\pi: P \longrightarrow M$ be a principal $G$-bundle. The automorphism group of $P$ acts naturally on its space of Kaluza-Klein metrics, so that one has a left action
    \begin{align*}
        \mu:\Aut(P)\times \mathrm{Met}_{KK}(P)\longrightarrow \mathrm{Met}_{KK}(P),
    \end{align*}
    so that given $\varphi\in \Aut(P)$, $\widehat{g}\in \mathrm{Met}_{KK}(P)$, then $\mu_{\varphi}(\widehat{g}):=\varphi\cdot\widehat{g}=(\varphi^{-1})^{*}\widehat{g}$. 
\end{Pro}
Let $\pi: P\longrightarrow M$ be a principal $G$-bundle and let $\rho: G\longrightarrow \text{Aut(V)}$ be a linear representation. The Lie group $G$ acts on $P\times V$ as  $(p,v)\cdot g=(p\cdot  g,\rho_{g^{-1}}(v))$. The quotient manifold of $P\times V$ by this action is the associated vector bundle $P\times_{\rho}V$ over $M$. If $q: P\times V\longrightarrow P\times_{\rho}V$ is the natural projection,
one has the following commutative diagram, 

\[\begin{tikzcd}
	{P\times V} && {P\times_{\rho}V} \\
	P && M
	\arrow["q", from=1-1, to=1-3]
	\arrow["{\pi_{1}}"', from=1-1, to=2-1]
	\arrow["{\pi_{\rho}}", from=1-3, to=2-3]
	\arrow["\pi", from=2-1, to=2-3]
\end{tikzcd}\]
\begin{Pro}
  The automorphism group $\Aut(P)$ of the principal bundle $\pi: P\longrightarrow M$ is represented in the automorphism group of each of its associated bundles $P\times_{\rho}V$ in the following way:   
  \begin{align*}
      (-)_{\rho}: \Aut(P)&\longrightarrow \Aut_M(P\times_{\rho}V)\\
      \varphi&\longrightarrow\varphi_{\rho}
  \end{align*}
 where $\varphi_{\rho}(q(p,v)):=q(\varphi(p),v)$ for each $\varphi\in \Aut(P)$. 
\end{Pro}
If we denote by $\varpi_{\rho}: \Aut_M(P\times_{\rho}V)\longrightarrow \text{Diff} (M)$ the natural group morphism, we have $\varpi_{\rho}(\varphi_{\rho})=\varpi(\varphi)$. A straightforwardly check proves the following result.
\begin{thm}\label{teo:structure-Kaluza-Klein-metrics}
    Let $\pi: P\longrightarrow M$ be a principal $G$-bundle. The space of Kaluza-Klein metrics  $\mathrm{Met}_{KK}(P)$ on $P$ is equipped with an $\Aut(P)$-equivariant bijection.
\begin{eqnarray*}
\Psi:\mathrm{Met}_{KK}(P) &\overset{\sim}{\longrightarrow} & \mathcal{A}(P)\times \mathrm{Met}(\mathrm{ad}P)\times \mathrm{Met}(M)  \\
\widehat{g} &\longmapsto & (\omega,\quad \quad  \quad {g}_\mathrm{{ad}},\quad \quad \quad \overline{g})
\end{eqnarray*}
 in which $\Aut(P)$ acts naturally on each of the factors (i.e., $\varphi\in \Aut(P)$ acts on $\mathcal{A}(P)$ through $\varphi_{\mathcal{A}}$, on $\mathrm{Met}(\mathrm{ad}P)$ through $(\varphi^{-1}_{ad})^{*}$, and on $\mathrm{Met}(M)$ through  $[\varpi(\varphi^{-1})]^{*}$) and the map $\Psi$ is given by
\begin{enumerate}
    \item $\omega:= \pi_{\mathcal{V}P}^{\perp_{\widehat{g}}}$ orthogonal projector onto the vertical tangent bundle defined by the metric $\widehat{g}$.
    \item $({g}_{\mathrm{ad}})_{\pi(p)}(q(p,\xi),q(p,\xi')):=\widehat{g}_{p}(\xi^{*}_{p},\xi'^{*}_{p}), \quad \forall p\in P, \forall \xi, \xi'\in \mathfrak{g}$,
    \item $(\overline{g})_{\pi(p)}(\pi_{*p}H^{1}_{p},\pi_{*p}H^{2}_{p}):=\widehat{g}_{p}(H^{1}_{p},H^{2}_{p}), \quad \forall p\in P, \quad \forall H^{1}_{p}, H^{2}_{p}\in \mathcal{H}_pP=\ker\omega_{p}=(\mathcal V_pP)^{\perp_{\widehat g_p}}$.
\end{enumerate}
With these data, we have a $\widehat{g}$-orthogonal direct sum decomposition $TP=\mathcal{V}P\perp_{\widehat{g}}\mathcal{H}P=\mathcal{V}P\perp_{\widetilde{g}}(\mathcal{V}P)^{\perp_{\widehat{g}}},$
and the projection $\pi: (P,\widehat{g})\longrightarrow (M,\overline{g})$ is a Riemannian submersion. The inverse bijection is given by 
\begin{align*}
    \widehat{g} &:=\Psi^{-1}(\omega,g_{\mathrm{ad}},\overline{g})=(\omega^{*}g_{\mathrm{ad}})^{\mathcal{V}}+\pi^{*}\overline{g},
\end{align*}  where  $(\omega^{*}g_{\mathrm{ad}})^\mathcal{V}_{p}(D_{p},D'_{p}):=(g_\mathrm{ad})_{\pi(p)}(q(p,\widehat{\omega}_{p}(D_{p})),q(p,\widehat{\omega}_{p}(D'_{p})))$.  
\end{thm}

\begin{Cor}\label{cor:existence-Kaluza-Klein-metrics}
	For any Lie group $G$, every principal $G$-bundle $\pi\colon P\longrightarrow M$ over a paracompact manifold $M$ admits  Kaluza-Klein metrics.
\end{Cor}

\begin{proof}
If $M$ is paracompact, then any principal $G$-bundle (resp. vector bundle)  over $M$ admits  a principal connection (resp. a Riemannian metric). Hence $\mathcal A(P)\times\mathrm{Met}(\ad P)\times \mathrm{Met}(M)\neq\emptyset$.
\end{proof}

\begin{De}Let $\pi: P\longrightarrow M$ be a principal $G$-bundle. Given a Kaluza-Klein metric $\widehat{g}$ on $P$ such that
\begin{align*}
    \widehat{g}=\Psi^{-1}(\omega,{g}_{\mathrm{ad}},\overline{g})=(\omega^{*}{g}_{\mathrm{ad}})^{\mathcal{V}}+\pi^{*}\overline{g}
\end{align*}
    then we say that $\overline{g}\in \mathrm{Met}(M)$ is the Riemannian metric induced on the base manifold, $\omega\in \mathcal{A}(P)$ is the principal G-connection induced on P and ${g}_{\mathrm{ad}}\in \mathrm{Met}(\mathrm{ad}P)$ is the Riemannian metric induced on the adjoint bundle $\mathrm{ad}P$. We call $ \widehat{g}=(q^{*}{g}_{\mathrm{ad}})^{\mathcal{V}}+\pi^{*}\overline{g}$ the canonical decomposition of the Kaluza-Klein metric $\widehat{g}$. A Kaluza-Klein principal $G$-bundle is a pair formed by a principal $G$-bundle $\pi\colon P\to M$ and a Kaluza-Klein metric $\widehat g\in\mathrm{Met}_{KK}(P)$. If $\widehat g=\Psi^{-1}(\omega,\beta,\overline g)$ we denote by $\pi\colon (P,\widehat g)\to (M,\overline g)$ the Kaluza-Klein bundle that supports it.
\end{De}
We can describe more explicitly the metrics on the adjoint bundle. 
\begin{Lem}
   For any principal $G$-bundle $\pi: P\longrightarrow M$ the space of Riemannian metrics on its adjoint bundle is given by the space  $$\mathrm{Met}(\mathrm{ad}P)=C^\infty(P,\mathrm{Met}(\mathfrak g))^G=\{\beta\in C^\infty(P,\mathrm{Met}(\mathfrak g))\colon \beta_{pg}=\Ad_g^*\beta_p,\quad p\in P, g\in G\}$$ of $G$-equivariant maps from $P$ into the space of euclidean metrics of the Lie algebra $\mathfrak g$. In particular, constant maps define metrics if an only if they take values in $\mathrm{Met}(\mathfrak g)^G$, the space of $\Ad$-invariant metrics on $\mathfrak{g}$.
\end{Lem}
\begin{proof}
This is a consequence of the compatibility of smooth functors acting on the category of vector bundles and the associated vector bundle functor. In particular, the space of metrics $\mathrm{Met}(\mathrm{ad}P)$ on the adjoint bundle is given by the positive definite subspace of the space of sections of the symmetric second power of its dual bundle  $S^2((\mathrm{ad}P)^*)$. By the above mentioned compatibility we get $$S^2((\mathrm{ad}P)^*)\simeq S^2((P\times_{\Ad}\mathfrak g)^*)\simeq S^2(P\times_{\Ad^*}\mathfrak g^*)\simeq P\times_{S^2(\Ad^*)} S^2(\mathfrak g^*).$$ Therefore, one has $$\Gamma(M,S^2((\mathrm{ad}P)^*))\simeq\Gamma(M, P\times_{S^2(\Ad^*)} S^2(\mathfrak g^*))=C^\infty(P,S^2(\mathfrak g^*))^G$$ and the claimed result follows.
\end{proof} 

\begin{Rem}
    Given $\beta\in C^\infty(P,\mathrm{Met}(\mathfrak g))^G$ the corresponding  metric $g_\mathrm{ad}$ on the adjoint bundle $\mathrm{ad}P$ is given by $$(g_\mathrm{ad})_{\pi(p)}(q(p,\xi_1),q(p,\xi_2)):=\beta_p(\xi_1,\xi_2),\quad p\in P,\xi_1,\xi_2\in\mathfrak{g}.$$ Reading this in the other direction determines $\beta$ when the metric $g_\mathrm{ad}$ is known. Therefore, there is an $\Aut(P)$-equivariant bijection that we continue to denote with the same letter
\begin{eqnarray*}
\Psi:\mathrm{Met}_{KK}(P) &\overset{\sim}{\longrightarrow} & \mathcal{A}(P)\times C^\infty(P,\mathrm{Met}(\mathfrak g))^G\times \mathrm{Met}(M)  \\ 
\widehat{g}\quad &\longmapsto & \quad (\omega,\quad \quad\quad\quad  \quad\beta,\quad \quad \quad\quad\quad \overline{g})
\end{eqnarray*} where now given $\xi,\xi'\in\mathfrak{g}$ one has $\beta_p(\xi,\xi'):=\widehat g_p(\xi^*_p,\xi'^*_p)$. The inverse isomorphism now is given by \begin{align*}
    \widehat{g} &:=\Psi^{-1}(\omega,\beta,\overline{g})=(\omega^{*}\beta)^{\mathcal{V}}+\pi^{*}\overline{g},
\end{align*}  where  $(\omega^{*}\beta)^\mathcal{V}_{p}(D_{p},D'_{p}):=\beta_p(\widehat{\omega}_{p}(D_{p}),\widehat{\omega}_{p}(D'_{p}))$.  

\end{Rem}

A $G$-connection $\omega$ on $P$ gives rise to a $G$-vector subbundle of the tangent bundle $TP$, the horizontal subbundle $\mathcal HP=\mathrm{Ker}\,\omega$, and to a $G$-invariant Whitney sum decomposition 
$$\begin{array}{rcl}
TP  & = & \mathcal{V}P\oplus\mathcal{H}P\\
   D_{p}  & \longmapsto & \left(\omega_{p}(D_{p}),D_{p}-\omega_{p}(D_{p})\right).
\end{array}$$ One says that $(D_p)^\mathcal{V}:=\omega_p(D_p)$ is the $\omega$-vertical part of $D_p$ , whereas $(D_p)^\mathcal{H}:=D_p-\omega_p(D_p)$ is called its $\omega$-horizontal part and $\mathcal H:=\Id_{TP}-\omega\in\Gamma(P,\mathrm{End}_P(TP))$ is called the horizontal endomorphism of $\omega$.

The curvature $2$-form $\widehat\Omega^\omega$ of the connection ${\omega}$ is the $\mathfrak{g}$-valued $2$-form on $P$ defined by
\begin{align}
    \widehat\Omega^\omega&=d\widehat{\omega}+[\widehat{\omega},\widehat{\omega}]\in\Omega^2(P,\mathfrak{g})
\end{align}
where $[\cdot.\cdot]$ is the Lie bracket of $\mathfrak{g}$.
Let $s_U:U\subset M\longrightarrow P_U:=\pi^{-1}(U)\subset P$ be a local section, then
$A_U:=s_U^{*}\widehat{\omega}\in\Omega^{1}(U,\mathfrak{g})$, $F_U:=s_U^{*}\widehat \Omega^\omega=dA_U+[A_U,A_U]$ are called, respectively, the gauge potential and the field strength defined on $U$ by the section $s_U$.  If $s_V$ is another local section defined on an open subset $V\subset M$ such that $U\cap V\neq\emptyset$, then on this intersection one has $s_V=s_U\cdot g_{UV}$, where $g_{UV}\colon U\cap V\to G$ is the transition function from $P_U$ to $P_V$, and it holds $F_V=\Ad_{g^{-1}_{UV}}F_U$.

\subsection{Local expressions}\label{subsection:local-expressions} The local section $s_U\in\Gamma(U,P)$ yields a $G$-equivariant diffeomorphism  $\psi_U\colon U\times G\overset{\sim}{\longrightarrow} P_U$ that makes commutative the diagram 
$$\begin{tikzcd}
	{U\times G} && {P_U} \\
	& U & 
	\arrow["\psi_U", "\sim"', from=1-1, to=1-3]
	\arrow["{\pi_{1}}"', from=1-1, to=2-2]
	\arrow["\pi", from=1-3, to=2-2]
\end{tikzcd}$$ and is defined on $(x,g)\in U\times G$ by $\psi_U(x,g)=s_U(x)\cdot g.$ Thus, given a Kaluza-Klein metric $\widehat g=\Psi^{-1}(\omega,\beta,\overline g)\in\mathrm{Met}_{KK}(P)$ we can define a metric $\widehat g_{U}:=\psi_U^*\widehat g\in\mathrm{Met}(U\times G)$.

Choosing a coordinate system $\{x^1,\ldots, x^n\}$ on $U$ and a basis $B=\{\xi_1,\ldots,\xi_d\}$ of the Lie algebra $\mathfrak{g}$, if we write $s=s_U$, $A=A_U$, $F=F_U$, then there exist $A_\mu^a,F_{\mu\nu}^a\in C^\infty(U)$ such that 
\begin{align*}
A &=dx^{\mu}\otimes A_{\mu}=A^{a}_{\mu}dx^{\mu}\otimes\xi_{a},\quad\quad \quad F=dx^{\mu}\wedge dx^{\nu}\otimes F_{\mu\nu}=F^{a}_{\mu\nu} dx^{\mu}\wedge dx^{\nu}\otimes\xi_{a},
\end{align*}
\begin{align*}
A_{\mu} &=A(\partial_{\mu})=A^{a}_{\mu}\xi_{a},\quad\quad \quad \quad F_{\mu\nu} =F(\partial_{\mu},\partial_{\nu})=F^{a}_{\mu\nu} \xi_{a}=\partial_{\mu}(A_{\nu})-\partial_{\nu}(A_{\mu})+[A_{\mu},A_{\nu}].
   \end{align*}

\begin{Pro}\label{pro:expr-loc-metrica-Kaluza-Klein} Let us consider on the principal $G$-bundle $P$ a Kaluza-Klein metric $\widehat{g}=\Psi^{-1}(\omega,\beta,\bar g)$. Given a local section $s\in\Gamma(U,P)$, the metric $\widehat g_{U}=\psi_U^*\widehat g\in\mathrm{Met}(U\times G)$ is a Kaluza-Klein metric on the trivial principal $G$- bundle $\pi_1\colon U\times G\to U$ such that $\widehat g_U=\Psi^{-1}(\omega_U,\beta_U,\overline g_U)$ where:
\begin{enumerate}
\item the principal connection	$\omega_U$ satisfies $\omega_U=\psi^{-1}_{U,*}\circ\omega\circ\psi_{U,*}$ and is given by $$\omega_U((D,E))=(0,[A(D)]^R+E),\quad D\in\X(U),E\in\X(G),$$ where $[-]^R\colon\mathfrak g\to \X(G)^R$ is the Lie algebra antihomomorphism that sends $\xi\in\mathfrak g$ to the right invariant vector field $\xi^R$ that it generates.
\item the $G$ equivariant map $\beta_U\in \cin(U\times G,\mathrm{Met}(\mathfrak g))$ is $$\beta_U=\beta\circ \psi_U.$$
\item the Riemannian metric $\overline g_U$ is just the restriction of $\overline g$ to $U$.
\end{enumerate}
Moreover, one has $$\widehat g_{U}=\pi_1^*\bar g_U+\bar\beta(\pi_1^*A+\pi_2^*\theta^R,\pi_1^*A+\pi_2^*\theta^R),$$
where $\bar\beta=\beta\circ s\in\cin(U,\mathrm{Met}(\mathfrak{g}))$, $A\in\Omega^1(U,\mathfrak{g})$ is the gauge potential of $\omega$ defined by the local section $s$, $\theta^R$ is the right invariant Maurer-Cartan $1$-form of the Lie group $G$ and $\pi_2$ is the natural projection of $U\times G$ onto its second factor. Therefore, we also write the local Kaluza-Klein metric as $\widehat g_U=\Psi^{-1}(A,\bar\beta,\overline g_U)$.
\end{Pro}

\begin{proof} Since $\psi_U$ is $G$-equivariant it follows immediately that $\widehat g_{U}=\psi_U^*\widehat g$ is a $G$-invariant metric on the trivial principal bundle. Given $D^1_x,D^2_X\in T_xX$, $E^1_g,E^2_g\in T_gG$, one has
     \begin{align*}
        (\widehat g_U)_{(x,g)}((D^1_x,E^1_g),(D^2_x,E^2_g))&=(\psi_U^*\widehat g)_{(x,g)}((D^1_x,E^1_g),(D^2_x,E^2_g))=\\&=\widehat g_{s(x)g}(\psi_{U,*_{(x,g)}}(D^1_x,E^1_g),\psi_{U,*_{(x,g)}}(D^2_x,E^2_g)).
    \end{align*}
    
    Now by mean of Leibniz's rule, taking into account that $\psi_U= R\circ (s\times \Id_G)$, where $R$ is the $G$-action on the principal bundle $G$, given $D_x\in T_xM$, $E_g\in T_gG$ after computation  we get \begin{align*}
        \psi_{U,*_{(x,g)}}(D_x,E_g)=R_{g*}\left(s_*D_x+[R_{g^{-1}*}E_g]^*_{s(x)}\right).\end{align*}
    Therefore, bearing in mind that $\widehat g$ is $G$-invariant and that the vertical and horizontal bundles are $\widehat g$-orthogonal,  a routinary computation gives 
        \begin{align*}
        &(\widehat g_U)_{(x,g)}((D^1_x,E^1_g),(D^2_x,E^2_g))=\\&=(\widehat g)_{s(x)}\left(\left[s_*D^1_x\right]^\mathcal{H},\left[s_*D^2_x\right]^\mathcal{H}\right) +(\widehat g)_{s(x)}\left(\left[s_*D^1_x\right]^\mathcal{V}+[R_{g^{-1}*}E^1_g]^*_{s(x)}, \left[s_*D^2_x\right]^\mathcal{V} +[R_{g^{-1}*}E^2_g]^*_{s(x)}  \right).
        \end{align*}
        Since $  \widehat{g}=\Psi^{-1}(\omega,\beta,\overline{g})=(\omega^{*}\beta)^{\mathcal{V}}+\pi^{*}\overline{g}$,  it holds \begin{align*}
            (\widehat g)_{s(x)}\left(\left[s_*D^1_x\right]^\mathcal{H},\left[s_*D^2_x\right]^\mathcal{H}\right)&=\overline g_x\left(\pi_*\left[s_*D^1_x\right]^\mathcal{H},\pi_*\left[s_*D^2_x\right]^\mathcal{H}\right)=\overline g_x(D^1_x,D^2_x),\\
             (\widehat g)_{s(x)}\left(\xi^*_{s(x)},\zeta^*_{s(x)}\right)&=\beta_{s(x)}(\xi,\zeta),\quad \xi,\zeta\in\mathfrak g.
        \end{align*} 
On the other hand, given $D_x\in T_xX$ one has 
        \begin{align*}
    \left[s_{*}D_x\right]^{\mathcal{V}}_{s(x)}=\omega_{s(x)}(s_{*}D_x)=[\widehat \omega_{s(x)}(s_{*}D_x)]^*_{s(x)}=[(s^*\widehat\omega)_x(D_x)]^*_{s(x)}=\left[A_x(D_x)\right]^{*}_{s(x)}.
        \end{align*} Taking into account these results we get
\begin{align*}
        &(\widehat g_U)_{(x,g)}((D^1_x,E^1_g),(D^2_x,E^2_g))=\overline g_x(D^1_x,D^2_x)+\beta_{s(x)}\left(A_x(D_x^1)+ R_{g^{-1}*}E^1_g, A_x(D_x^2)+ R_{g^{-1}*}E^2_g\right).
        \end{align*}The expression of $\widehat g_U$ follows since for any $E_g\in T_gG$ the right invariant Maurer-Cartan form $\theta^R$ of $G$ satisfies $$\theta^R_g(E_g)=R_{g^{-1}*}E_g.$$ On the other hand, defining $\omega_U=\psi^{-1}_{U,*}\circ\omega\circ\psi_{U,*}$, after a straightforward   computation we get
    \begin{align*}
    (\omega_U)_{(x,g)}((D_x,E_g))&=R_{g*}\left(\psi^{-1}_{U,*_{s(x)\cdot g}}\left(R_{s(x),*}\left(A_x(D_x)+R_{g^{-1}*}E_g\right)\right)\right).
    \end{align*} 
 Since $\psi_U$ is $G$-equivariant and it verifies $\pi\circ\psi_U=\pi_1$, it follows that $\varphi_U:=\psi_U^{-1}=(\pi,\widehat\varphi_U)$ where $\widehat\varphi_U\colon P_U\to G$ is $G$-equivariant. Whence, we have $R_g\circ \varphi_U\circ R_{s(x)}=R_{(x,g)}\circ c_{g^{-1}}$. Thus we get 
 \begin{align*}
    (\omega_U)_{(x,g)}((D_x,E_g))&=\left[\Ad_{g^{-1}}\left(A_x(D_x)+R_{g^{-1}*}E_g\right) \right]^*_{(x,g)}.
    \end{align*}Hence, the connection $1$-form $\widehat\omega_U$ satisfies \begin{align*}
    (\widehat\omega_U)_{(x,g)}((D_x,E_g))&=\Ad_{g^{-1}}\left(A_x(D_x)+R_{g^{-1}*}E_g\right).
    \end{align*}
    After some computations we have \begin{align*}
    	[(\omega_U^{*}\beta_U)^{\mathcal{V}}&]_{(x,g)}((D^1_x,E^1_g),(D^2_x,E^2_g))=\beta_{s(x)}\left(A_x(D^1_x)+R_{g^{-1}*}E^1_g,A_x(D^2_x)+R_{g^{-1}*}E^2_g\right).
 \end{align*}Taking this into account we obtain
 \begin{align*}
 [(\omega_U^{*}\beta_U)^{\mathcal{V}}+\pi^{*}\overline{g}_U]_{(x,g)}((D^1_x,&E^1_g),(D^2_x,E^2_g))=(\widehat g_U)_{(x,g)}((D^1_x,E^1_g),(D^2_x,E^2_g)).
 \end{align*}This shows that $\widehat g_U=\Psi^{-1}(\omega_U,\beta_U,\overline g_U)$ as claimed in the statement and the proof is finished.

   \end{proof}

\begin{Cor}
    If $\{x^1,\ldots, x^n\}$ is a coordinate system on $U$, $B=\{\xi_1,\ldots,\xi_d\}$ is a basis of the Lie algebra $\mathfrak{g}$ and $B^*=\{\theta^1,\ldots,\theta^d\}$ is a basis of the right invariant $1$-forms on $G$ dual to $B$, then one has 
    $$\widehat g_U= [\bar g_{\mu\nu}+\bar\beta_{ab} A^a_\mu A^b_\nu]dx^\mu\otimes dx^\nu+\bar\beta_{ab} A^a_\mu dx^\mu\otimes\theta^b+ \bar\beta_{ab} A^b_\nu \theta^a\otimes dx^\nu +\bar\beta_{ab}\theta^a\otimes\theta^b,$$ where $$\bar g_{\mu\nu}=\overline g\left(\frac{\partial\ }{\partial x^\mu},\frac{\partial\ }{\partial x^\nu}\right)\in\cin(U),\quad \bar\beta_{ab}=\bar\beta(\xi_a,\xi_b)=\beta_{s(x)}(\xi_a,\xi_b)\in\cin (U).$$
\end{Cor}

\begin{proof}
    The duality of $B$ and $B^*$ means that $\theta^a(\xi_b^R)=\delta^a_b$, where $\xi_b^R$ denotes the right invariant vector field on $G$ defined by $\xi_b\in\mathfrak g$. Therefore, the right invariant Maurer-Cartan form is $$\theta^R=\theta^a\otimes\xi_a$$and hence the result follows immediately taking into account Proposition \ref{pro:expr-loc-metrica-Kaluza-Klein}.
\end{proof}

Let us determine now the Levi-Civita connection of the local Kaluza-Klein metric $\widehat g_U$ on an adapted reference frame. Notice that we can write $$\widehat g_U=\bar g_{\mu\nu}\omega^\mu\otimes\omega^\nu+\bar\beta_{ab}\omega^a\otimes\omega^b,$$
with $
 \omega^\mu=dx^\mu,\quad\ \ \, 1\leq\mu\leq m,\quad\quad 
 \omega^a=A^a+\theta^a,\quad 1\leq a\leq d.
$ This way we get a reference frame $\mathcal R^*=\{\omega^\mu,\omega^a\}_{\mu=1,a=1}^{m,d}$ on $\Omega^1(U\times G)$. We consider the dual reference frame on $\X(U\times G)$  $$\mathcal R=\left\{E_\mu=\left(\frac{\partial\ }{\partial x^\mu}\right)^h=\frac{\partial\ }{\partial x^\mu}-A^a_\mu\, \xi^R_ a, E_a=\xi^R_a\right\}_{\mu=1,a=1}^{m,d}.$$
 A lengthy but otherwise straightforward computation using the Koszul formula, shows that 
 
 \begin{Pro}\label{pro:Levi-Civita-connection}
 The Levi-Civita connection of the local Kaluza-Klein metric $\widehat g_U$ is given on the reference frame  $\mathcal R=\{E_\mu,E_a\}_{\mu=1,a=1}^{m,d}$ by 
\begin{align*}
\nabla_{E_\mu}E_\nu&=\overline\Gamma_{\mu\nu}^\rho E_\rho-\frac{1}{2}F_{\mu\nu}^a E_a,\\
\nabla_{E_\mu}E_a&=-\frac{1}{2}\bar\beta_{ab}F^b_{\alpha\mu}\bar g^{\alpha\rho} E_\rho+\frac{\bar\beta^{db}}{2}\left\{\partial_\mu\bar\beta_{ad}+\bar\beta([A_\mu,\xi_a],\xi_d)-\bar\beta(\xi_a,[A_\mu,\xi_d]) \right\} E_b,\\
\nabla_{E_a}E_\mu&=-\frac{1}{2}\bar\beta_{ab}F^b_{\alpha\mu}\bar g^{\alpha\rho} E_\rho+\frac{\bar\beta^{db}}{2}\left\{\partial_\mu\bar\beta_{ad}-\bar\beta([A_\mu,\xi_a],\xi_d)-\bar\beta(\xi_a,[A_\mu,\xi_d]) \right\} E_b,\\
 	\nabla_{E_a}E_b&=-\frac{\bar g^{\mu\rho}}{2}\{ \partial_\mu\bar\beta_{ab}-\bar\beta([A_\mu,\xi_a],\xi_b)-\bar\beta(\xi_a,[A_\mu,\xi_b])\}\, E_\rho+\\&\ \ +\frac{\bar\beta^{dc}}{2}\{ \bar\beta (\xi_a,[\xi_b,\xi_d])-\bar\beta (\xi_b,[\xi_d,\xi_a])-\bar\beta (\xi_d,[\xi_a,\xi_b])\}\, E_c,
 \end{align*}where $\overline\Gamma^\rho_{\mu\nu}$ are the Christoffel symbols of $\overline g_U$ with respect to the coordinate system $\{x^1,\ldots, x^n\}$.
 \end{Pro}

\section{Harmonic maps into Kaluza-Klein principal bundles}\label{sec:Harmonic-Kaluza-Klein}

\setlength{\parindent}{0pt} Let us consider now  a Kaluza-Klein principal $G$-bundle  $\pi\colon (P,\widehat g)\to (M,\overline g)$ such that $\widehat g=\Psi^{-1}(\omega,\beta,\overline g)$. Given a Riemannian manifold $(N,g)$ let 
$\widetilde{\Phi}: N\longrightarrow P$, $\Phi: N\longrightarrow M$ be two smooth maps such that the following diagram of maps of Riemannian manifolds
\[\begin{tikzcd}\label{CO}
	(N,g) && (P,\widehat{g}) \\
	\\
	&& (M,\overline{g})
	\arrow["\widetilde{\Phi}", from=1-1, to=1-3]
	\arrow["\Phi"', from=1-1, to=3-3]
	\arrow["\pi", from=1-3, to=3-3]
\end{tikzcd}\]
is commutative, i.e., $\pi\circ\widetilde{\Phi}=\Phi$. The energy functional of $\widetilde{\Phi}$ is
 $E_{g,\hat g}(\widetilde{\Phi})= \frac{1}{2}\int_{N}\|d\widetilde{\Phi}\|_{g,\widehat g}^{2}\Vol_{g}$ and $\widetilde{\Phi}$ is harmonic if $\tau(\widetilde{\Phi})=0$. In a similar way the energy functional of $\Phi$
is $E_{g,\overline g}(\Phi)= \frac{1}{2}\int_{N}\|d\Phi\|_{g,\overline g}^{2}\Vol_{g}$.

Now, we want to establish the relationship between the tension fields of $\widetilde{\Phi}$ and $\Phi$. Applying well known properties of the second fundamental form of a map between Riemannian manifolds, see for instance \cite[Chapter 3]{baird2003harmonic}, one gets straightforwardly the following basic results.
\begin{Pro}\label{pro:second-fund-form-composition}
 The second fundamental form $\mathrm{II}_{\Phi}\in S^2(N,\widetilde\Phi^*TP)$  of the composition
$\Phi=\pi\circ\widetilde{\Phi}: (N,g)\longrightarrow (M,\overline{g})$ satisfies 
\begin{equation}\label{secenn}
    \mathrm{II}_{{\Phi}} =\pi_{*}\circ \mathrm{II}_{{\widetilde{\Phi}}}+\widetilde{\Phi}^{*}(\mathrm{II}_{\pi})
\end{equation}
and the tension field $\tau(\widetilde\Phi)\in\X_\Phi(M)$ of $\Phi$, which is just the $g$-Riemannian trace of $\mathrm{II}_{{\Phi}} $, verifies the identity
\begin{align*}
    \tau(\Phi) &=\pi_{*}(\tau(\widetilde{\Phi}))+\Tr_g\big[\widetilde{\Phi}^{*}(\mathrm{II}_{\pi})\big].
\end{align*}
\end{Pro}
\begin{Cor}
     
    If $\widetilde{\Phi}$ is a harmonic map, then one has
    \begin{align*}
        \tau(\Phi) &=\Tr_g\big[\widetilde{\Phi}^{*}(\mathrm{II}_{\pi})\big].
    \end{align*}
    Therefore, if $\pi$ is totally geodesic (i.e., its second fundamental form $\mathrm{II}_{\pi}$ vanishes), then $\Phi$ is also harmonic.
\end{Cor}

Now let us recall some key concepts.
   
\begin{De}
Let $\pi\colon P\to M$ be a principal $G$-bundle and let $\widehat g=\Psi^{-1}(\omega,\beta,\overline g)\in\mathrm{Met}_{KK}(P)$ be a Kaluza-Klein metric. One says that a vector field $D\in\mathfrak X(P)$ is:
\begin{enumerate}
    \item $\pi$-projectable if it is $\pi$-related to a vector field $\bar D\in\mathfrak X(M)$ that is called its $\pi$-projection.
    \item $\omega$-horizontal if its vertical part $D^{\mathcal{V}}=\omega(D)$ vanishes.
    \item $\omega$-basic if it is $\omega$-horizontal and $\pi$-projectable.
    \item the $\omega$-horizontal lift of a vector field $\bar D\in\mathfrak{X}(M)$ if it is basic and its $\pi$-projection is $\bar D$.
\end{enumerate}  We might suppress the $\omega$ prefix if the connection is known from the context and there is no risk of confusion.
\end{De}

Notice that a vector field $D\in\mathfrak{X}(P)$ is $\pi$-projectable if and only if its $\omega$-horizontal part $D^\mathcal{H}$ is $\pi$-projectable; that is, $D$ is $\pi$-projectable if and only if, its horizontal part $D^\mathcal{H}$ is $\omega$-basic . 

Since the Kaluza-Klein metric $\widehat{g}=\Psi^{-1}(\omega,\beta,\overline{g})$ on the principal $G$-bundle $\pi\colon P\to M$  is given by $$  \widehat{g}=(\omega^{*}\beta)^{\mathcal{V}}+\pi^{*}\overline{g}$$  it follows that $\pi_ {*}$ maps the horizontal tangent bundle $\mathcal{H}P=(\ker\pi_{*})^{\perp_{\hat g}}$ isometrically onto the tangent bundle $TM$ of the base manifold $M$. Therefore, every Kaluza-Klein principal $G$-bundle $\pi\colon (P,\widehat g)\to (M,\overline g)$ is a Riemannian submersion.

\begin{De}[O'Neill's Tensors \cite{o1966fundamental}]
 Let $\pi: (P,\widehat g)\to (M,\overline g)$ be a Kaluza-Klein principal $G$-bundle  and let $\nabla^{P}$ be the Levi-Civita connection of $\widehat{g}$. The fundamental $(2,1)$-tensors $\bm{A}$, $\bm{T}$ of $\pi$ are given, for arbitrary vector fields $D_{1},D_{2}$ on $P$ by 
 \begin{align}\label{tena}
     \bm{A}(D_{1},D_{2}) &=\big[\nabla^{P}_{D_{1}^{\mathcal{H}}}D_{2}^{\mathcal{V}}\big]^{\mathcal{H}}+\big[\nabla^{P}_{D_{1}^{\mathcal{H}}}D_{2}^{\mathcal{H}}\big]^{\mathcal{V}},
     \end{align}
     \begin{align}\label{tent}
     \bm{T}(D_{1},D_{2}) &=\big[\nabla^{P}_{D_{1}^{\mathcal{V}}}D_{2}^{\mathcal{H}}\big]^{\mathcal{V}}+[\nabla^{P}_{D_{1}^{\mathcal{V}}}D_{2}^{\mathcal{V}}]^{\mathcal{H}}.
 \end{align}
The tensor $\bm{T}$ represents the family of fundamental forms of the fibers of $P$.
\end{De}

Using the properties of these tensors described in \cite{o1966fundamental}, we get the following results.  

\begin{Lem}\label{lem:descomposicion-derivada-covariante} Let $\pi: (P,\widehat g)\to (M,\overline g)$ be a Kaluza-Klein principal $G$-bundle. If $\nabla^{P,\mathcal{V}}$, $\nabla^{P,\mathcal{H}}$ are the connections induced on the vertical $\mathcal VP\to P$ and horizontal subbundles $\mathcal HP\to P$, respectively, by the Levi-Civita of $P$, then for any $D_1,D_2\in\X(P)$ it holds:
\begin{align*}
 	[\nabla^P_{D_1}D_2]^\mathcal{V}&=	\nabla^{P,\mathcal{V}}_{D_1}D_2^\mathcal{V}+\bm{T}(D_1^\mathcal{V},D_2^\mathcal{H})+\bm{A}(D_1^\mathcal{H},D_2^\mathcal{H}),\\
 	[\nabla^P_{D_1}D_2]^\mathcal{H}&=\nabla^{P,\mathcal{H}}_{D_1}D_2^\mathcal{H}+\bm{T}(D_1^\mathcal{V},D_2^\mathcal{V})+\bm{A}(D_1^\mathcal{H},D_2^\mathcal{V}).
 	\end{align*}

\end{Lem}

\begin{Lem} Let $\pi: (P,\widehat g)\to (M,\overline g)$ be a Kaluza-Klein principal $G$-bundle. If  $\widetilde{\Phi}: (N,g)\longrightarrow (P,\widehat{g})$ is a smooth map and $\widetilde\Phi^*\nabla^P$ is the pullback under $\widetilde\Phi$ of the Levi-Civita connection of $(P,\widehat g)$, then for any $U_1,U_2\in\X(N)$ it holds:
\begin{align*}
 	[(\widetilde\Phi^*\nabla^P&)_{U_1}(\widetilde\Phi_*U_2)]^\mathcal{V}=\\&=	(\widetilde\Phi^*\nabla^{P,\mathcal{V}})_{U_1}[\widetilde\Phi_*(U_2)]^\mathcal{V}+(\widetilde{\Phi}^{\#}\bm{T})([\widetilde\Phi_*(U_1)]^\mathcal{V},[\widetilde\Phi_*(U_2)]^\mathcal{H})+(\widetilde{\Phi}^{\#}\bm{A})([\widetilde\Phi_*(U_1)]^\mathcal{H},[\widetilde\Phi_*(U_2)]^\mathcal{H}),\\
 	\\
 	[(\widetilde\Phi^*\nabla^P&)_{U_1}(\widetilde\Phi_*U_2)]^\mathcal{H}=\\&=(\widetilde\Phi^*\nabla^{P,\mathcal{H}})_{U_1}[\widetilde\Phi_*(U_2)]^\mathcal{H}+(\widetilde\Phi^{\#}\bm{T})([\widetilde\Phi_*(U_1)]^\mathcal{V},[\widetilde\Phi_*(U_2)]^\mathcal{V})+(\widetilde\Phi^{\#}\bm{A})([\widetilde\Phi_*(U_1)]^\mathcal{H},[\widetilde\Phi_*(U_2)]^\mathcal{V}).
 	\end{align*}
\end{Lem}

\begin{De}
	Let $\pi: (P,\widehat g)\to (M,\overline g)$ be a Kaluza-Klein principal $G$-bundle such that $\widehat g=\Psi^{-1}(\omega,\beta,\overline g)$ and let $\widetilde\Phi\colon N\to P$ be a smooth map.  The decomposition $$TP=\mathcal{V}P\oplus\mathcal{H}P$$ into the Whitney sum of the vertical  and  $\omega$-horizontal subbundles induces a decomposition of the vector valued $1$-form $d\widetilde\Phi\in\Omega^1(N,\widetilde\Phi^*TP)$ into a sum of  $1$-forms taking values in the pullback under $\widetilde\Phi$ of the vertical and $\omega$-horizontal subbundles
	$$d\widetilde\Phi=d^\mathcal{V}\widetilde\Phi+d^\mathcal{H}\widetilde\Phi .$$ We say that $d^\mathcal{V}\widetilde\Phi\in \Omega^1(N,\widetilde\Phi^*\mathcal VP)$, $d^\mathcal{H}\widetilde\Phi\in\Omega^1(N,\widetilde\Phi^*\mathcal HP)$ are, respectively, the vertical and $\omega$-horizontal differentials of $\widetilde\Phi$.
	
	Analogously, if $\nabla=\nabla^N\otimes1+1\otimes\widetilde\Phi^*\nabla^P$ is the connection induced on $TN\otimes\widetilde\Phi^*TP$ by the Levi-Civita connections $\nabla^N,\nabla^P$ of $(N,g)$ and $(P,\widehat g)$, respectively, then the second fundamental form $\mathrm{II}_{\widetilde\Phi}=\nabla(d\widetilde\Phi)\in S^2(N,\widetilde\Phi^*TP)$ of $\widetilde\Phi$ decomposes into a sum $$\mathrm{II}_{\widetilde\Phi}=\left[\mathrm{II}_{\widetilde\Phi}\right]^\mathcal{V}+\left[\mathrm{II}_{\widetilde\Phi}\right]^\mathcal{H}.$$ We say that $\left[\mathrm{II}_{\widetilde\Phi}\right]^\mathcal{V}\in S^2(N,\widetilde\Phi^*\mathcal VP)$, $\left[\mathrm{II}_{\widetilde\Phi}\right]^\mathcal{H}\in S^2(N,\widetilde\Phi^*\mathcal HP)$ are, respectively, the  vertical and $\omega$-horizontal components of the second fundamental $\mathrm{II}_{\widetilde\Phi}$ of $\widetilde\Phi$. In a similar way, we say that $$[\tau(\widetilde\Phi)]^\mathcal{V}:=\Tr_g([\mathrm{II}_{\widetilde\Phi}]^\mathcal{V})\in\X^\mathcal{V}_{\widetilde\Phi}(P),\quad\quad\quad [\tau(\widetilde\Phi)]^\mathcal{H}:=\Tr_g([\mathrm{II}_{\widetilde\Phi}]^\mathcal{H})\in\X^\mathcal{H}_{\widetilde\Phi}(P),$$ are the vertical and $\omega$-horizontal components of the tension field $\tau(\widetilde\Phi)$ of $\widetilde\Phi$.
\end{De}

Taking into account the previous lemmas, one straightforwardly proves the following result. 

\begin{Pro}
	Let $\pi: (P,\widehat g)\to (M,\overline g)$ be a Kaluza-Klein principal $G$-bundle such that $\widehat g=\Psi^{-1}(\omega,\beta,\overline g)$ and let $\widetilde\Phi\colon (N,g)\to (P,\widehat g)$ be a smooth map. The vertical and $\omega$-horizontal components of the second fundamental form $\mathrm{II}_{\widetilde\Phi}$ of $\widetilde\Phi$  are given by \begin{align*}
 	\left[\mathrm{II}_{\widetilde\Phi}\right]^\mathcal{V}&=[\nabla(d^\mathcal{V}\widetilde\Phi)]^\mathcal{V}+[\nabla(d^\mathcal{H}\widetilde\Phi)]^\mathcal{V},\\
 	\left[\mathrm{II}_{\widetilde\Phi}\right]^\mathcal{H}&=[\nabla(d^\mathcal{V}\widetilde\Phi)]^\mathcal{H}+[\nabla(d^\mathcal{H}\widetilde\Phi)]^\mathcal{H}.
 \end{align*} Moreover, one has 
 \begin{align*}
 	\left[\mathrm{II}_{\widetilde\Phi}\right]^\mathcal{V}&=\nabla^\mathcal{V}(d^\mathcal{V}\widetilde\Phi)+\widetilde\Phi^*\left(\bm{T}^{\mathcal{V},\mathcal{H}}\right)+\widetilde\Phi^{*}\left(\bm{A}^{\mathcal{H},\mathcal{H}}\right),\\
 	\left[\mathrm{II}_{\widetilde\Phi}\right]^\mathcal{H}&=\nabla^\mathcal{H}(d^\mathcal{H}\widetilde\Phi)+\widetilde\Phi^*\left(\bm{T}^{\mathcal{V},\mathcal{V}}\right)+\widetilde\Phi^{*}\left(\bm{A}^{\mathcal{H},\mathcal{V}}\right),
 \end{align*} where $\nabla^\mathcal{V}=\nabla^N\otimes1+1\otimes\widetilde\Phi^*\nabla^{P,\mathcal{V}}$, $\nabla^\mathcal{H}=\nabla^N\otimes1+1\otimes\widetilde\Phi^*\nabla^{P,\mathcal{H}}$ are, respectively, the induced connections  on $TN\otimes\widetilde\Phi^*\mathcal VP$ and $TN\otimes\widetilde\Phi^*\mathcal HP$ and given $D_1,D_2\in \X(N)$ it is
 \begin{align*}
 	\widetilde\Phi^*\left(\bm{T}^{\mathcal{V},\mathcal{H}}\right)(D_1,D_2)&=(\widetilde\Phi^{\#}\bm{T})([\widetilde\Phi_*(D_1)]^\mathcal{V},[\widetilde\Phi_*(D_2)]^\mathcal{H}),\\
 	\widetilde\Phi^*\left(\bm{A}^{\mathcal{H},\mathcal{H}}\right)(D_1,D_2)&=(\widetilde\Phi^{\#}\bm{A})([\widetilde\Phi_*(D_1)]^\mathcal{H},[\widetilde\Phi_*(D_2)]^\mathcal{H}),\\
 	\widetilde\Phi^*\left(\bm{T}^{\mathcal{V},\mathcal{V}}\right)(D_1,D_2)&=(\widetilde\Phi^{\#}\bm{T})([\widetilde\Phi_*(D_1)]^\mathcal{V},[\widetilde\Phi_*(D_2)]^\mathcal{V}),\\
 	\widetilde\Phi^*\left(\bm{A}^{\mathcal{H},\mathcal{V}}\right)(D_1,D_2)&=(\widetilde\Phi^{\#}\bm{A})([\widetilde\Phi_*(D_1)]^\mathcal{H},[\widetilde\Phi_*(D_2)]^\mathcal{V}).
 \end{align*}
\end{Pro}

Therefore, it makes sense to introduce the following:

\begin{De}
	Let $\pi: (P,\widehat g)\to (M,\overline g)$ be a Kaluza-Klein principal $G$-bundle such that $\widehat g=\Psi^{-1}(\omega,\beta,\overline g)$ and let $\widetilde\Phi\colon (N,g)\to (P,\widehat g)$ be a smooth map. One says that $$\mathrm{II}_{\widetilde\Phi}^\mathcal{V}:=\nabla^\mathcal{V}(d^\mathcal{V}\widetilde\Phi)\in S^2(N,\widetilde\Phi^*\mathcal VP),\quad\quad\quad  \mathrm{II}_{\widetilde\Phi}^\mathcal{H}:=\nabla^\mathcal{H}(d^\mathcal{H}\widetilde\Phi)\in S^2(N,\widetilde\Phi^*\mathcal HP)$$ are, respectively, the intrinsic vertical and horizontal second fundamental forms of $\widetilde\Phi$. Analogously,  $$\tau^\mathcal{V}(\widetilde\Phi):=\Tr_g(\mathrm{II}_{\widetilde\Phi}^\mathcal{V})\in\X^\mathcal{V}_{\widetilde\Phi}(P),\quad\quad\quad \tau^\mathcal{H}(\widetilde\Phi):=\Tr_g(\mathrm{II}_{\widetilde\Phi}^\mathcal{H})\in\X^\mathcal{H}_{\widetilde\Phi}(P),$$ are called, respectively, the intrinsic vertical and horizontal tension fields of $\widetilde\Phi$.
\end{De}
Taking into account that $\omega(\bm A)$ is skew-symmetric, we immediately get:
\begin{Cor}\label{cor:vertical-horizontal-decomposition}
	Let $\pi: (P,\widehat g)\to (M,\overline g)$ be a Kaluza-Klein principal $G$-bundle such that $\widehat g=\Psi^{-1}(\omega,\beta,\overline g)$ and let $\widetilde\Phi\colon (N,g)\to (P,\widehat g)$ be a smooth map. It holds \begin{align*}
 	\left[\mathrm{II}_{\widetilde\Phi}\right]^\mathcal{V}&=\mathrm{II}_{\widetilde\Phi}^\mathcal{V}+\omega\left(\widetilde\Phi^*\bm{T}+\widetilde\Phi^{*}\bm{A}\right),\quad\quad\quad \left[\tau(\widetilde\Phi)\right]^\mathcal{V}&&=\tau^\mathcal{V}(\widetilde\Phi)+\omega\left(\Tr_g\left[\widetilde\Phi^*\bm{T}\right]\right),\\
 	\left[\mathrm{II}_{\widetilde\Phi}\right]^\mathcal{H}&=\mathrm{II}_{\widetilde\Phi}^\mathcal{H}+\mathcal{H}\left(\widetilde\Phi^*\bm{T}+\widetilde\Phi^{*}\bm{A}\right), \quad\quad\quad \left[\tau(\widetilde\Phi)\right]^\mathcal{H}&&=\tau^\mathcal{H}(\widetilde\Phi)+\mathcal H\left(\Tr_g\left[\widetilde\Phi^*(\bm{T}+\bm{A})\right]\right).
 \end{align*}
\end{Cor}

As a consequence of the previous results, it makes sense to introduce the following:
 
 \begin{De}
 	Let $\pi: (P,\widehat g)\to (M,\overline g)$ be a Kaluza-Klein principal $G$-bundle such that $\widehat g=\Psi^{-1}(\omega,\beta,\overline g)$. A map $\widetilde\Phi\colon (N,g)\to (P,\widehat g)$ is called vertically  (resp. horizontally) harmonic  if the vertical (resp. horizontal) component of its tension field $\tau(\widetilde\Phi)$ vanishes $$\left[\tau(\widetilde\Phi)\right]^\mathcal{V}=0.\quad\quad (\text{resp.}\ \left[\tau(\widetilde\Phi)\right]^\mathcal{H}=0.)$$
 \end{De}

 It is clear that $\widetilde\Phi$ is a harmonic map if and only if it is horizontally and vertically harmonic, i.e
        $$[\tau(\widetilde{\Phi})]^{\mathcal{H}}=0  \quad \text{and} \quad
        [\tau(\widetilde\Phi{})]^{\mathcal{V}}=0.$$

\begin{Pro}\label{vbv}
    Let $\pi: (P,\widehat{g})\longrightarrow (M,\overline{g})$  be a Kaluza-Klein principal $G$-bundle. For every   vector field $D\in\mathfrak{X}(P)$,  one has
    \begin{align}\label{mmmm}
        \mathrm{II}_{\pi}(D,D) &=[-\pi_{*}\circ(2\bm{A}+\bm{T})](D,D),
    \end{align}
    where ${\nabla}$ is the connection on $T^{*}P\otimes\pi^{*}TM$ induced by the Levi-Civita connections of $(P,\widehat{g})$ and $(M,\overline{g})$.
\end{Pro}
\begin{proof} Vertical vector fields are $\pi$-projectable and one has $\X^\mathcal{H}(P)=\cin(P)\otimes_{\cin(M)}\X_B(P)$. Therefore, since $\mathrm{II}_{\pi}$ is a tensor, it is enough to prove that the claimed identity holds for a $\pi$-projectable $D\in \mathfrak{X}(P)$. Applying Lemmas $1$ and $3$ 
 of \cite{o1966fundamental}, we get\begin{align*}
        \mathrm{II}_{\pi}(D,D) &=({\nabla}\pi_{*})(D,D)=(\pi^*{\nabla^P})_{D}\pi_{*}D-\pi_{*}(\nabla^{P}_{D}D)=\\
        &= \nabla^{M}_{\pi_{*}D}\pi_{*}D-\pi_{*}\left(\nabla^{P}_{D^{\mathcal{H}}}D^{\mathcal{H}}+\nabla^{P}_{D^{\mathcal{H}}}D^{\mathcal{V}}+\nabla^{P}_{D^{\mathcal{V}}}D^{\mathcal{H}}+\nabla^{P}_{D^{\mathcal{V}}}D^{\mathcal{V}}\right)=\\
&=\nabla^{M}_{\pi_{*}D}\pi_{*}D-\pi_{*}\left((\nabla^{P}_{D^{\mathcal{H}}}D^{\mathcal{H}})^{\mathcal{H}}\right)-\pi_{*}\left(\bm{A}(D^{\mathcal{H}},D^{\mathcal{V}})+(\nabla^{P}_{D^{\mathcal{V}}}D^{\mathcal{H}})^{\mathcal{H}}+\bm{T}(D^{\mathcal{V}},D^{\mathcal{V}})\right)=\\
&=-\pi_{*}\left(2\bm{A}(D^{\mathcal{H}},D^{\mathcal{V}})+\bm{T}(D^{\mathcal{V}},D^{\mathcal{V}})\right)=[-\pi_{*}\circ(2\bm{A}+\bm{T})](D,D),
    \end{align*}
    hence (\ref{mmmm}).
\end{proof}

Taking the $g$-Riemannian trace of $\widetilde{\Phi}^{*}(\mathrm{II}_{\pi})$ we get the:

\begin{Cor}\label{cor:traza-segunda-forma-fundamental-fibras}  If $\pi: (P,\widehat{g})\longrightarrow (M,\overline{g})$ is a Kaluza-Klein principal $G$-bundle, then one has
\begin{align}\label{prp} 
     \Tr_g\big[\widetilde{\Phi}^{*}(\mathrm{II}_{\pi})\big] &=-\pi_{*}\left(\Tr_g\big[ \widetilde{\Phi}^{*}(2\bm{A}+\bm{T})\big]\right).
\end{align}
\end{Cor}
Combining Corollary \ref{cor:traza-segunda-forma-fundamental-fibras}, Proposition \ref{pro:second-fund-form-composition} and bearing in mind that $\pi_*\colon\mathcal HP\xrightarrow{\sim}\pi^*TM$ is a vector bundle isomorphism, we obtain the key result:

\begin{thm}\label{teo:horizontal-component-tension-field}
    Let $\pi: (P,\widehat{g})\longrightarrow (M,\overline{g})$  be a Kaluza-Klein principal $G$-bundle,  $\widetilde{\Phi}: (N,g)\longrightarrow (P,\widehat{g})$ a smooth map and $\Phi=\pi\circ\widetilde \Phi\colon (N,g)\longrightarrow (M,\overline{g})$ their composition. The tension field of $\Phi$ is given by \begin{align*}
    \tau(\Phi) &=\pi_{*}(\tau(\widetilde{\Phi})) -\pi_{*}\left(\Tr_g\big[\widetilde{\Phi}^{*}(2\bm{A}+\bm{T})\big]\right).
\end{align*}
In particular, $\widetilde{\Phi}$ is a horizontal harmonic map if and only if $\Phi$ satisfies
\begin{align}\label{CVC}
    \tau(\Phi) &= -\pi_{*}\left(\Tr_g\big[\widetilde{\Phi}^{*}(2\bm{A}+\bm{T})\big]\right).
\end{align}
\end{thm}

   \begin{thm}\label{teo:vertical-tension-field}  Let $\pi: (P,\widehat{g})\longrightarrow (M,\overline{g})$  be a Kaluza-Klein principal $G$-bundle. The vertical component of the tension field of a smooth map $\widetilde{\Phi}: (N,g)\longrightarrow (P,\widehat{g})$ is given by
   \begin{align*}
       [\tau(\widetilde{\Phi})]^{\mathcal{V}} &=-\delta^{(\nabla^{N},\widetilde\Phi^*\nabla^{P,\mathcal{V}})}(\widetilde{\Phi}^{*}\omega)+\omega(\Tr_{g}(\widetilde{\Phi}^{*}\bm{T})),
   \end{align*}
       where 
 $\widetilde{\Phi}^{*}\omega\in\Omega^{1}(N,\widetilde{\Phi}^{*}(\mathcal{V}P))$ is the pullback of the connection $\omega\in\Omega^{1}(P,\mathcal{V}P)$  and $\delta^{(\nabla^{N},\widetilde\Phi^*\nabla^{P,\mathcal{V}})}$ is the codifferential operator defined on the vector bundle $T^{*}N\otimes\widetilde{\Phi}^{*}(\mathcal{V}P)$ with respect to the connections $\nabla^{N},\widetilde\Phi^*\nabla^{P,\mathcal{V}}$. In particular, $\widetilde\Phi$ is vertically harmonic  if and only if 
        \begin{align}\label{eq:vertic}
\delta^{(\nabla^{N},\widetilde\Phi^*\nabla^{P,\mathcal{V}})}(\widetilde{\Phi}^{*}\omega)=\omega\left(\Tr_g(\widetilde{\Phi}^{*}(\bm{T}))\right).
        \end{align} Thus, if $\pi$ has totally geodesic fibers, then $\widetilde{\Phi}^{*}\omega$ satisfies the Hodge gauge equation $\delta^{(\nabla^{N},\widetilde\Phi^*\nabla^{P,\mathcal{V}})}(\widetilde{\Phi}^{*}\omega)=0$. 

   \end{thm}
   \begin{proof}
        Thanks to Corollary \ref{cor:vertical-horizontal-decomposition} one has 
        $$\left[\tau(\widetilde\Phi)\right]^\mathcal{V}=\tau^\mathcal{V}(\widetilde\Phi)+\omega\left(\Tr_g\left[\widetilde\Phi^*\bm{T}\right]\right),$$ with $\tau^\mathcal{V}(\widetilde\Phi)=\Tr_g(\mathrm{II}_{\widetilde\Phi}^\mathcal{V})=\Tr_g(\nabla^\mathcal{V}(d^\mathcal{V}\widetilde\Phi))=\Tr_g(\nabla^\mathcal{V}(\widetilde\Phi^*\omega))$. Considering now a local orthonormal frame $\{U_r\}_{r=1}^n$ on $(N,g)$, we get \begin{align*}
 	\Tr_g(\nabla^\mathcal{V}(\widetilde\Phi^*\omega))= \sum_{r=1}^{n}\left[(\widetilde\Phi^*\nabla^{P,\mathcal{V}})_{U_{r}}\left((\widetilde{\Phi}^{*}\omega)(U_{r})\right)-(\widetilde{\Phi}^{*}\omega)(\nabla^{N}_{U_{r}}U_{r})\right]=-\delta^{(\nabla^{N},\widetilde\Phi^*\nabla^{P,\mathcal{V}})}(\widetilde{\Phi}^{*}\omega)
 \end{align*}and the proof is finished.       

       \end{proof}
       
If $\pi: (P,\widehat{g})\longrightarrow (M,\overline{g})$ is a Kaluza-Klein principal $G$-bundle, then the $\widehat{g}$-vertical projection $\omega\in\Omega^{1}(P,\mathcal{V}P)$ is a principal connection and we have its curvature $2$-form $\widehat\Omega^\omega\in\Omega^2(P,\mathfrak g)$. It is well known that  $\widehat\Omega^\omega$ is a basic $2$-form on $P$;  that is, $\widehat\Omega^\omega$ is $G$-invariant and horizontal (i.e for every vertical vector field $V\in\mathfrak{X}^{\mathcal{V}}(P)$ one has $i_{V}\widehat\Omega^{\omega}=0$). The curvature form defines a $\mathcal V P$-valued $2$-form  $\Omega^\omega=R_\bullet\circ\widehat\Omega^\omega\in \Omega^2(P,\mathcal V P)$, where $R_\bullet$ is the vector bundle isomorphism $R_\bullet\colon\mathfrak g_P\xrightarrow{\sim} \mathcal V P$ described in Proposition \ref{pro:principal-connections}. It is well known that $$\Omega^\omega(D_1,D_2)=-\omega([D_1^\mathcal{H},D_2^\mathcal{H}]),\quad D_1,D_2\in\X(P).$$ Moreover, one has $$R_g\cdot[\Omega^\omega(D_1,D_2)]=\Omega^\omega(R_{g}\cdot D_1,R_{g}\cdot D_2),\quad g\in G,D_1,D_2\in\X(P),$$ where for any $D\in \X(P)$ and $\sigma\in G$, we denote by $R_\sigma\cdot D$ the pushforward of the vector field $D$ under the diffeomorphism $R_\sigma\colon P\xrightarrow{\sim} P$. The previous identity shows that if $D_1, D_2\in \X(P)^G$ are $G$-invariant vector fields, then $\Omega^\omega(D_1,D_2)\in \X^{\mathcal V}(P)^G$ is an invariant vertical vector field. It is well known that there is a natural isomorphism of $\cin(M)$-modules, which at the same time is  a $\cin(M)$-Lie algebra anti-isomorphism, between the space of sections of the adjoint bundle and the space of $G$-invariant vertical vector fields $$\widetilde{(-)}\colon \Gamma(M,\mathrm{ad} P)=\cin(P,\mathfrak g)^G\xrightarrow{\sim}\X^{\mathcal V}(P)^G$$that associates to $\boldsymbol{\nu}\in \cin(P,\mathfrak g)^G$ the $G$-invariant vector field $\widetilde{\boldsymbol{\nu}}\in \X^{\mathcal V}(P)^G$ given by $\widetilde{\boldsymbol{\nu}}_p:=[\boldsymbol{\nu}(p)]^*_p$. 
    
The polarity maps associated to the $\X^{\mathcal V}(P)$-valued $2$-form $\Omega^{\omega}$ and to the Kaluza-Klein metric $\widehat{g}=\Psi^{-1}(\omega,\beta,\overline g)$ define a $C^{\infty}(P)$-linear map $\mathscr{F}^{\omega}:\mathfrak{X}(P)\longrightarrow\mathfrak{X}(P)\otimes_{\cin(P)}\X^\mathcal{V}(P)$  given by $$\mathscr{F}^{\omega}=(p^{-1}_{\widehat{g}}\otimes\Id_{\mathfrak{g}})\circ p_{\Omega^{\omega}}\in T^{1}_{1}(P)\otimes_{\cin(P)}\X^\mathcal{V}(P).$$ That is, $\mathscr{F}^{\omega}$ is the  $\X^\mathcal{V}(P)$-valued horizontal $(1,1)$-tensor field on $P$ such  that if $\langle\ ,\ \rangle_{\mathcal{V}P}$ denotes the restriction of $\widehat g$ to vertical vector fields, then for every $V\in \X^\mathcal{V}(P)$   it holds
\begin{align*}
   \langle V, \Omega^{\omega}\rangle_{\mathcal VP}(D_{1},D_{2}) &=\widehat{g}(\langle V,\mathscr{F}^{\omega}\rangle_{\mathcal VP}(D_{1}),D_{2}),\quad D_{1}, D_{2}\in \mathfrak{X}(P).
\end{align*} Whence, $\langle V,\mathscr{F}^{\omega}\rangle_{\mathcal VP}=p^{-1}_{\widehat{g}}\circ p_{\langle V,\Omega^{\omega}\rangle_{\mathcal VP}}$.  Since $\Omega^\omega$ is horizontal, it follows immediately that the endomorphism $\langle V,\mathscr{F}^{\omega}\rangle_{\mathcal VP}\colon \mathfrak{X}(P)\to \mathfrak{X}(P)$ vanishes on vertical vector fields and preserves horizontal ones. 
\begin{De}\label{defi:Lorentz-endomorphism}
    The $\mathcal{V}P$-valued $(1,1)$-tensor field $\mathscr{F}^{\omega}$ will be called the Lorentz endomorphism of the Kaluza-Klein principal $G$-bundle  $\pi:(P,\widehat{g})\longrightarrow(M,\overline{g})$. The curvature modified metric $\widehat g_{\mathscr{F}^\omega}$ is the $\mathfrak{X}^\mathcal{H}(P)$-valued $2$-covariant symmetric tensor field on $P$ given by $$\widehat g_{\mathscr{F}^\omega}(D_1,D_2)=\frac{1}{2}\left\{ \langle D_1^\mathcal{V},\mathscr{F}^\omega\rangle_{\mathcal VP}(D_2^\mathcal{H})+\langle D_2^\mathcal{V},\mathscr{F}^\omega\rangle_{\mathcal VP}(D_1^\mathcal{H})\right\},\quad D_1,D_2\in\mathfrak{X}(P).$$
\end{De}

  Since a Kaluza-Klein principal bundle  $\pi: (P,\widehat{g}){\longrightarrow} (M,\overline{g})$ is a Riemannian submersion, the restriction $\pi_{*}|_{\mathcal{H}P}: \mathcal{H}P=(\mathcal{V}P)^{\perp_{\widehat{g}}}\overset{\sim}{\longrightarrow} \pi^{*}TM$ gives an isomorphism of vector bundles over $P$. Its inverse 
$h:\pi^{*}TM\overset{\sim}{\longrightarrow} \mathcal{H}P\subset TP$
    is called the  horizontal map of the connection $\omega$. Any $X\in\mathfrak{X}(M)$ induces a vector field  $X^{\#}\in \Gamma(P,\pi^{*}TM)$ along $\pi$ such that its value at $p\in P$ is  $X^{\#}_{p}:=(p,X_{\pi(p)})\in T_{\pi(p)}M$. We denote $X^{h}:=h(X^{\#})\in\X(P)$ and call it the horizontal lift of $X$ with respect to the connection $\omega$. Therefore, $X^{h}\in\mathfrak{X}(P)$ is the unique $\omega$-basic vector field whose projection is $X$. Any $\omega$-basic vector field on $P$ can be written in this way. The horizontal lift map $h\colon\mathfrak{X}(M)\longrightarrow \mathfrak{X}(P)$ provides an isomorphism of $C^{\infty}(M)$-modules between $\mathfrak{X}(M)$ and the $C^{\infty}(M)$-module $\mathfrak{X}_B(P)$ of basic vector fields on $P$. Moreover, the $\cin(P)$-module of horizontal vector fields $\X^H(P)$ satisfies $\X^\mathcal{H}(P)=\cin(P)\otimes_{\cin(M)}\X_B(P)\simeq \cin(P)\otimes_{\cin(M)}\X(M)$.
\begin{De}
    Given a Kaluza-Klein principal $G$-bundle  $\pi:(P,\widehat{g})\longrightarrow (M,\overline{g})$, its Faraday form is the $\X^\mathcal{V}(P)^G$-valued $2$-form  $\Omega_{\omega}\in\Omega^{2}(M)\otimes_{\cin(M)}\X^\mathcal{V}(P)^G$ defined by $$\Omega_{\omega}(X_{1},X_{2})=\Omega^{\omega}(X^h_{1},X^h_{2})=-\omega([X^{h}_{1},X^{h}_{2}])=[X_{1},X_{2}]^h-[X^h_{1},X^h_{2}],\quad X_1,X_2\in\mathfrak{X}(M).$$    We also have a $\mathfrak{X}^\mathcal{V}(P)^G$-valued $(1,1)$-tensor field  $\mathscr{F}_{\omega}\in T^{1}_{1}(M)\otimes_{C^{\infty}(M)}\mathfrak{X}^{\mathcal{V}}(P)^G$ such that 
     \begin{align}\label{relacion}
         \langle V, \Omega_{\omega}\rangle_{\mathcal{V}P}(X_{1},X_{2})=\overline{g}(\langle V, \mathscr{F}_{\omega}\rangle_{\mathcal{V}P}(X_{1}),X_{2}),\quad V\in\mathfrak{X}(M), X_1,X_2\in\mathfrak{X}(M).
     \end{align} One says that $\Omega_\omega$, $\mathscr{F}_{\omega}$ are, respectively, the field strength $2$-form and the Lorentz strength endomorphism of the Kaluza-Klein principal $G$-bundle.
     \end{De}

      Thanks to \cite[Lemma 2]{o1966fundamental}, one gets the following result.
      
       \begin{Lem}\label{lem:A-vs-Curvatura}
   For any horizontal vector fields $D_1,D_2\in\mathfrak{X}(P)$ it holds
    \begin{align}\label{force}
        2\bm{A}(D_{1},D_{2}) &=[D_{1},D_{2}]^{\mathcal{V}}=-\Omega^{\omega}(D_{1},D_{2}).
    \end{align}
    Moreover, if $D_{1}=X_{1}^h, D_{2}=X_{2}^h$ are $\omega$-basic, then one also has
    \begin{align}\label{force2}
        2\bm{A}(D_{1},D_{2}) &=-\Omega_{\omega}(X_{1},X_{2}).
    \end{align}
  \end{Lem}
   \begin{thm}\label{teo:horizontal-hamonic-curvature-modified-metric}
    Let $\pi: (P,\widehat{g})\longrightarrow (M,\overline{g})$ be a Kaluza-Klein principal bundle with $\widehat g=\Psi^{-1}(\omega,\beta,\overline g)$. A map $\widetilde{\Phi}: (N,g)\longrightarrow (P,\widehat{g})$ is horizontally harmonic if and only if  $\Phi=\pi\circ\widetilde \Phi\colon (N,g)\longrightarrow (M,\overline{g})$ satisfies
    \begin{align}\label{asss}
\tau(\Phi)&=-\pi_*\left(\Tr_{g}\big[\widetilde\Phi^{*}\widehat{g}_{\mathscr{F}^{\omega}}\big]\right)-\pi_{*}\left(\Tr_g\big[\widetilde{\Phi}^{*}\bm{T}\big]\right).
    \end{align}
\end{thm}
\begin{proof}[Proof of Theorem (\ref{teo:horizontal-hamonic-curvature-modified-metric})]
Thanks to Theorem \ref{teo:horizontal-component-tension-field}, $\widetilde \Phi$ is horizontally harmonic if and only if  $$\tau(\Phi)=-\pi_{*}\left(\Tr_g\big[\widetilde{\Phi}^{*}(2\bm{A}+\bm{T})\big]\right).$$ Therefore, we just need to compute the term $-\pi_{*}\left(\Tr_g\big[\widetilde{\Phi}^{*}(2\bm{A})\big]\right)$. In order to proceed, we recall that by \cite[2' p. 460]{o1966fundamental} the tensor $\bm A$ is horizontal in its first component. Therefore, for any vector field $D\in\mathfrak{X}(P)$ one has $$2\bm A(D,D)=2\bm A(D^\mathcal{H},D)=2\bm A(D^\mathcal{H},D^\mathcal{H})+2\bm A(D^\mathcal{H},D^\mathcal{V}).$$ 
By Lemma \ref{lem:A-vs-Curvatura} we have $2\bm A(D^\mathcal{H},D^\mathcal{H})=-\Omega^\omega(D^\mathcal{H},D^\mathcal{H})=0$ since $\Omega^\omega$ is antisymmetric. Hence $2\bm A(D,D)=2\bm A(D^\mathcal{H},D^\mathcal{V})$. Now, applying \cite[Lemma 3 3.]{o1966fundamental} we get $\bm A(D^\mathcal{H},D^\mathcal{V})=(\nabla^P_{D^\mathcal{H}}D^\mathcal{V})^\mathcal{H}$. Hence, $2\bm A(D,D)=2\bm A(D^\mathcal{H},D^\mathcal{V})$ is a horizontal vector field.  On the other hand, by \cite[1' p. 460]{o1966fundamental} the map $\bm A(D^\mathcal{H},-)$ is $\widehat g$-skew-symmetric and thus for any vector field $E\in\mathfrak X(P)$ one has \begin{align*}
\widehat g(2\bm A(D^\mathcal{H}, D^\mathcal{V}),E^\mathcal{H})=-\widehat g(D^\mathcal{V},2\bm A(D^\mathcal{H},E^\mathcal{H}))=\widehat g(D^\mathcal{V},\Omega^\omega(D^\mathcal{H},E^\mathcal{H}))
\end{align*}where the last equality follows from Lemma \ref{lem:A-vs-Curvatura}. Now bearing in mind the Lorentz endomorphism, Definition \ref{defi:Lorentz-endomorphism}, we get \begin{align*}
\widehat g(D^\mathcal{V},\Omega^\omega(D^\mathcal{H},E^\mathcal{H}))=\langle D^\mathcal{V},\Omega^\omega\rangle_{\mathcal{V}P}(D^\mathcal{H},E^\mathcal{H})=\widehat g(\langle D^\mathcal{V},\mathscr{F}^\omega\rangle_{\mathcal{V}P}(D^\mathcal{H}),E^\mathcal{H}).
\end{align*}
 Hence $\widehat g(2\bm A(D^\mathcal{H}, D^\mathcal{V}),E^\mathcal{H})=\widehat g(\langle D^\mathcal{V},\mathscr{F}^\omega\rangle_{\mathcal{V}P}(D^\mathcal{H}),E^\mathcal{H})$. Since this is true for any horizontal vector field $E^\mathcal{H}\in\mathfrak X(P)$ and $2\bm A(D^\mathcal{H}, D^\mathcal{V})$, $\langle D^\mathcal{V},\mathscr{F}^\omega\rangle_{\mathcal{V}P}(D^\mathcal{H})$  are horizontal, we conclude that $$2\bm A(D,D)=2\bm A(D^\mathcal{H}, D^\mathcal{V})=\langle D^\mathcal{V},\mathscr{F}^\omega\rangle_{\mathcal{V}P}(D^\mathcal{H})=\widehat g_{\mathscr{F}^\omega}(D,D).$$ Using this identity and considering a local orthonormal frame $\{U_r\}_{r=1}^n$ on $(N,g)$ we obtain \begin{align*}
\pi_*\left(\Tr_g\big[\widetilde{\Phi}^{*}(2\bm{A})\big]\right)&=\pi_*\left(\sum_{r=1}^n(2\bm A)(\widetilde\Phi_*(U_r),\widetilde\Phi_*(U_r))\right)=\pi_*\left(\sum_{r=1}^n\widetilde\Phi^*(\widehat g_{\mathscr{F}^\omega})(U_r,U_r)\right)=\\&=\pi_*(\Tr_g\big[\widetilde\Phi^*\widehat g_{\mathscr{F}^\omega}\big])
 \end{align*}as claimed in the statement.
\end{proof}

\begin{Rem}
As far as we know, equation (\ref{asss}) has not appeared previously in the literature, while a particular case of equation (\ref{eq:vertic}) was obtained  in \cite{manabe1992pluriharmonic} under the simplifying assumption that $\pi$ had totally geodesic fibers. On the other hand, \cite{Catuogno} contains a related result obtained by means of stochastic techniques that however is not able to capture neither the Lorentz force nor the vertical influence of  the second fundamental form of the fibers of the principal bundle and moreover does not lead to Wong's equations, see below Example \ref{ex:magnetic-curves}.    
\end{Rem}
 
This, together with Theorem \ref{teo:vertical-tension-field}, yields the following:
 \begin{Cor}\label{cor:harmonic-equations}
      Let $\pi: (P,\widehat{g})\longrightarrow(M,\overline{g})$ be a Kaluza-Klein principal bundle and let $\widetilde{\Phi}: (N,g)\longrightarrow(P,\widehat{g})$ be a smooth map and consider the composition $\Phi=\pi\circ\widetilde{\Phi}$. Then $\widetilde{\Phi}$ is harmonic if and only if 
      \begin{align}\label{eq:harmonic-equations}
\tau(\Phi) &=-\pi_{*}[\Tr_{g}(\widetilde{\Phi}^{*}\widehat{g}_{\mathscr{F}^{\omega}})]-\pi_{*}[\Tr_{g}(\widetilde{\Phi}^{*}\bm{T})],\quad \ \quad \delta^{(\nabla^{N}, \nabla^\mathcal{V})}(\widetilde{\Phi}^{*}\omega)=\omega[\Tr_{g}(\widetilde{\Phi}^{*}\bm{T})].
      \end{align}      
      \end{Cor}

If $B=\{\xi_1,\ldots,\xi_d\}$ is a basis of the Lie algebra, then one has $$\widehat \Omega^\omega=\sum_{a=1}^d\Omega^{\omega,a}\otimes\xi_a\in\Omega^2(P,\g)\quad \Longrightarrow\quad  \Omega^\omega=\sum_{a=1}^d\Omega^{\omega,a}\otimes\xi_a^*\in\Omega^2(P,\mathcal VP),$$ where the $\Omega^{\omega,a}\in\Omega^2(P)$ are basic $2$-forms over $P$. In general,  $\Omega^{\omega,a}$ is not $G$-invariant, and therefore it does not descend to $M$ and we do not have a direct expression for the Faraday $2$-form $\Omega_\omega$ in terms of the $\Omega^{\omega,a}$. However, if the Lie group $G$ has trivial adjoint group, $\Ad G=\Id_\g$, then $\xi_a^*$ is $G$-invariant. Since $\Omega^\omega$ is $G$-invariant, it follows that $\Omega^{\omega,a}$ is a basic $2$-form, whence it descends to $\Omega^a_\omega\in\Omega^2(M)$. Therefore, setting $\mathscr{F}^a_\omega=p_{\overline g}^{-1}\circ p_{\Omega^a_\omega}\in T^1_1(M)$,  one has $$\Omega_\omega=\sum_{a=1}^d\Omega_\omega^a\otimes\xi_a^*\in\Omega^2(M)\otimes_{\cin(M)}\X^\mathcal{V}(P)^G\quad \Longrightarrow\quad \mathscr{F}_\omega=\sum_{a=1}^d\mathscr F^a_\omega\otimes\xi^*_a\in T_1^1(M) \otimes_{\cin(M)}\X^\mathcal{V}(P)^G.$$
      
In the general case, one can express $\mathscr{F}_\omega$ in a similar way, leading to the following result:

\begin{Pro}\label{pro:Lorentz-strength}
    If $\pi: (P,\widehat{g})\longrightarrow (M,\overline{g})$ is a Kaluza-Klein principal $G$-bundle with $\widehat g=\Psi^{-1}(\omega,\beta,\overline g)$, then its Lorentz strength endomorphism can be written as a finite sum $$\mathscr{F}_\omega=\sum_{a=1}^k\mathscr{F}_{\omega}^a\otimes \widetilde{\boldsymbol{\xi}}_a,\quad \mathscr{F}_{\omega}^a\in T_1^1(M),\widetilde{\boldsymbol{\xi}}_a\in\X^\mathcal{V}(P)^G.$$ Moreover, a smooth map $\widetilde{\Phi}: (N,g)\longrightarrow (P,\widehat{g})$ is horizontally harmonic if and only if the tension field of $\Phi=\pi\circ\widetilde \Phi$ is given by
     \begin{align}\label{eq:tension-field}
         \tau(\Phi)=-\sum_{a=1}^k     [\Phi^*(\mathscr{F}_\omega^a)](p_g^{-1}(\widetilde\Phi^*[p_{\widehat g}(\widetilde{\boldsymbol{\xi}}_a)]))-\pi_{*}\left(\Tr_g\big[\widetilde{\Phi}^{*}\bm{T}\big]\right).
     \end{align} 
     If $G$ has trivial adjoint group, then one can take $k=\dim\g$, $\widetilde{\boldsymbol{\xi}}_a=\xi_a^*$ and $\mathscr{F}^a_\omega=p_{\overline g}^{-1}\circ p_{\Omega^a_\omega}\in T^1_1(M)$. 
\end{Pro}

\begin{proof} 
    Since $\pi: (P,\widehat{g})\longrightarrow (M,\overline{g})$ is a Riemannian submersion, given $E\in\mathfrak{X}(M)$ and $D\in\mathfrak{X}(P)$ that is $\pi$-related to $X\in\mathfrak{X}(M)$, one has \begin{align*}
    \bar g(\pi_*[\widehat g_{\mathscr{F}^\omega}(D,D)],E)&= \bar g(\pi_*[\widehat g_{\mathscr{F}^\omega}(D,D)],\pi_*(E^h))=\widehat g(\widehat g_{\mathscr{F}^\omega}(D,D),E^h)=\\&=\widehat g(\langle D^\mathcal{V},\mathscr{F}^\omega\rangle_{\mathcal{V}P}(D^\mathcal{H}),E^h)=\widehat g(\langle D^\mathcal{V},\mathscr{F}^\omega\rangle_{\mathcal{V}P}(X^h),E^h)=\\&=\langle D^\mathcal{V},\Omega^\omega\rangle_{\mathcal{V}P}(X^h,E^h)=\langle D^\mathcal{V},h^*\Omega^\omega\rangle_{\mathcal{V}P}(X,E)=\\&=\langle D^\mathcal{V},\Omega_\omega\rangle_{\mathcal{V}P}(X,E)=\bar g(\langle D^\mathcal{V},\mathscr{F}_\omega\rangle_{\mathcal{V}P}(X),E).
\end{align*}Since $E\in\mathfrak{X}(M)$ is arbitrary and $\bar g$ is non degenerated, it follows that 
$$\pi_*[\widehat g_{\mathscr{F}^\omega}(D,D)]=\langle D^\mathcal{V},\mathscr{F}_\omega\rangle_{\mathcal{V}P}(X).$$ Now, let $\{U_r\}_{r=1}^n$ be a local orthonormal frame on $(N,g)$. Taking into account that $\widetilde \Phi_*(U_r)$ is $\pi$-related to $\Phi_*(U_r)$, we get
\begin{align*}
    \pi_*(\Tr_g\big[\widetilde\Phi^*\widehat g_{\mathscr{F}^\omega}\big])=\pi_*\left(\sum_{r=1}^n\widetilde\Phi^*(\widehat g_{\mathscr{F}^\omega})(U_r,U_r)\right)=\sum_{r=1}^n \langle [\widetilde\Phi_*(U_r)]^\mathcal{V},\mathscr{F}_\omega\rangle_{\mathcal{V}P}(\Phi_*(U_r)).
\end{align*} We know that $\mathscr{F}_{\omega}$ is an element of the tensor product $ T^{1}_{1}(M)\otimes_{C^{\infty}(M)}\mathfrak{X}^{\mathcal{V}}(P)^G$. Therefore, we have that $\mathscr{F}_{\omega}=\sum_{a=1}^k \mathscr{F}_{\omega}^a\otimes \widetilde{\boldsymbol{\xi}}_a $ is given by a finite sum as in the statement.  Substituting this expression above, using that $\Phi^*(\mathscr F^a_\omega):=\Phi^\#(\mathscr F^a_\omega)\circ\Phi_*\in\Hom(TN,\Phi^*TM)$, after some compuations we obtain \begin{align*}
    \pi_*(\Tr_g\big[\widetilde\Phi^*\widehat g_{\mathscr{F}^\omega}\big])&=\sum_{r=1}^n\sum_{a=1}^k\langle[\widetilde\Phi_*(U_r)]^\mathcal{V},\widetilde\Phi^\#(\widetilde{\boldsymbol{\xi}}_a)\rangle_{\mathcal{V}P}\,  \Phi^\#(\mathscr{F}_\omega^a)(\Phi_*(U_r))=\\&\sum_{a=1}^k\Tr_g\left(\widetilde\Phi^*[p_{\widehat g}(\widetilde{\boldsymbol{\xi}}_a)]\otimes \Phi^*(\mathscr{F}_\omega^a)\right)=\sum_{a=1}^k [\Phi^*(\mathscr{F}_\omega^a)](p_g^{-1}(\widetilde\Phi^*[p_{\widehat g}(\widetilde{\boldsymbol{\xi}}_a)])).
\end{align*}In the last equality we have used  that given $\theta,\theta'\in\Omega^1(N)$, $D\in\mathfrak{X}(N)$ one has $\Tr_g(\theta\otimes\theta'\otimes D)=\Tr_g(\theta\otimes\theta')\cdot D=\theta'(p_g^{-1}\theta)\cdot D$. Therefore (\ref{eq:tension-field}) follows from Theorem \ref{teo:horizontal-hamonic-curvature-modified-metric}. The final claim is just a consequence of the discussion preceding this proposition.
\end{proof}

\section{Kaluza-Klein harmonic maps and generalized magnetic maps}
Now we introduce the key players of this paper.
 \begin{De}\label{defi:generalized-magnetic-maps} A Kaluza-Klein harmonic map is a harmonic map whose target is a Kaluza-Klein principal bundle. Let $\pi:(P,\widehat g)\longrightarrow (M,\overline g)$ be a Kaluza-Klein principal $G$-bundle with $\widehat g=\Psi^{-1}(\omega,g_\ad,\overline g)$ and $(N,g)$  a Riemannian manifold.  The Euler-Lagrange equations  for a Kaluza-Klein harmonic map $\widetilde{\Phi}: (N,g)\longrightarrow (P,\widehat g)$   \begin{align}
\tau(\Phi) &=-\pi_{*}[\Tr_{g}(\widetilde{\Phi}^{*}\widehat{g}_{\mathscr{F}^{\omega}})]-\pi_{*}[\Tr_{g}(\widetilde{\Phi}^{*}\bm{T})],\quad  \quad \delta^{(\nabla^{N}, \widetilde\Phi^*\nabla^{P,\mathcal{V}})}(\widetilde{\Phi}^{*}\omega)=\omega[\Tr_{g}(\widetilde{\Phi}^{*}\bm{T})],
      \end{align}are called the generalized Wong's equations for the  Riemannian manifold $(N,g)$ and the Kaluza-Klein principal $G$-bundle $\pi:(P,\widehat g)\longrightarrow (M,\overline g)$, (see below  Example \ref{ex:magnetic-curves} for the motivation of this terminology). We denote by $\mathbf{Har}_{KK}((N,g),(P,\widehat g))$ the space of Kaluza-Klein harmonic maps whose target is $(P,\widehat g)$.
 
A smooth map  $\Phi\colon (N,g)\longrightarrow (M,\overline g)$ is a generalized magnetic map with background Kaluza-Klein principal bundle $(P,\widehat g)$  if there exists $\widetilde{\Phi}\in \mathbf{Har}_{KK}((N,g),(P,\widehat g)) $ such that  $\Phi=\pi\circ\widetilde{\Phi}$, or equivalently  if $\widetilde\Phi$ satisfies  \begin{align}\label{eq:Lorentz}
\tau(\Phi) &=-\pi_{*}[\Tr_{g}(\widetilde{\Phi}^{*}\widehat{g}_{\mathscr{F}^{\omega}})]-\pi_{*}[\Tr_{g}(\widetilde{\Phi}^{*}\bm{T})],\quad \ \quad\\\label{eq:Hodge} \delta^{(\nabla^{N}, \widetilde\Phi^*\nabla^{P,\mathcal{V}})}(\widetilde{\Phi}^{*}\omega)&=\omega[\Tr_{g}(\widetilde{\Phi}^{*}\bm{T})].
      \end{align} We call (\ref{eq:Lorentz}) or (\ref{eq:tension-field})
      the generalized Lorentz equation satisfied by the generalized magnetic map $\Phi$. The terms on the right side of (\ref{eq:Lorentz}) are, respectively, the Lorentz and internal space horizontal shape strengths.  We say that (\ref{eq:Hodge}) is the generalized Hodge gauge equation and the term on its left is the internal space vertical shape strength. We will denote by $\mathbf{Mag}^{(P,\widehat g)}((N,g),(M,\overline g))$ the space of generalized magnetic maps from $(N,g)$ into $(M,\overline g)$ with background Kaluza-Klein principal bundle $\pi:(P,\widehat g)\longrightarrow (M,\overline g)$.  One says that a generalized magnetic map  $\Phi$ is uncharged if its Lorentz strength vanishes; that is,  $\pi_*\left(\Tr_{g}\big[\widetilde\Phi^{*}\widehat{g}_{\mathscr{F}^{\omega}}\big]\right)=0$.
   \end{De}

Now notice that there is a natural right action of the gauge group $\cin(N,G)$
	\begin{align*}
		\rho^\mathcal{V}: C^{\infty}(N&,P)\times\cin(N,G)\longrightarrow C^{\infty}(N,P)\\&\widetilde\Phi\quad,\quad\quad\chi\quad\quad\ \ \longmapsto\quad\quad \widetilde\Phi\cdot\chi
	\end{align*}
such that the natural surjective map \begin{align*}
 	\boldsymbol{\pi}\colon \cin(N&,P)\longrightarrow  \cin(N,M)\\
 	&\widetilde\Phi\quad\longmapsto\quad \Phi=\pi\circ\widetilde\Phi
 \end{align*}identifies the $\cin(N,G)$-orbits with maps in $\cin(N,P)$ having the same projection to $M$. Thus \begin{thm}\label{teo:aplicaciones-magneticas-cociente}
	Let $\pi:(P,\widehat g)\longrightarrow (M,\overline g)$ be a Kaluza-Klein principal $G$-bundle with $\widehat g=\Psi^{-1}(\omega,g_\ad,\overline g)$ and $(N,g)$  a Riemannian manifold. There is a natural surjective map \begin{align*}
 	\boldsymbol{\pi}\colon\mathbf{Har}_{KK}((N,g),&{(P,\widehat g)})\longrightarrow  \mathbf{Mag}^{(P,\widehat g)}((N,g),(M,\overline g))\\
 	&\widetilde\Phi\quad\quad\ \longmapsto\quad\quad\quad\quad \Phi=\pi\circ\widetilde\Phi
 \end{align*}that establishes a quotient identification $$\mathbf{Mag}^{(P,\widehat g)}((N,g),(M,\overline g))=\mathbf{Har}_{KK}((N,g),{(P,\widehat g)})/\sim$$ under the equivalence  relation: $\widetilde\Phi_1\sim\widetilde\Phi_2$ iff there exists  $\chi\in\cin(N,G)$ with $\widetilde\Phi_2=\widetilde\Phi_1\cdot\chi\Leftrightarrow \tau(\widetilde\Phi_1\cdot\chi)=0$.

\end{thm}

Now we provide a necessary condition for generalized magnetic maps.

\begin{Pro}\label{pro:necessary-condition-generalized-magnetic-maps}
	Let $\pi:(P,\widehat g)\longrightarrow (M,\overline g)$ be a Kaluza-Klein principal $G$-bundle and $(N,g)$  a Riemannian manifold.  A necessary condition for a smooth map  $\Phi\colon N\longrightarrow M$ to be a generalized magnetic map is that the principal $G$-bundle  $\Phi^*P\to N,$ obtained as the pullback of the principal $G$-bundle $\pi\colon P\to M$ along $\Phi$,   must be trivializable.
\end{Pro}

\begin{proof}
	If $\Phi$ is a generalized magnetic map, then there exist a horizontal harmonic map $\widetilde\Phi$ into $P$ such that the following diagram is commutative \[\begin{tikzcd}\label{CO}
	(N,g) && (P,\widehat{g}) \\
	\\
	&& (M,\overline{g})
	\arrow["\widetilde{\Phi}", from=1-1, to=1-3]
	\arrow["\Phi"', from=1-1, to=3-3]
	\arrow["\pi", from=1-3, to=3-3]
\end{tikzcd}\]By the universal property of the fiber product, one has that $\widetilde\Phi$ defines a global section of the pullback principal $G$-bundle $\Phi^*P\to N$. Since a principal $G$-bundle is trivializable if and only if it admits a global section, the result follows.
\end{proof}

When this necessary condition is satisfied, one would like to know whether generalized magnetic maps actually do exist. Before addressing this question we first study the action of the gauge group. The following can be proved by careful calculation.

\begin{Pro}[gauge variation-harmonic gauge fixing]\label{pro:gauge-variation-harmonic-gauge-fixing} If $\widetilde\Phi\colon (N,g)\to (P,\widehat g)$ is a smooth map into a Kaluza-Klein principal $G$-bundle with $\widehat g=\Psi^{-1}(\omega,g_\ad,\overline g)$, then for any  $\chi\in\cin(N,G)$ one has:
\begin{align*}
	[\tau(\widetilde\Phi\cdot\chi)]^\mathcal{H}&=R_{\chi,*}\left[[\tau(\widetilde\Phi)]^\mathcal{H}+\mathcal C^\mathcal{H}_{\widetilde\Phi,\omega,\bm{A},\bm{T}}(\chi,\chi)\right],\\ [\tau(\widetilde\Phi\cdot\chi)]^\mathcal{V}&=R_{\chi,*}\left[[\tau(\widetilde\Phi)]^\mathcal{V}-\delta^{(\nabla^N,\widetilde\Phi^*\nabla^{P,\mathcal V})}
( \theta^{R,\widetilde\Phi,\chi})+\mathcal C^\mathcal{V}_{\widetilde\Phi,\omega,\bm{T}}(\chi,\chi)\right],\
\end{align*} where $\theta^{R,\widehat\Phi,\chi}=R^P_{\bullet;\widetilde\Phi}[\chi^*\theta^R]$, with $\theta^R$ the right invariant Maurer-Cartan $1$-form of $G$, and $\mathcal C^\mathcal{V}_{\widetilde\Phi,\omega,\bm{T}}(\chi,\chi)$, $\mathcal C^\mathcal{H}_{\widetilde\Phi,\omega,\bm{A},\bm{T}}(\chi,\chi)$ are quadratic first order differential operators on $\chi$. Hence, given a  Kaluza-Klein harmonic map $\widetilde\Phi\in\mathbf{Har}_{KK,\Phi}((N,g),{(P,\widehat g)}):=\bm{\pi}^{-1}(\Phi)$ if the nonlinear elliptic harmonic gauge fixing equations  $$\bm{HGF}_{\widetilde\Phi}(\chi)=0\quad \equiv\quad \mathcal C^\mathcal{H}_{\widetilde\Phi,\omega,\bm{A},\bm{T}}(\chi,\chi)=0,\quad\quad  \delta^{(\nabla^N,\widetilde\Phi^*\nabla^{P,\mathcal V})}
( \theta^{R,\widetilde\Phi,\chi})=\mathcal C^\mathcal{V}_{\widetilde\Phi,\omega,\bm{T}}(\chi,\chi)$$ have a solution, then $\widetilde\Phi\cdot\chi$ is also harmonic. Whence, the solutions  to the harmonic gauge fixing equations determine the equivalence relation $\sim$ on $\mathbf{Har}_{KK}((N,g),{(P,\widehat g)})$ and the equivalence class of $\widetilde\Phi$ is $$\mathbf{Har}_{KK,\Phi}((N,g),{(P,\widehat g)})=\{\chi\in\cin(N,G)\colon \bm{HGF}_{\widetilde\Phi}(\chi)=0\}.$$ In particular,  every element of $G$, thought as a constant map, acts on $\mathbf{Har}_{KK}((N,g),{(P,\widehat g)})$ reproducing the action, described in Proposition \ref{pro:invarianza-funcional-energia-Dirichlet}, of $G$ as a subgroup of the Lie group of isometries $O(P,\widehat g)$. \end{Pro}

However, the crucial point we must address is the existence problem for (generalized) magnetic maps. In this respect we have:
\begin{thm}\label{teo:existence-generalized-magnetic-maps}
Let $\pi:(P,\widehat g)\longrightarrow (M,\overline g)$ be a Kaluza-Klein principal $G$-bundle and $(N,g)$  a Riemannian manifold. We have the following existence results for (generalized) magnetic maps:
\begin{enumerate}
	\item $\dim N=1$.
\begin{itemize}
\item[1.a)] If $N$ is an open subset of $\R$ and $(P,\widehat g)$ is geodesically complete, then every initial condition $D_x\in TM$ determines a (generalized) magnetic map $\Phi_{D_x}$ defined everywhere on $N$.	
\item[1.b)] If we endow $N=S^1$ with any Riemannian metric $g$ and $(P,\widehat g)$ is compact, then every non-trivial free homotopy class $\ell\in\pi_0(LM)$ of loops in $M$ in the image  of $\pi_*\colon \pi_0(LP)\to \pi_0(LM)$ contains a (generalized) magnetic map $\Phi\colon (S^1,g)\to (M,\overline g)$. If $G$ is connected, then $\pi_*$ is surjective, thus any class $\ell\in\pi_0(LM)$ can be realized by a (generalized) magnetic map.
\end{itemize}
\item $\dim N=2$.
\begin{itemize}
\item If $N, P$ are compact and $\pi_2(P)=0$, then any homotopy class $\boldsymbol{\Phi}\in[N,M]$ that trivializes $P$ can be represented by a  (generalized) magnetic map.
\end{itemize}
\item $\dim N\geq 3$.
\begin{itemize}
\item If $N$, $P$ are compact, then for any homotopy class $\boldsymbol{\Phi}\in[N,M]$ that trivializes $P$ there exists a Riemannian metric $g'$ in $N$ conformally equivalent to $g$ and a smooth map $\Phi\in\boldsymbol{\Phi}$ such that $\Phi\colon (N,g')\to (M,\overline g)$ is a (generalized) magnetic map.
 \end{itemize}
\item $\dim N$ arbitrary.
\begin{itemize}
\item If $N$, $P$ are compact and $(P,\widehat g)$ has non positive sectional curvature, then any homotopy class $\boldsymbol{\Phi}\in[N,M]$ that trivializes $P$ can be represented by a   (generalized) magnetic map.
\end{itemize}
\end{enumerate}
\end{thm}

\begin{proof} 1.a) Given $D_x\in T_xM$, since $\pi$ is a submersion, there exists $D_p\in T_pP$ such that $\pi_*D_p=D_x$. Since $P$ is complete there exists a geodesic $\widetilde\Phi_{D_p}$ in $P$ with initial condition $D_p$ that is defined everywhere on $N$. Taking $\Phi_{D_x}:=\pi\circ \widetilde\Phi_{D_p}$ gives the desired harmonic magnetic map. 

	1.b) The class $\ell$ lifts to a non trivial free homotopy class of loops $\widetilde\ell$ in $P$. It is a well known, see for instance  \cite[Theorem 6.10]{Berger} that $\widetilde\ell$ can be represented by a closed geodesic $\widetilde\gamma \colon S^1\to P$ that after reparametrization, if necessary, is a harmonic map. Therefore, $\gamma=\pi\circ\widetilde\gamma$ is the desired (generalized) magnetic map.  
	
	2) Since $\boldsymbol{\Phi}$ trivializes $P$ it follows that there exists a homotopy class $\widetilde{\boldsymbol{\Phi}}\in[N,P]$ such that $\boldsymbol{\Phi}=\pi_*(\widetilde{\boldsymbol{\Phi}})$. By \cite[Theorem 5.1]{SU} if follows that there exists a harmonic map $\widetilde\Phi\colon N\to P$ that represents $\widetilde{\boldsymbol{\Phi}}$, hence $\Phi=\pi\circ\widetilde\Phi$ is the required  (generalized) magnetic map.
	
	3), 4) Proceeding as in the proof of 2) the results follow  from \cite{EF} and \cite[Thm. p. 156]{eells1964harmonic}.
\end{proof}

In the special situation where the fibers are totally geodesic, one has the following:

\begin{Lem}
    If the fibers of a Kaluza-Klein principal bundle are totally geodesic (i.e $\bm{T}=0$), then the generalized Lorentz equation reduces to
    \begin{align}\label{cuuu}
        \tau(\Phi) &=-\pi_{*}(\Tr_{g}[\widetilde{\Phi}^{*}\widehat{g}_{\mathscr{F}}]).
    \end{align} \end{Lem}

\section{Local expressions for Kaluza-Klein principal bundles}

We can make  Proposition \ref{pro:Lorentz-strength} somewhat more explicit  by considering local sections of a Kaluza-Klein principal $G$-bundle.  Indeed given a local section $s\in\Gamma(U,P)$ inducing the $G$-equivariant trivialization $ P_U\underset{\sim}{\xrightarrow{\varphi_U}} U\times G$, with $\varphi_U=(\pi,\widehat\varphi_U)$, then the isomorphism of $\cin(M)$-modules $$\widetilde{(-)}\colon \Gamma(M,\mathrm{ad} P)=\cin(P,\mathfrak g)^G\xrightarrow{\sim}\X^{\mathcal V}(P)^G$$ yields by restriction an isomorphism of $\cin(U)$-modules $$ \widehat{(-)}\colon  \cin(U,\mathfrak{g})\underset{\sim}{\xrightarrow{{\varphi_{U\bullet}}}} \cin(P_U,\mathfrak g)^G\underset{\sim}{\xrightarrow{\widetilde{(-)}}}\X^\mathcal{V}(P_U)^G$$ such that for $\boldsymbol{\xi}\in\cin(U,\mathfrak g)$ one has $[\varphi_{U\bullet}(\boldsymbol{\xi})](p):=\Ad_{\widehat\varphi^{-1}_U(p)}\boldsymbol{\xi}(\pi(p))$ and $\widehat{\boldsymbol{\xi}}:=\widetilde{[\varphi_{U\bullet}(\boldsymbol{\xi})]}$. Expressing this isomorphism  on the trivializing open set $U\times G$ we get $\varphi_U\cdot\widehat {\boldsymbol{\xi}}=\boldsymbol{\xi}^R$ where $\boldsymbol{\xi}^R\in \X(U\times G)$ is the right invariant vector field given by $\boldsymbol{\xi}^R_{(x,g)}=(0_x,[\boldsymbol{\xi}(x)]^R_g)$, with $[\boldsymbol{\xi}(x)]^R_g=R_{g_*}(\boldsymbol{\xi}(x))$. Now a direct derivation proves the following result.

\begin{Pro}\label{pro:local-expressions-Kaluza-Klein}
Let $\pi:(P,\widehat g)\longrightarrow (M,\overline g)$ be a Kaluza-Klein principal $G$-bundle with $\widehat{g}=\Psi^{-1}(\omega,\beta,\bar g)$. Given a local section $s\in\Gamma(U,P)$, we have the local Kaluza-Klein metric $\widehat g_U=\Psi^{-1}(\omega_U,\beta_U,\overline g_U)$. If the local connection $\omega_U$ has gauge potential $A$ and field strength $F=dA+[A,A]$ defined on $U$ by the section $s$, then it holds:
\begin{enumerate}
\item The local vector field valued curvature $2$-form $\Omega^{\omega_U}$ is given by $$\Omega^{\omega_U}((D_1,E_1),(D_2,E_2))=[F(D_1,D_2)]^R,\quad D_1,D_2\in\X(U),E_1,E_2\in\X(G).$$
\item The local field strength $2$-form $\Omega_{\omega_U}$ is given by $$\Omega_{\omega_U}(D_1,D_2)=[F(D_1,D_2)]^R,\quad D_1,D_2\in\X(U),$$
\end{enumerate}

If $B=\{\xi_1,\ldots,\xi_d\}$ is a basis of the Lie algebra $\mathfrak g$ and the field strength is $F=\sum_{a=1}^d F^a\otimes\xi_a$,\ with  $F^a\in\Omega^2(U)$, then one has: \begin{align*}
 	 \text{Local field strength}\ 2\text{-form}\quad &\Omega_{\omega_U}=\sum_{a=1}^d F^a\otimes\xi_a^R,\\
 	\text{Local Lorentz strength endomorphism}\quad &\mathscr{F}_{\omega_U}= \sum_{a=1}^d\mathscr F^a_{\omega_U}\otimes \xi_a^R,\quad \mathscr F^a_{\omega_U}=p_{\overline g_U}^{-1}(F^a).
 \end{align*} 
Moreover, a smooth map $\widetilde{\Phi}: (N,g)\longrightarrow (P,\widehat{g})$ is horizontally harmonic if and only if for every local section $s\in \Gamma(U,P)$ the tension field of $\Phi=\pi\circ\widetilde \Phi$  is given on the open subset $\Phi^{-1}(U)\subset N$ by
     \begin{align}\label{eq:local-tension-field}
         \tau(\Phi)=-\sum_{a=1}^d     [\Phi^*(\mathscr{F}^a_{\omega_U})](p_g^{-1}(\widetilde\Phi^*[p_{\widehat g_U}(\xi^R_a)]))-\pi_{*}\left(\Tr_g\big[\widetilde{\Phi}^{*}\bm{T}\big]\right).
     \end{align}
     \end{Pro}

Now we use the notation introduced in Section \ref{subsection:local-expressions} and in particular we consider on the open set $U\times G$ the reference frame $\mathcal R=\{E_\mu,E_a\}_{\mu=1,a=1}^{m,d}$   introduced there. Bearing in mind Proposition \ref{pro:Levi-Civita-connection}, a somewhat lengthy but straightforward computation proves the following results.

 \begin{Pro}\label{pro:local-expression-Wong-Lorentz}
The O'Neill tensors of a Kaluza-Klein  principal $G$-bundle $\pi\colon (P,\widehat g)\to (M,\overline g)$ are  expressed in the reference  frame $\mathcal R=\{E_\mu,E_a\}_{\mu=1,a=1}^{m,d}$ by \begin{align*}
 	\mathbf{T}(E_\mu,E_\nu)&=\mathbf{T}(E_\mu,E_a)=0,\\
 	\mathbf{T}(E_a,E_\mu)&=\left[\nabla_{E_a}E_\mu\right]^\mathcal{V}=\frac{\bar\beta^{db}}{2} B_{a\mu d}\, E_b,\\
 	\mathbf{T}(E_a,E_b)&=\left[\nabla_{E_a}E_b\right]^\mathcal{H}=-\frac{\bar g^{\mu\rho}}{2} B_{a\mu b}\, E_\rho,\\
 	\mathbf{A}(E_\mu,E_\nu)&=\left[\nabla_{E_\mu}E_\nu\right]^\mathcal{V}=-\frac{1}{2}F_{\mu\nu}^a E_a,\\
 	\mathbf{A}(E_\mu,E_a)&=\left[\nabla_{E_\mu}E_a\right]^\mathcal{H}=-\frac{1}{2}\bar\beta_{ab}F^b_{\alpha\mu}\bar g^{\alpha\rho} E_\rho,\\
 	\mathbf{A}(E_a,E_\mu)&=\mathbf{A}(E_a,E_b)=0.
 \end{align*} where $B_{a\mu d}=\partial_\mu\bar\beta_{ad}-\bar\beta([A_\mu,\xi_a],\xi_d)-\bar\beta(\xi_a,[A_\mu,\xi_d]) $.
 \end{Pro}
 
Let us consider now a smooth map $\widetilde{\Phi}: N\longrightarrow P$ and its composition $\Phi=\pi\circ\widetilde \Phi\colon N\longrightarrow M$ with the projection of a Kaluza-Klein principal $G$-bundle. Shrinking $U$, if needed, we can assume without loss of generality that $\Phi^{-1}(U)$ is an open chart of $N$ with coordinate functions $\{y^1,\ldots,y^n\}$.
\begin{thm}
Let $\pi: (P,\widehat{g})\longrightarrow (M,\overline{g})$  be a  Kaluza-Klein principal $G$-bundle,  $\widetilde{\Phi}: (N,g)\longrightarrow (P,\widehat{g})$ be a smooth map and  $\Phi=\pi\circ\widetilde \Phi\colon (N,g)\longrightarrow (M,\overline{g})$ the composed map. The map $\widetilde\Phi$ is harmonic if and only if  on every local trivialization $U\times G$ where $\widetilde \Phi=(\Phi,\widehat \Phi)$ it satisfies the local generalized Wong's equations \begin{align}\nonumber
 g^{ij}\left\{\frac{\partial^2\Phi^\rho}{\partial y^i\partial y^j}+\frac{\partial\Phi^\mu}{\partial y^i} \frac{\partial\Phi^\nu}{\partial y^j}\Phi^{*}(\overline\Gamma_{\mu\nu}^\rho)-\Gamma_{ij}^k\frac{\partial{\Phi^\rho}}{\partial y^k}\right\}&=\\\label{eq:Local-Lorentz} 
 g^{ij}\left\{
 	\frac{\partial\Phi^\mu}{\partial y^i}\,\Phi^{*}(\bar\beta_{bc}F^c_{\alpha\mu}\bar g^{\alpha\rho})+\frac{1}{2}	\left[\Phi^*A^a(\partial_i)+\widehat\Phi^*\theta^a(\partial_i)
\,\Phi^{*}({\bar g^{\mu\rho}} B_{a\mu b})\right]
\right\}& 
\left[\Phi^*A^b(\partial_j)+\widehat\Phi^*\theta^b(\partial_j)\right],	
 \end{align}
 \begin{align}\nonumber
	&\frac{g^{ij}}{2}\left[\Phi^*A^a(\partial_i)+\widehat\Phi^*\theta^a(\partial_i)\right]\,\left\{\left[\Phi^*A^c(\partial_j)+\widehat\Phi^*\theta^c(\partial_j)\right]\,\Phi^{*}\left({\bar\beta^{db}}L_{acd}\right)+\frac{\partial\Phi^\nu}{\partial y^j}\,\Phi^{*}({\bar\beta^{db}} B_{a\nu d})\right\}=\\ \label{eq:Local-Hodge-gauge}&=g^{ij}\left\{\Gamma_{ij}^k\left[\Phi^*A^b(\partial_k)+\widehat\Phi^*\theta^b(\partial_k)\right]-\frac{\partial}{\partial y^i}\left[\Phi^*A^b(\partial_j)+\widehat\Phi^*\theta^b(\partial_j)\right]\right\},
\end{align}with $L_{acd}= \bar\beta (\xi_a,[\xi_c,\xi_d])-\bar\beta (\xi_c,[\xi_d,\xi_a])-\bar\beta (\xi_d,[\xi_a,\xi_c])$ and (\ref{eq:Local-Lorentz}), (\ref{eq:Local-Hodge-gauge}) are the generalized Lorentz and Hodge gauge equations, respectively.
 \end{thm}
 
 \begin{Rem}
 These equations provide a natural generalization of Wong's equations, which describe the motion of particles coupled to a gauge field, see for instance \cite{Montgomery}. The generalized equations represent the dynamics of extended objects, like strings or branes, under the action of a gauge field. 	
 \end{Rem}

\section{The geometry of the fibres of a Kaluza-Klein principal bundle}
\begin{Lem}\label{lem:T-vanishing}
    Let $\pi:(P,\widehat g)\longrightarrow (M,\overline g)$ be a Kaluza-Klein principal $G$-bundle with $\widehat{g}=\Psi^{-1}(\omega,\beta,\bar g)$.\\
    If $\nabla^{P}$ is the Levi-Civita connection of $\widehat{g}$ and  $V, V_{1}, V_{2}$ are vertical vector fields on $P$ and $H,H_{1}, H_{2}$ are $\omega$-horizontal vector fields, then it holds
\begin{align*}
    \bm{T}(V_{1},V_{2}) &=\nabla^{P}_{V_{1}}V_{2}-\omega(\nabla^{P}_{V_{1}}V_{2}), \quad\quad \bm{T}(V,H)=\omega(\nabla^{P}_{V}H), \quad\quad \bm{T}(H_1,H_2)=0, \quad\quad \bm{T}(H,V)=0.
\end{align*}
In particular, if we consider $\xi,\xi_{1},\xi_{2}\in\mathfrak{g}$, we have:
\begin{align*}
    \bm{T}(\xi^{*}_{1},\xi^{*}_{2}) &=\nabla^{P}_{\xi^{*}_{1}}\xi^{*}_{2}-\omega(\nabla^{P}_{\xi^{*}_{1}}\xi^{*}_{2}), \quad \quad \bm{T}(\xi^{*},H)=\omega(\nabla^{P}_{\xi^{*}}H)=\omega(\nabla^{P}_{H}\xi^{*}).
\end{align*} 
    Moreover, the following conditions are equivalent:
    \begin{enumerate}
        \item $\bm{T}=0$, i.e., the fibers on the Riemannian submersion $\pi: (P,\widehat{g})\longrightarrow (M,\overline{g})$ are totally geodesic.\label{tot1}
\item For any $\xi_{1}, \xi_{2}\in\mathfrak{g}$ and any $\omega$-horizontal vector field $H$ we have\quad  $\widehat{g}(\nabla^{P}_{\xi^{*}_{1}}\xi^{*}_{2},H) =0.$ \label{tot2}
\item For any $\xi_{1}$, $\xi_{2}\in\mathfrak{g}$ and any $\omega$-horizontal vector field $H$ we have\quad $\widehat{g}(\nabla^{P}_{\xi^{*}_{1}}H,\xi^{*}_{2}) =0$. \label{tot3}
\item For any $\xi_{1}, \xi_{2}\in\mathfrak{g}$ and any $\omega$-horizontal vector field $H$ we have\quad $\widehat{g}(\nabla^{P}_{H}\xi^{*}_{1},\xi^{*}_{2}) =0.$ \label{tot4}
\item For any $\xi_{1}, \xi_{2}\in\mathfrak{g}$ and any $\omega$-horizontal vector field $H$ we have\quad  $H[\widehat{g}(\xi^{*}_{1},\xi^{*}_{2})] =0.$
\end{enumerate}
\end{Lem}
\begin{proof}
    The first assertions are derived from the fact that $\omega$ is the orthogonal projector onto the vertical sub-bundle $\mathcal{V}P\longrightarrow P$ with respect to the Kaluza-Klein metric $\widehat{g}$. The equality for $T(\xi^*,H)$ follows since $\mathrm{Tor}_{\nabla^P}=0$ and the fact that $[\xi^*,H]$ is always horizontal, see \cite[Lemma p. 78]{kobayashi1963foundations}. Regarding the equivalences, we have 
    \begin{align*}
        0 &=\bm{T}(\xi^{*}_{1},\xi^{*}_{2})=(\nabla^{P}_{\xi^{*}_{1}}\xi^{*}_{2})^{\mathcal{H}}=0\iff\widehat{g}(\nabla^{P}_{\xi^{*}_{1}}\xi^{*}_{2},H)=0, \forall H,\\
        0 &= \bm{T}(\xi^{*}_{1},H)=(\nabla^{P}_{\xi^{*}_{1}}H)^{\mathcal{V}}=0\iff \widehat{g}(\nabla^{P}_{\xi^{*}_{1}}H,\xi^{*}_{2})=0, \forall\xi_{2}.
    \end{align*}
    However, since $\nabla^{P}$ is the Levi-Civita connection of $\widehat{g}$, it holds $\nabla^{P}\widehat{g}=0$ and $\mathcal{V}P\perp\mathcal{H}P$, hence
    \begin{align*}
        0 &=\xi^{*}_{1}[\widehat{g}(H,\xi^{*}_{2})]=\widehat{g}(\nabla^{P}_{\xi^{*}_{1}}H,\xi^{*}_{2})+\widehat{g}(H,\nabla^{P}_{\xi^{*}_{1}}\xi^{*}_{2})=0\iff\widehat{g}(\nabla^{P}_{\xi^{*}_{1}}H,\xi^{*}_{2})=-\widehat{g}(H,\nabla^{P}_{\xi^{*}_{1}}\xi^{*}_{2}).
    \end{align*}
    As $\mathrm{Tor}_{\nabla^{P}}=0$, it holds
    \begin{align*}
        \widehat{g}(\nabla^{P}_{\xi^{*}_{1}}H,\xi^{*}_{2})=\widehat{g}(\nabla^{P}_{H}\xi^{*}_{1}+[\xi^{*}_{1},H],\xi^{*}_{2})=\widehat{g}(\nabla^{P}_{H}\xi^{*}_{1},\xi^{*}_{2}),
    \end{align*}
    where we have taken into account again that $[\xi^{*}_{1},H]$ is horizontal. Therefore, we get the equalities
    \begin{align*}
        -\widehat{g}(H,\nabla^{P}_{\xi^{*}_{1}}\xi^{*}_{2}) &=\widehat{g}(\nabla^{P}_{\xi^{*}_{1}}H,\xi^{*}_{2})=\widehat{g}(\nabla^{P}_{H}\xi^{*}_{1},\xi^{*}_{2})
    \end{align*}
    from which we deduce the equivalence of (\ref{tot1}), (\ref{tot2}), (\ref{tot3}), (\ref{tot4}).\\
Finally, using again that $\nabla^{P}\widehat{g}=0$, the equality $\widehat{g}(\nabla^{P}_{H}\xi^{*}_{1},\xi^{*}_{2})=-\widehat{g}(H,\nabla^{P}_{\xi^{*}_{1}}\xi^{*}_{2})$ we have just tested, and that $\Tor_{\nabla^{P}}=0$, we have
\begin{align*}
    H[\widehat{g}(\xi^{*}_{1},\xi^{*}_{2})] &=\widehat{g}(\nabla^{P}_{H}\xi^{*}_{1},\xi^{*}_{2})+\widehat{g}(\xi^{*}_{1}.\nabla^{P}_{H}\xi^{*}_{2})=-\widehat{g}(H,\nabla^{P}_{\xi^{*}_{1}}\xi^{*}_{2})-\widehat{g}(H,\nabla^{P}_{\xi^{*}_{2}}\xi^{*}_{1})=\\
    &=-\widehat{g}(H,\nabla^{P}_{\xi^{*}_{1}}\xi^{*}_{2}+\nabla^{P}_{\xi^{*}_{2}}\xi^{*}_{1})=-\widehat{g}(H,\nabla^{P}_{\xi^{*}_{1}\xi^{*}_{2}}+\nabla^{P}_{\xi^{*}_{1}}\xi^{*}_{2}+[\xi^{*}_{2},\xi^{*}_{1}])=-2\widehat{g}(H,\nabla^{P}_{\xi^{*}_{1}}\xi^{*}_{2})
\end{align*}
where we have used that $[\xi^{*}_{2},\xi^{*}_{1}]$ is vertical. Therefore, we  finally obtain the following equalities
\begin{align*}
    -\widehat{g}(H,\nabla^{P}_{\xi^{*}_{1}}\xi^{*}_{2})=\widehat{g}(\nabla^{P}_{\xi^{*}_{1}}H,\xi^{*}_{2})=\widehat{g}(\nabla^{P}_{H}\xi^{*}_{1},\xi^{*}_{2})=\frac{1}{2}H[\widehat{g}(\xi^{*}_{1},\xi^{*}_{2})],
\end{align*}
finishing the proof.
\end{proof}

Now let us recall, see for instance \cite[Chapter VI]{GHV2}, that for any vector space $V$ the connection $\omega$ induces  an  exterior covariant differential operator $$d^\omega\colon \Omega^\bullet(P,V)\to\Omega^{\bullet+1}(P,V)$$ given by $d^\omega:=\mathcal{H}^*\circ d$, where $\mathcal H=\Id_{TP}-\omega$ is the horizontal endomorphism of $P$ defined by $\omega$. Moreover, if one has a representation $\rho\colon G\to\Aut(V)$ the exterior covariant differential $d^\omega$ leaves invariant the $\cin(M)$-submodule $\Omega^\bullet_\mathrm{Bas}(P,V)\subset \Omega^\bullet(P,V)$ of basic forms (i.e $G$-invariant and horizontal forms). Since there is a natural isomorphism of $\mathbb Z$-graded $\cin(M)$-modules $$\varphi^\bullet_\rho\colon \Omega^\bullet_\mathrm{Bas}(P,V)\xrightarrow{\sim}\Omega^\bullet(M,P\times_\rho V),$$ with $\pi_\rho\colon P\times_\rho V\to M$ the associated vector bundle, it follows that $\omega$ induces a connection $\nabla^\rho$ on $P\times_\rho V$ such that its exterior covariant differential $d^{\nabla^\rho}$ makes commutative the following diagram 
\[\begin{tikzcd}
	\Omega^\bullet_\mathrm{Bas}(P,V) && \Omega^{\bullet+1}_\mathrm{Bas}(P,V) \\
	\\
	\Omega^\bullet(M,P\times_\rho V) && \Omega^{\bullet+1}(M,P\times_\rho V)
	\arrow["d^\omega", from=1-1, to=1-3]
	\arrow["\varphi^\bullet_\rho"', "\wr", from=1-1, to=3-1]
	\arrow["\varphi^\bullet_\rho", "\wr
	"', from=1-3, to=3-3]
	\arrow["d^{\nabla^\rho}", from=3-1, to=3-3]
\end{tikzcd}\] In particular, one has $$\nabla^\rho_D s= \varphi^0_\rho(i_{D^h}d^\omega ((\varphi^0_\rho)^{-1}(s))),\quad\quad D\in\X(M),s\in\Gamma(M,P\times_\rho V).$$

\begin{thm}\label{teo:beta-cov-constant}
    Let $\pi: (P,\widehat g)\longrightarrow (M,\overline g)$ be a Kaluza-Klein principal $G$-bundle such that  $\widehat g=\Psi^{-1}(\omega,\beta,\overline g)$. If $\nabla^{P}$ is the Levi-Civita connection of $\widehat{g}$, then the second fundamental form $\bm{T}$ of the fibers of the Riemannian submersion $\pi: (P,\widehat{g})\longrightarrow (M,\overline{g})$ vanishes (i.e., the fibers are totally geodesic) if and only if $\beta\in\cin(P,\mathrm{Met}(\mathfrak{g}))$ is $\omega$-covariantly constant, that is $$d^\omega\beta=0.$$   
Equivalently, if $\widehat g=\Psi^{-1}(\omega,g_\mathrm{ad},\overline g)$, then $\bm{T}=0$ if and only if the adjoint bundle $(\ad P,g_\ad,\nabla^\ad)$ is a Riemannian vector bundle, i.e $\nabla^\mathrm{ad}g_\mathrm{ad}=0$. That is, $\bm{T}=0$, if and only if the connection $\nabla^\mathrm{ad}$ of $\mathrm{ad}P$ is compatible with the Riemannian metric $g_\mathrm{ad}$ and thus for every $D\in\X(M),\boldsymbol{\xi}_1, \boldsymbol{\xi}_2\in\Gamma(M,\mathrm{ad}P)$ it holds $$0=(\nabla^\mathrm{ad}_D g_\mathrm{ad})(\boldsymbol{\xi}_1, \boldsymbol{\xi}_2)=D[g_\mathrm{ad}(\boldsymbol{\xi}_1, \boldsymbol{\xi}_2)]-g_\mathrm{ad}(\nabla^\mathrm{ad}_D\boldsymbol{\xi}_1, \boldsymbol{\xi}_2)-g_\mathrm{ad}(\boldsymbol{\xi}_1,\nabla^\mathrm{ad}_D\boldsymbol{\xi}_2).$$
\end{thm}

\begin{proof}
	We have seen in Lemma \ref{lem:T-vanishing}, that $\bm{T}=0$ if and only if for any $\omega$-horizontal vector field $H$ and any $\xi_1,\xi_2\in\mathfrak{g}$ one has $$H[\widehat g(\xi_1^*,\xi_2^*)]=0.$$
	Since $\widehat g(\xi_1^*,\xi_2^*)=\beta(\xi_1,\xi_2)$, the above equality is equivalent to $$H[\beta(\xi_1,\xi_2)]=0\Longleftrightarrow H\beta=0\Longleftrightarrow (d\beta)(H)=0\Longleftrightarrow  \mathcal H^*(d\beta)=d^\omega\beta=0.$$ This proves the first claim of the statement. For the second one, $\omega$ induces a connection $\nabla^\mathrm{ad}$ on the adjoint bundle $P\times_\mathrm{ad}\mathfrak{g}$. The compatibility of differential functors with the associated bundle functor gives in particular a vector bundle isomorphism  $$S^2((P\times_\rho\mathfrak{g})^*)\simeq P\times_{S^2(\rho^\vee)}S^2(\mathfrak{g}^*)$$ and the connection induced by $\omega$ on $S^2((P\times_\rho\mathfrak{g})^*)$ is just the connection naturally induced by $\nabla^\mathrm{ad}$, hence the proof is finished.
\end{proof}

\begin{Cor}\label{cor:subvariedad-metricas_KK-fibras-tot-geodesicas}
	  If we present the space of Kaluza-Klein metrics on a principal $G$-bundle $\pi: P\longrightarrow M$ as$$\mathrm{Met}_{KK}(P)=\mathcal{A}(P)\times \cin(P,\mathfrak{g})^G\times \mathrm{Met}(M),$$ then the subspace $\mathrm{Met}_{KK,gf}(P)$ of $\mathrm{Met}_{KK}(P)$ formed by those Kaluza-Klein metrics with totally geodesic fibers gets identified with $$\mathrm{Met}_{KK,tgf}(P)=\{(\omega,\beta,\overline g)\in\mathrm{Met}_{KK}(P)\colon d^\omega\beta=0\}.$$ Equivalently, if we consider the presentation $\mathrm{Met}_{KK}(P)=\mathcal{A}(P)\times \mathrm{Met}(\mathrm{ad}P)\times \mathrm{Met}(M),$ then the subspace $\mathrm{Met}_{KK,tgf}(P)$ gets identified with $$\mathrm{Met}_{KK,tgf}(P)=\{(\omega,g_\mathrm{ad},\overline g)\in\mathrm{Met}_{KK}(P)\colon \nabla^\mathrm{ad}g_\mathrm{ad}=0\}.$$
\end{Cor}

\begin{De}Let $\pi: (P,\widehat g)\longrightarrow (M,\overline g)$ be a Kaluza-Klein principal $G$-bundle such that  $\widehat g=\Psi^{-1}(\omega,\beta,\overline g)$ and let $\nabla^{P,{\mathcal{V}}}$, $\nabla^{\omega}$ be the connections on the vertical subbundle $\mathcal{V}P\to P$ induced, respectively, by the Levi-Civita connection of $(P,\widehat g)$ and the connection $\widehat\nabla^\omega$ on the trivial vector bundle $\mathfrak{g}_P=P\times\mathfrak{g}\to P$ given by 
   \begin{align}\label{intro}
\widehat\nabla^{\omega}\boldsymbol\xi=d\boldsymbol\xi+\frac{1}{2}[\widehat\omega,\boldsymbol\xi],\quad   \boldsymbol\xi\in C^{\infty}(P,\mathfrak{g});
   \end{align}that is, given $D\in\X(P)$ one has $$\nabla^\omega_D\boldsymbol{\xi}^*:=[\widehat \nabla^\omega_D\boldsymbol{\xi}]^*.$$ 
  The difference between the connections $\nabla^{P,\mathcal{V}}$ and $\nabla^{\omega}$ on $P$ is the tensor $\Theta\in T_2^1(P)$ 
  \begin{align}\label{tensor}
\bm{\Theta}:=\nabla^{P,\mathcal{V}}-\nabla^{\omega}\colon\X(P)\times\X^\mathcal{V}(P)\to \X^\mathcal{V}(P)
  \end{align}
  such that for arbitrary vector fields $D_{1}$, $D_{2}$ on $P$ we have
  \begin{align}\label{vert}
\bm{\Theta}(D_{1},D_{2})=\nabla^{P,\mathcal{V}}_{D_{1}}D_{2}^{\mathcal{V}}-\nabla^{\omega}_{D_{1}}D_{2}^{\mathcal{V}}\iff \nabla^{P,\mathcal{V}}_{D_{1}}D_{2}^{\mathcal{V}}=\nabla^{\omega}_{D_{1}}D_{2}^{\mathcal{V}}+\bm{\Theta}(D_{1},D_{2})
  \end{align}
  where $D_{2}^{\mathcal{V}}=\omega(D_{2})\in\mathfrak{X}^\mathcal{V}(P)$.\\
  \end{De}
  
  \begin{Pro}\label{pro:tensor-diferencia-conexiones}
  	Let $\pi: (P,\widehat g)\longrightarrow (M,\overline g)$ be a Kaluza-Klein principal $G$-bundle with  $\widehat g=\Psi^{-1}(\omega,\beta,\overline g)$. The difference tensor $\bm{\Theta}$ of the connections $\nabla^{P,\mathcal{V}}$, $\nabla^\omega$ is given by $$\bm{\Theta}(D,\boldsymbol{\xi}^*)=\frac{1}{2}\left\{\ad^*_{\boldsymbol{\xi}}(\widehat\omega(D))+\ad_{\widehat\omega(D)}^*(\boldsymbol{\xi})\right\}^*+\bm{T}(\boldsymbol{\xi}^*, D^\mathcal{H}),\quad D\in\X(P),\boldsymbol{\xi}\in\cin(P,\mathfrak{g}),$$ where for any $\zeta\in\mathfrak{g}$ the map $\ad_\zeta^*$ is the adjoint of $\ad_\zeta=[\zeta,-]\colon \mathfrak g\to\mathfrak g$ with respect to the metric $\beta$; that is, $$\beta(\ad^*_\zeta(\xi_1),\xi_2)=\beta(\xi_1,\ad_\zeta(\xi_2)),\quad\xi_1,\xi_2\in\mathfrak{g}.$$
  	Moreover, covariant derivation along basic vector fields with the connection $\nabla^{P,\mathcal V}$ leaves invariant the $\cin(M)$-submodule $\X^\mathcal{V}(P)^G\simeq\Gamma(M,\ad P)$ and thus it induces in a natural way a connection $\nabla^{\ad,\mathcal V}$ on the adjoint bundle $\ad P$ whose difference with the connection $\nabla^\ad$ induced on it by the principal connection $\omega$ is given by $$\nabla^{\ad,\mathcal V}-\nabla^\ad=\bm{T}^\ad$$ where $\bm{T}^\ad\in \Omega^1(M,\End(\ad P))$ is the $\End(\ad P)$-valued $1$-form naturally induced by $\bm{T}$. 
  	\end{Pro} 
\begin{proof}
By the Koszul formula, given $\zeta\in\mathfrak{g}$, one has \begin{align*}
 		\widehat g(\nabla^{P,\mathcal{V}}_D\boldsymbol{\xi}^*,\zeta^*)=\widehat g(\nabla^{P}_D\boldsymbol{\xi},\zeta^*)=\frac{1}{2}&\left\{\boldsymbol{\xi}^*\widehat g(D,\zeta^*)-\widehat g([\boldsymbol{\xi}^*,D],\zeta^*)-\widehat g(D,[\boldsymbol{\xi}^*,\zeta^*])+\right. \\&+\left.D\widehat g(\boldsymbol{\xi}^*,\zeta^*)-{\zeta}^*\widehat g(\boldsymbol{\xi}^*,D)-\widehat g(\boldsymbol{\xi}^*,[D,\zeta^*]) \right\}.
 	\end{align*}In particular, if $D=E^h$ with $E\in\X(M)$, then one has 
\begin{align*}
\widehat g(\nabla^{P}_{E^h}\boldsymbol{\xi}^*,\zeta^*)=\frac{1}{2}&\left\{\boldsymbol{\xi}^*\widehat g(E^h,\zeta^*)-\widehat g([\boldsymbol{\xi}^*,E^h],\zeta^*)-\widehat g(E^h,[\boldsymbol{\xi}^*,\zeta^*])+\right. \\&+\left.E^h\widehat g(\boldsymbol{\xi}^*,\zeta^*)-{\zeta}^*\widehat g(\boldsymbol{\xi}^*,E^h)-\widehat g(\boldsymbol{\xi}^*,[E^h,\zeta^*]) \right\}.	
\end{align*}
Bearing in mind that $\mathcal HP$ is $\widehat g$-orthogonal to $\mathcal VP$, it follows that $\widehat g(E^h,\zeta^*)=0$,  $\widehat g(E^h,[\boldsymbol{\xi}^*,\zeta^*])=0$, because $[\boldsymbol{\xi}^*,\zeta^*]$ is vertical, $\widehat g(\boldsymbol{\xi}^*,E^h)=0$ and $\widehat g(\boldsymbol{\xi}^*,[E^h,\zeta^*])=0$ as $[E^h,\zeta^*]=0$ because on the one hand it is horizontal due to \cite[Lemma p. 78]{kobayashi1963foundations}, whereas on the other it is vertical since it is $\pi$-related to $[E,0]=0$. After a straightforward computation we obtain
 \begin{align*}
\widehat g(\nabla^{P,\mathcal{V}}_{E^h}\boldsymbol{\xi}^*,\zeta^*)&=\frac{1}{2}\left\{-\widehat g([\boldsymbol{\xi}^*,E^h],\zeta^*)+E^h\widehat g(\boldsymbol{\xi}^*,\zeta^*) \right\}=\widehat g([E^h,\boldsymbol{\xi}^*],\zeta^*)+\widehat g(\bm{T}({\boldsymbol{\xi}^*},{E^h}) ,\zeta^*).	
\end{align*}Where we have used again  that $[E^h,\zeta^*]=0$ and the properties of $\bm{T}$ described in \cite{o1966fundamental}. On the other hand, if $\boldsymbol{\xi}=f^a\xi_a^*$ with $f^a\in\cin(P)$,  it holds $$[E^h,\boldsymbol{\xi}^*]=[E^h,f^a\xi_a^*]=E^h(f^a)\xi_a^*+f^a[E^h,\xi_a^*]=[E^h(\boldsymbol{\xi})]^*,$$ since as before $[E^h,\xi_a^*]=0$. Thus we get
 \begin{align*}
\widehat g(\nabla^{P,\mathcal{V}}_{E^h}\boldsymbol{\xi}^*,\zeta^*)=\widehat g([E^h(\boldsymbol{\xi})]^*+\bm{T}({\boldsymbol{\xi}^*},{E^h}) ,\zeta^*).
\end{align*}Taking into account that $\widehat g$ is non-degenerate and that $\zeta\in\mathfrak{g}$ is arbitrary, it follows that  $$\nabla^{P,\mathcal{V}}_{E^h}\boldsymbol{\xi}^*=[E^h(\boldsymbol{\xi})]^*+\bm{T}({\boldsymbol{\xi}^*},{E^h}).$$ On the other hand we have $$\nabla^\omega_{E^h}\boldsymbol{\xi}=\left[\widehat \nabla^\omega_{E^h}\boldsymbol{\xi}\right]^*=\left[E^h(\boldsymbol{\xi})+\frac{1}{2}[\widehat\omega(E^h),\boldsymbol{\xi}]\right]^*=\left[E^h(\boldsymbol{\xi})\right]^*,$$ because $\widehat\omega(E^h)=0$. Therefore, we have proved that $$\boldsymbol{\Theta}(E^h,\boldsymbol{\xi})=\nabla^{P,\mathcal{V}}_{E^h}\boldsymbol{\xi}-\nabla^\omega_{E^h}\boldsymbol{\xi}=\bm{T}({\boldsymbol{\xi}^*},{E^h}).$$ Since $\boldsymbol{\Theta}$ is $\cin(P)$-linear and $\X^\mathcal{H}(P)=\cin(P)\otimes_{\cin(M)}\X_{B}(P)\simeq\cin(P)\otimes_{\cin(M)}\X(M)$, it follows that for any horizontal vector field $H\in\X^\mathcal{H}(P)$ it holds  $$\boldsymbol{\Theta}(H,\boldsymbol{\xi})=\bm{T}({\boldsymbol{\xi}^*},H).$$

Now, given $\xi_1,\xi_2,\xi_3\in\mathfrak{g}$, since $L_{\xi^*_2}\widehat g=0$, by the Koszul formula we get\begin{align*}
 \widehat g(\nabla^{P,\mathcal{V}}_{\xi_1^*}\xi_2^*,\xi_3^*)&=\frac{1}{2}\left\{\xi_1^*\,\beta(\xi_2,\xi_3)-\xi_3^*\,\beta(\xi_2,\xi_1)-\beta(\xi_2,[\xi_1,\xi_3]) \right\}.
 \end{align*}
However, given $p\in P$ and remembering that $\beta_{pg}=\Ad_g^*\beta_p$, we have \begin{align*}
 \{\xi_1^*\,\beta(\xi_2,\xi_3)\}(p)&=\left.\frac{d}{dt}\right|_0\{\beta(\xi_2,\xi_3)\}(p\cdot\exp_{\xi_1}(t))=\left.\frac{d}{dt}\right|_0\left\{\beta_{p\cdot\exp_{\xi_1}(t)}(\xi_2,\xi_3)\right\}=\\&=\beta_p(\ad_{\xi_1}(\xi_2),\xi_3)+\beta_p(\xi_2,\ad_{\xi_1}(\xi_3))=\left\{\beta(\ad_{\xi_1}(\xi_2),\xi_3)+\beta(\xi_2,\ad_{\xi_1}(\xi_3))\right\}(p),
 \end{align*} and the analogous expression $\xi^*_3\,\beta(\xi^*_2,\xi^*_1)=\beta(\ad_{\xi_3}(\xi_2),\xi_1)+\beta(\xi_2,\ad_{\xi_3}(\xi_1)).$ Substituting these results above, after some computations we obtain    
\begin{align*}
 \widehat g(&\nabla^{P,\mathcal{V}}_{\xi_1^*}\xi_2^*,\xi_3^*)=\frac{1}{2}\left\{\beta(\ad_{\xi_1}(\xi_2),\xi_3)+\beta(\ad^*_{\xi_1}(\xi_2),\xi_3)+\beta(\ad^*_{\xi_2}(\xi_1),\xi_3)\right\}.
 \end{align*}
On the other hand, after some manipulations one has \begin{align*}
 \widehat g(\nabla^\omega_{\xi_1^*}\xi^*_2,\xi^*_3)&=	\widehat g\left([\widehat\nabla^\omega_{\xi_1^*}\xi_2]^*,\xi^*_3\right)=	\widehat g\left(\left\{i_{\xi_1^*}d\xi_2+\frac{1}{2}[\widehat\omega({\xi_1^*}),\xi_2]\right\}^*,\xi^*_3\right)=\beta\left(\frac{1}{2}\ad_{\xi_1}(\xi_2),\xi_3\right).
  \end{align*}
Therefore, we get
\begin{align*}
 \widehat g(&\nabla^{P,\mathcal{V}}_{\xi_1^*}\xi_2^*-\nabla^\omega_{\xi_1^*}\xi^*_2,\xi_3^*)= \frac{1}{2}\beta(\ad^*_{\xi_1}(\xi_2)+\ad^*_{\xi_2}(\xi_1),\xi_3)=\frac{1}{2}\widehat g\left(\left\{\ad^*_{\xi_1}(\xi_2)+\ad^*_{\xi_2}(\xi_1)\right\}^*,\xi_3^*\right).
 \end{align*} The same reasoning as in the previous case implies the equality $$\bm{\Theta}(\xi_1^*,\xi_2^*)=\nabla^{P,\mathcal{V}}_{\xi_1^*}\xi_2^*-\nabla^\omega_{\xi_1^*}\xi^*_2=\frac{1}{2}\left\{\ad^*_{\xi_1}(\xi_2)+\ad^*_{\xi_2}(\xi_1)\right\}^*.$$ Since $\bm{\Theta}$ is $\cin(P)$-linear, it follows that  for any $\boldsymbol{\xi_1}, \boldsymbol{\xi_2}\in\cin(P,\mathfrak{g})$ one has 
 $$\bm{\Theta}(\boldsymbol{\xi_1}^*,\boldsymbol{\xi_2}^*)=\nabla^{P,\mathcal{V}}_{\boldsymbol{\xi_1}^*}\boldsymbol{\xi_2}^*-\nabla^\omega_{\boldsymbol{\xi_1}^*}\boldsymbol{\xi}^*_2=\frac{1}{2}\left\{\ad^*_{\boldsymbol{\xi_1}}(\boldsymbol{\xi_2})+\ad^*_{\boldsymbol{\xi_2}}(\boldsymbol{\xi_1})\right\}^*.$$
 Therefore, given $D\in\X(P)$ and $\xi\in\mathfrak{g}$  we  finally obtain
\begin{align*}
	\bm{\Theta}(D,\xi^*)&=\bm{\Theta}(D^\mathcal{V}+D^\mathcal{H},\xi^*)=\bm{\Theta}(\omega(D)+D^\mathcal{H},\xi^*)=\bm{\Theta}([\widehat\omega(D)]^* +D^\mathcal{H},\xi^*)=\\&=\bm{\Theta}([\widehat\omega(D)]^*,\xi^*)+\bm{\Theta}(D^\mathcal{H},\xi^*)=\frac{1}{2}\left\{\ad_{\widehat\omega(D)}^*(\xi)+\ad_\xi^*(\widehat\omega(D))\right\}^*+\bm{T}(\xi^*,D^\mathcal{H}),
\end{align*} as claimed in the first part of the statement.

For the second claim, let us recall that there is an isomorphism of $\cin(M)$-modules between  the space of sections of the adjoint bundle and the space of $G$-invariant vertical vector fields $$\widetilde{(-)}\colon \Gamma(M,\mathrm{ad} P)=\cin(P,\mathfrak g)^G\xrightarrow{\sim}\X^{\mathcal V}(P)^G$$ that associates to $\boldsymbol{\nu}\in \cin(P,\mathfrak g)^G$ the $G$-invariant vector field $\widetilde{\boldsymbol{\nu}}\in \X^{\mathcal V}(P)^G$ given by $\widetilde{\boldsymbol{\nu}}_p:=[\boldsymbol{\nu}(p)]^*_p$. That is, $\widetilde{(-)}$ is just the restriction to $\cin(P,\mathfrak g)^G$ of the isomorphism $(-)^*\colon \cin(P,\mathfrak{g})\xrightarrow{\sim}\X^\mathcal{V}(P)$ that we have considered above.
 Now recall that Levi-Civita connections behave naturally under isometric diffeomorphisms and therefore are left invariant under isometries. In particular, the Levi-Civita connection  $\nabla^P$ of $(P,\widehat g)$ is preserved under isometries.  Since $\widehat g$ is $G$-invariant, it follows that for any $g\in G$ one has $R_g\cdot\nabla^P=\nabla^P$. Now, we define the map $$\widetilde\nabla^{\ad,\mathcal V}\colon \X_{\mathrm{Bas}}(P)\times \X^\mathcal{V}(P)^G\to \X^\mathcal{V}(P)^G$$ such that $$\widetilde\nabla^{\ad,\mathcal V}_BV:=\nabla^{P,\mathcal V}_BV=\omega\left(\nabla^P_BV\right),\quad\quad B\in \X_{\mathrm{Bas}}(P), V\in \X^\mathcal{V}(P)^G.$$
 The definition is correct; that is, $\widetilde\nabla^{\ad,\mathcal V}_BV$ is a $G$-invariant vertical vector field. Indeed, using the $G$-equivariance of $\omega$ and the $G$-invariance of $\nabla^P$,  given $g\in G$,  $B\in\X_{\mathrm{Bas}}(P)$, $V\in\X^\mathcal{V}(P)^G$ we have $$R_g\cdot\left(\widetilde\nabla^{\ad,\mathcal V}_BV\right)=R_g\cdot\left(\omega\left(\nabla^P_BV\right)\right)=\omega\left(R_g\cdot\left(\nabla^P_BV\right)\right)=\omega\left(\nabla^P_{R_g\cdot B}(R_g\cdot V)\right)=\omega\left(\nabla^P_BV\right)=\widetilde\nabla^{\ad,\mathcal V}_BV.$$ Now, using the isomorphism of $\cin(M)$-modules $\widetilde{(-)}\colon \Gamma(M,\mathrm{ad} P)\xrightarrow{\sim}\X^{\mathcal V}(P)^G$  we define a map $$\nabla^{\ad,\mathcal{V}}\colon\X(M)\times\Gamma(M,\ad P)\to\Gamma(M,\ad P)$$ such that $$\widetilde{\nabla^{\ad,\mathcal{V}}_E{\boldsymbol{\nu}}}=\nabla^{P,\mathcal V}_{E^h}\widetilde{\boldsymbol{\nu}},\quad\quad\quad E\in\X(M), \boldsymbol{\nu}\in\Gamma(M,\ad P).$$ One easily checks that $\nabla^{\ad,\mathcal{V}}$ is a connection on $\ad P$. Based on the results we have shown previously to demonstrate the first statement, it follows that
$$\widetilde{\nabla^{\ad,\mathcal{V}}_E{\boldsymbol{\nu}}}=\nabla^{P,\mathcal{V}}_{E^h}\widetilde{\boldsymbol{\nu}}=\nabla^{P,\mathcal{V}}_{E^h}\boldsymbol{\nu}^*=[E^h(\boldsymbol{\nu})]^*+\bm{T}({\boldsymbol{\nu}^*},{E^h})=\widetilde{i_{E^h}d^\omega \boldsymbol{\nu}}+\bm{T}(\widetilde{\boldsymbol{\nu}},E^h).$$ Since $$\bm{T}(\widetilde{\boldsymbol{\nu}},E^h)=\omega\left(\nabla^{P}_{\widetilde{\boldsymbol{\nu}}}E^h\right)$$ a similar reasoning as above, bearing in mind now the properties of $\bm{T}$ proved in \cite{o1966fundamental}, shows that $\bm{T}(\widetilde{\boldsymbol{\nu}},E^h)\in\X^\mathcal{V}(P)^G$ is a $G$-invariant vertical vector field, and therefore there is a section $\bm{T}^\ad(E)(\boldsymbol{\nu})\in\Gamma(M,\ad P)$ such that $\bm{T}(\widetilde{\boldsymbol{\nu}},E^h)=\widetilde{\bm{T}^\ad(E)({\boldsymbol{\nu}})}$. Hence, we have $$\widetilde{\nabla^{\ad,\mathcal{V}}_E{\boldsymbol{\nu}}} =\widetilde{\nabla^{ad}\boldsymbol{\nu}}+\widetilde{\bm{T}^\ad(E)({\boldsymbol{\nu}})}.$$ This implies the second claim and the proof is finished.
\end{proof}

\begin{Cor}\label{cor:difference-of-connections}
	Let $\pi: (P,\widehat g)\longrightarrow (M,\overline g)$ be a Kaluza-Klein principal $G$-bundle with  $\widehat g=\Psi^{-1}(\omega,\beta,\overline g)$. The two connections $\nabla^{P,\mathcal{V}}$, $\nabla^\omega$ of the vertical bundle $\mathcal VP\to P$  are equal, if and only if the following conditions are satisfied:
	\begin{enumerate}
	\item $\beta$ is $\ad$-invariant; that is, $$\beta(\ad_\xi(\xi_1),\xi_2)+\beta(\xi_1,\ad_\xi(\xi_2))=0,\quad\xi,\xi_1,\xi_2\in\mathfrak{g}.$$
	\item $\beta$ is $\omega$-covariantly constant; that is $$d^\omega\beta=0\Longleftrightarrow \bm{T}=0.$$
	\end{enumerate}
	
Moreover, the connections $\nabla^{\ad,\mathcal V},\nabla^\ad$ on the adjoint bundle $\ad P\to M$ are equal if and only if $\bm{T}=0$.
\end{Cor}

\begin{proof}
We have $$	\nabla^{P,\mathcal{V}}=\nabla^\omega\quad \Longleftrightarrow\quad \bm{\Theta}=0.$$ For any $\xi,\xi_1,\xi_2\in\mathfrak{g}$ it holds \begin{align*}
 	\widehat g(\bm{\Theta}(\xi_1^*,\xi_2^*),\xi^*)&=\frac{1}{2}\widehat g\left(\left\{\ad^*_{\xi_1}(\xi_2)+\ad^*_{\xi_2}(\xi_1)\right\}^*,\xi^*\right)=\frac{1}{2}\beta(\ad^*_{\xi_1}(\xi_2)+\ad^*_{\xi_2}(\xi_1),\xi)=\\&=\frac{1}{2}\left\{\beta(\xi_2,\ad_{\xi_1}(\xi))+\beta(\xi_1,\ad_{\xi_2}(\xi))\right\}=-\frac{1}{2}\left\{\beta(\ad_{\xi}(\xi_1),\xi_2)+\beta(\xi_1,\ad_{\xi}(\xi_2))\right\}.
 \end{align*} Similarly, given $D\in\X(P)$ it holds $$\bm{\Theta}(D^\mathcal{H},\xi^*)=\bm{T}(\xi^*,D^\mathcal{H}).$$Moreover, given $\zeta\in\mathfrak{g}$ since $\bm{T}(\xi^*,-)$ is $\widehat g$-skew-symmetric it holds $$\widehat g(\bm{T}(\xi^*,\zeta^*),D^\mathcal{H})=-\widehat g(\zeta^*,\bm{T}(\xi^*,D^\mathcal{H})).$$Therefore, since $\widehat g$ is nondegenerate it follows that $\bm{\Theta}=0$ if and only if $\beta$ is $\ad$-invariant and $\bm{T}=0$, but this last conditions is, thanks to Theorem \ref{teo:beta-cov-constant}, equivalent to $d^\omega\beta=0$. Whence, the proof of the first claim is finished. The second one follows in a similar way.
\end{proof}

   Proceeding in a similar way as in Theorems \ref{teo:horizontal-component-tension-field}, \ref{teo:vertical-tension-field}, we get the following results.   
    \begin{thm}\label{teo:tension-vertical-modificada}
         Let $\pi: (P,\widehat g)\longrightarrow (M,\overline g)$ be a Kaluza-Klein principal $G$-bundle with $\widehat g=\Phi^{-1}(\omega,\beta,\overline g)$. If $\widetilde{\Phi}: N\longrightarrow P$ is  a smooth map, then the vertical component of its tension field is given by
        \begin{align}
[\tau(\widetilde{\Phi})]^{\mathcal{V}}&= -\delta^{(\nabla^{N}, \widetilde\Phi^*{\nabla}^{\omega})}(\widetilde{\Phi}^{*}\omega)+\omega\left(\Tr_g(\widetilde{\Phi}^{*}(\bm{\Theta}+\bm{T}))\right)
        \end{align}
         where  $\widetilde{\Phi}^{*}\omega\in\Omega^{1}(N,\widetilde{\Phi}^{*}(\mathcal{V}P))$ is the pullback of the connection $\omega\in\Omega^{1}(P,\mathcal{V}P)$  and $\delta^{(\nabla^{N},\widetilde\Phi^*\nabla^\omega)}$ is the codifferential operator defined on the vector bundle $T^{*}N\otimes\widetilde{\Phi}^{*}(\mathcal{V}P)$ with respect to the connections $\nabla^N,\widetilde\Phi^*\nabla^{\omega}$.
   \end{thm}

    \begin{Rem}
        Theorem \ref{teo:tension-vertical-modificada} generalizes the result obtained by H. Manabe in \cite[Theorem 2.4]{manabe1992pluriharmonic}.
    \end{Rem}
    \begin{Cor}\label{7}
        Let $\pi: (P,\widehat g)\longrightarrow (M,\overline g)$ be a Kaluza-Klein principal $G$-bundle with $\widehat g=\Psi^{-1}(\omega,\beta,\overline g)$. A smooth map $\widetilde{\Phi}: N\longrightarrow P$ is vertically harmonic (i.e $[\tau(\widetilde{\Phi})]^{\mathcal{V}}=0$) if and only if 
        \begin{align}
\delta^{(\nabla^{N}, \widetilde\Phi^*{\nabla}^{\omega})}(\widetilde{\Phi}^{*}\omega)=\omega\left(\Tr_g(\widetilde{\Phi}^{*}(\bm{\Theta}+\bm{T}))\right).
        \end{align}
    \end{Cor}
    \begin{Rem}
        Corollary \ref{7} is the generalization of \cite[Theorem B (i)]{manabe1992pluriharmonic}.
    \end{Rem}
    Combining Theorems \ref{teo:horizontal-component-tension-field}, \ref{teo:horizontal-hamonic-curvature-modified-metric}  and Corollary \ref{7} we get another version of Corollary \ref{cor:harmonic-equations}. 
\begin{Cor}\label{cor:equivalent-harmonic-equations} 
    Let $\pi: (P,\widehat g)\longrightarrow (M,\overline g)$ be a Kaluza-Klein principal $G$-bundle with $\widehat g=\Psi^{-1}(\omega,\beta,\overline g)$, let $\widetilde{\Phi}: N\longrightarrow P$ be a smooth map and consider the composition $\Phi=\pi\circ\widetilde{\Phi}$. Then $\widetilde{\Phi}$ is harmonic if and only if 
\begin{align}\label{fundamen}
    \tau(\Phi) =-\pi_{*}\big[\Tr_g\widetilde{\Phi}^{*}(2\bm{A}+\bm{T})\big]\quad and \quad \delta^{(\nabla^{N}, \widetilde\Phi^*{\nabla}^{\omega})}(\widetilde{\Phi}^{*}\omega) =\omega\left(\Tr_g(\widetilde{\Phi}^{*}(\bm{\Theta}+\bm{T}))\right).
\end{align} Equivalently, $\widetilde{\Phi}$ is harmonic if and only if \begin{align}\label{second-fundamental}
\tau(\Phi) &=-\pi_{*}[\Tr_{g}(\widetilde{\Phi}^{*}\widehat{g}_{\mathscr{F}^{\omega}})]-\pi_{*}[\Tr_{g}(\widetilde{\Phi}^{*}\bm{T})]\quad \text{and} \quad \delta^{(\nabla^{N}, \widetilde\Phi^*{\nabla}^{\omega})}(\widetilde{\Phi}^{*}\omega)=\omega[\Tr_{g}(\widetilde{\Phi}^{*}(\bm{\Theta+T}))].
      \end{align}   
\end{Cor}
    \begin{Rem}
       Let $\pi: (P,\widehat g)\longrightarrow (M,\overline g)$ be a Kaluza-Klein principal $G$-bundle with $\widehat g=\Psi^{-1}(\omega,\beta,\overline g)$. If $\widetilde{\Phi}: N\longrightarrow P$ is a vertically harmonic map, then in general, $\widetilde{\Phi}^{*}\omega$ does not satisfies the Hodge gauge equation $\delta^{(\nabla^N,\widetilde\Phi^*\nabla^\omega)}(\widetilde{\Phi}^{*}\omega)=0$. However, if the fibers of $P$ are totally geodesics $(\bm{T}=0)$ and the metric $\beta$ is $\ad$-invariant, then thanks to Corollary \ref{cor:difference-of-connections} the connections ${\nabla}^{P, \mathcal{V}}$ and ${\nabla}^{\omega}$ are equal,  hence $\bm{\Theta}$ vanishes and therefore by Corollary \ref{7} the Hodge gauge equation is automatically satisfied.
    \end{Rem}

\section[Examples of Kaluza-Klein harmonic maps and magnetic maps]{Examples of Kaluza-Klein harmonic maps and generalized magnetic maps}
In this section we introduce several instances of Kaluza-Klein harmonic maps. We start by considering the natural generalization of the classical example of geodesic curves.
\begin{Ex}[Generalized magnetic curves]\label{ex:magnetic-curves}
    We consider an open subset $N=I\subset\mathbb{R}$ and  $g$ is the restriction to $I$ of the standard euclidean metric of $\mathbb R$. Let $\pi: (P,\widehat{g})\longrightarrow (M,\overline{g})$ be a Kaluza-Klein principal $G$-bundle. Let $\widetilde{\gamma}: (I,g)\longrightarrow (P,\widehat{g})$ be a geodesic curve on $P$, that is, $\nabla^{P}_{\dot{\widetilde{\gamma}}}\dot{\widetilde{\gamma}}=0$ with $\Dim P=m+d$. Therefore, the curve  $\gamma=\pi\circ\widetilde{\gamma}$ induced on $M$ is a generalized  magnetic map.   

The Euler-Lagrange equations satisfied by $\tilde\gamma$ are:
\begin{align}\label{tenshorver}
    \tau(\gamma)=\nabla^M_{\dot{\gamma}}\dot{\gamma} &=-\pi_{*}[\Tr_{g}(\widetilde{\gamma}^{*}\widehat{g}_{\mathscr{F}^{\omega}})]-\pi_{*}[\Tr_{g}(\widetilde{\gamma}^{*}\bm{T})],\quad  \quad \delta^{(\nabla^{N}, \widetilde\gamma^*{\nabla}^{P,\mathcal{V}})}(\widetilde{\gamma}^{*}\omega)=\omega[\Tr_{g}(\widetilde{\gamma}^{*}\bm{T})].
\end{align}

  Moreover, the Lorentz strength endomorphism can be written as a finite sum $$\mathscr{F}_\omega=\sum_{a=1}^k\mathscr{F}_{\omega}^a\otimes \widetilde{\boldsymbol{\xi}}_a,\quad \mathscr{F}_{\omega}^a\in T_1^1(M),\widetilde{\boldsymbol{\xi}}_a\in\X^\mathcal{V}(P)^G,$$ and the Euler-Lagrange equations are equivalent to  
\begin{align}\label{funda}
\nabla^M_{\dot{\gamma}}\dot{\gamma}=-\sum_{a=1}^{k}{\widetilde\gamma}^{\#}\widehat{g}({\widetilde\gamma}^{\#}\widetilde{\boldsymbol{\xi}}_a,\dot{\widetilde{\gamma}})\mathscr{F}^{a}_{\omega}(\dot{\gamma})-\pi_{*}[\bm{T}(\dot{\widetilde{\gamma}},\dot{\widetilde{\gamma}})],\quad  \quad \delta^{(\nabla^{N}, \widetilde\gamma^*{\nabla}^{P,\mathcal{V}})}(\widetilde{\gamma}^{*}\omega)=\omega[\Tr_{g}(\widetilde{\gamma}^{*}\bm{T})].
\end{align}
Moreover, if $\widetilde{\boldsymbol{\xi}}_a$ is a $\widehat g$-Killing vector field for each $a$ (this is the case, for instance,  if $G$ has trivial adjoint group by Proposition \ref{pro:Lorentz-strength}), then it is a Jacobi field along the geodesic $\widetilde\gamma$, whence   $$\bm{\kappa}_a:=-{\widetilde\gamma}^{\#}\widehat g({\widetilde\gamma}^{\#}\widetilde{\boldsymbol{\xi}}_a,\dot{\widetilde\gamma})$$ is constant and  therefore equations (\ref{funda}) become 
\begin{align}\label{magnegeo}
\nabla^M_{\dot{\gamma}}\dot{\gamma}=\sum_{a=1}^k\bm{\kappa}_a\cdot\mathscr{F}^a_{\omega}(\dot{\gamma})-\pi_{*}[\bm{T}(\dot{\widetilde{\gamma}},\dot{\widetilde{\gamma}})],\quad \ \quad \delta^{(\nabla^{N}, \widetilde\gamma^*{\nabla}^{P,\mathcal{V}})}(\widetilde{\gamma}^{*}\omega)=\omega[\bm{T}(\dot{\widetilde{\gamma}},\dot{\widetilde{\gamma}})].
\end{align}
The first of them is called the generalized-magnetic-like equation or the generalized-Lorentz-like equation. It generalizes the result obtained by R. Kerner, see \cite{kerner1968generalization}, where in this case $\bm{\kappa}_a$ is the  $a$-color charge to mass ratio sensed by the particle in its motion in $M$, $\mathscr{F}^a_{\omega}$ represents the $a$-color  Lorentz force-like associated with the metric $\overline{g}$ and the magnetic-like field $\Omega_\omega^a$ is the $a$-th component of the field strength $2$-form $\Omega_\omega$ of the gauge field associated to the Kaluza-Klein principal $G$-bundle. If all the charge to mass ratios vanish, $\bm{\kappa}_a=0$, then the generalized magnetic map $\gamma$ is uncharged.
 The system of equations (\ref{magnegeo}) is a generalization of Wong's equations, see the equivalences proved in \cite{Montgomery} for different ways of  expressing these equations.\end{Ex}
\begin{Rem}
Equations (\ref{magnegeo}) show a deviation of particle motion from geodesics due to the magnetic field and the geometric force $\bm{T}$ induced by the geometry of the internal space of the gauge field.
\end{Rem}
When the fibers of the Kaluza-Klein principal $G$-bundle $\pi\colon P\to M$  are totally geodesic, one gets
\begin{align}
\nabla^M_{\dot{\gamma}}\dot{\gamma}=\sum_{a=1}^d\bm{\kappa}_a\cdot\mathscr{F}^a_{\omega}(\dot{\gamma}),\quad  \quad \delta^{(\nabla^{N}, \widetilde\gamma^*{\nabla}^{P,\mathcal{V}})}(\widetilde{\gamma}^{*}\omega)=0.
\end{align}

If the magnetic-like field vanishes, $\Omega_{\omega}=0$, then equations (\ref{magnegeo}) become
\begin{align}\label{fibra}
\nabla^M_{\dot{\gamma}}\dot{\gamma}=-\pi_{*}[\bm{T}(\dot{\widetilde{\gamma}},\dot{\widetilde{\gamma}})],\quad  \quad \delta^{(\nabla^{N}, \widetilde\gamma^*{\nabla}^{P,\mathcal{V}})}(\widetilde{\gamma}^{*}\omega)=\omega[\bm{T}(\dot{\widetilde{\gamma}},\dot{\widetilde{\gamma}})].
\end{align}
\begin{Rem}
        Equations (\ref{fibra}) show how the motion of a particle is affected by the shape and structure of the internal space of the gauge field, even if its field strength vanishes $\Omega_{\omega}=0$ or more generally if $\gamma$ is uncharged. In this case, the tensor $\bm{T}$ represents the geometric force arising from the shape of internal space.
    \end{Rem}
\begin{Ex}
    If $N=P$, $g=\widehat{g}$, then the identity map $Id_{P}: (P,\widehat{g})\longrightarrow (P,\widehat{g})$ is harmonic since its second fundamental form vanishes everywhere (i.e., $\mathrm{II}_{Id_{P}}=0$, that is, $Id_{p}$ is totally geodesic). Therefore, any Kaluza-Klein bundle $\pi\colon (P,\widehat{g})\to (M,\overline g)$ is an uncharged generalized magnetic map. The Wong's equations satisfied by $\Id_P$ are 
    \begin{align*} 
        \tau(\pi) &=-d\cdot \pi_*(\bm{H}^\mathcal{V}),\quad \quad \delta^{(\nabla^{N}, \nabla^{\omega})}\omega=\bm{V},
    \end{align*} with $\bm{H}^\mathcal{V}:=\frac{1}{d}\Tr_{\widehat g}\bm{T}^\mathcal{V,V}$ the vertical mean curvature vector field of $\pi$ and $\bm{V}=\left[\sum_{a=1}^d\ad^*_{\boldsymbol{\xi}_a}(\boldsymbol{\xi}_a)\right]^*$, where $\{\boldsymbol{\xi}_a^*\}_{a=1}^d$ is a $\widehat g$-orthonormal frame of $\mathcal VP$. To show that the Euler-Lagrange equations are the ones written above one uses Corollary \ref{cor:equivalent-harmonic-equations}  and the following facts: 1) The curvature modified metric  $\widehat{g}_{\mathscr{F}^\omega}(D,D)=\langle D,\mathscr{F}^\omega\rangle_{\mathcal VP}(D)$ is zero if either $D$ is horizontal or vertical, therefore, its trace $\Tr_{\widehat{g}}\widehat{g}_{\mathscr{F}^\omega}=0$ vanishes and hence $\pi$ is uncharged.  2) Since we are projecting to $M$, we get $\pi_*(\Tr_{\widehat{g}}\bm{T})=\pi_*(\Tr_{\widehat{g}}\bm{T}^\mathcal{V,V})$. 3) Finally, one has $\omega\left(\Tr_{\widehat{g}}(\bm{\Theta}+\bm{T})\right)=\Tr_{\widehat{g}}\left(\omega\left((\bm{\Theta}+\bm{T}\right)\right)=\Tr_{\widehat{g}}\left(\omega(\bm{\Theta})+\bm{T}^\mathcal{V,H}\right)=\Tr_{\widehat{g}}\left(\omega(\bm{\Theta})\right)=\bm{V}$ where the last equality follows immediately from Proposition \ref{pro:tensor-diferencia-conexiones}.
    
\end{Ex}
Now we describe two one-parameter families of generalized magnetic maps  based on spherical harmonic immersions; that is, harmonic immersions from a product of spheres into another sphere.
\begin{Ex}
    We construct a family of harmonic map from the two-dimensional torus into   the $U(1)$-Hopf fibration $\pi\colon S^{3}\longrightarrow S^{2}$ equipped with a Kaluza-Klein metric. Their images are called twisted Clifford Tori.\\
    \textbf{Step 1} Differential geometry of the $U(1)$-Hopf fibration:\\
The three dimensional sphere $S^3$ can be understood as the unit quaternions
    \begin{align*}
        S^{3} &=\{q\in\mathbb{H}: g_{euc}(q,q)=\|q\|^{2}=\overline{q}\cdot q=q\cdot\overline{q}=1\}.
    \end{align*}
The non-zero quaternions $\mathbb{H}^*$ form a Lie group of dimension $4$. Its Lie algebra $\mathscr{H}=\mathbb{H}$ is the space of quaternions endowed with Lie bracket given by the commutator and it gets identified with the Lie algebra $\mathfrak{X}(\mathbb{H}^*)^{L}$ of left invariant vector fields on $\mathbb{H}^*$, which has as its basis
\begin{align*}
    B_{4} =\{1^{*},i^{*},j^{*}.k^{*}\}
\end{align*}
where given $\xi\in\mathbb{H}, \xi^{*}\in\mathfrak{X}(\mathbb{H}^{*})$ denotes the fundamental vector field with respect to the right action
\begin{align*}
R:\mathbb{H}^*\times\mathbb{H}^*\longrightarrow \mathbb{H}^*
\end{align*}
of $\mathbb{H}^{*}$ on itself. Let us denote by $\widehat{g}$ the metric induced on $S^3$ by the euclidean metric $g_{\text{euc}}$ of $\mathbb{H}$ so that one has
\begin{align*}
\widehat{g}_q(D_q[\xi_1],D_q[\xi_2]) &=g_{\text{euc}}(\xi_1,\xi_2)
 =\frac{1}{2}(\overline{\xi}_1\cdot\xi_2+\overline{\xi}_{2}\cdot\xi_1)=\frac{1}{2}(\xi_1\cdot\overline{\xi}_{2}+\xi_2\cdot\overline{\xi}_1).
 \end{align*}

Through straightforward calculations, the following results are obtained.
\begin{Lem}
    The metric $g_{\text{euc}}$ in $\mathbb{H}$  is invariant under the actions of $S^3$ on the left $L:S^3\times\mathbb{H}\longrightarrow\mathbb{H}$ and on the right $R: \mathbb{H}\times S^{3}\longrightarrow\mathbb{H}$,  i.e., for any $q\in S^3$ one has
    \begin{align*}
        L^*_{q}g_{\text{euc}}=g_{\text{euc}}, \quad \quad R^*_{q}g_{\text{euc}}=g_{\text{euc}}.
    \end{align*}
    \end{Lem}
\begin{Cor}
The metric $\widehat{g}$ induced on $S^3$ by $g_{\text{euc}}$ is bi-invariant. Therefore $(S^3,R,\widehat g)$, where $R$ is the restriction of the right action $R$ to $U(1)\subset  S^3$, is a Kaluza-Klein principal $U(1)$-bundle. Furthermore,  $B_{S^3}=\{D_1=i^*, D_2=j^*,D_3=k^*\}$ is a $\widehat{g}$-orthonormal frame of $S^3$ with dual frame  $B^{*}_{S^3}=\{\omega^1, \omega^2, \omega^3\}$.
\end{Cor} 

\begin{thm}
    The vertical and horizontal sub-bundles of the Kaluza-Klein principal $U(1)$-bundle $(S^3,R,\widehat g)$ are given by
$$V(S^3)=<i^*>, \quad H(S^3)=<j^*,k^*>.$$
The principal connection defined by the Kaluza-Klein metric $\widehat{g}$ is given by $\omega=\omega^1\otimes i^*.$

The map $\beta\in\cin(S^3,\mathrm{Met}(\mathfrak{u}(1)))$ defined by $\widehat g$ is constant with $\beta(i,i)=1$ and $\beta$ is $\Ad$-invariant.

    The curvature of the connection $\omega$ is
$\Omega^\omega=d\omega^1\otimes i^*=-2\omega^2\wedge\omega^3\otimes i^*.$

The Lorentz endomorphism of $(S^3,\widehat{g})$ is
$\mathscr{F}^{\omega}=-2(\omega^2\otimes D_3-\omega^3\otimes D_2)\otimes i^*.$

The curvature modified metric of $(S^3,\widehat{g})$ is
$$\widehat{g}_{\mathscr{F}^\omega}=-[\omega^1\otimes\omega^2+\omega^2\otimes\omega^1]\otimes D_3+[\omega^1\otimes\omega^3+\omega^3\otimes\omega^1]\otimes D_2=\omega^1\cdot\omega^3\otimes D_2-\omega^1\cdot\omega^2\otimes D_3.$$
\end{thm}

 \textbf{Step 2} For any $\alpha\in\R$, let $\widetilde\Phi_\alpha: S^{1}\times S^{1}\longrightarrow S^{3}\subset\mathbb{R}^{4}$ be the map given by $\widetilde\Phi_\alpha(z_1,z_2)=\cos\alpha\, z_1+\sin\alpha\,z_2\cdot j$. There is a commutative diagram  \[\begin{tikzcd}
	S^{1}\times S^{1} && S^{3}\subset\mathbb{R}^{4} \\
	\\
	&& {S^{2}\subset\mathbb{R}^{3}}
	\arrow[" \widetilde\Phi_\alpha", from=1-1, to=1-3]
	\arrow["\Phi_\alpha"', from=1-1, to=3-3]
	\arrow["\pi", from=1-3, to=3-3]
\end{tikzcd}\] where  $\Phi_\alpha: S^{1}\times S^{1}\longrightarrow S^{2}\subset\mathbb{R}^{3}$ is the composition $\Phi_\alpha=\pi\circ\widetilde\Phi_\alpha$. 
Using \cite[Proposition 3.3.17]{baird2003harmonic} one straightforwardly proves that if $\alpha\neq\frac{\pi}{2}\mathbb Z$ then $\widetilde\Phi_\alpha$ is an spherical  harmonic immersion that we call the $\alpha$-twisted Clifford torus; that is, $\tau(\widetilde\Phi_\alpha)=0$, see \cite{dragomir1989sottovarieta} for  details for the standard Clifford torus (i.e. for $\alpha=\frac{\pi}{4}$). Thus, $\Phi_\alpha$ is a generalized magnetic map.

\textbf{Step 3} Since $\beta$ is constant and $\ad$-invariant, by Corollary \ref{cor:difference-of-connections} we get $\bm{T}=0$, $\bm{\Theta}=0$. By Corollary \ref{cor:equivalent-harmonic-equations} it follows that the Wong's equations satisfied by $\widetilde\Phi$ are 
\begin{align*}
    \tau(\Phi) =-\pi_{*}\left(\Tr_{g}[\widetilde\Phi_\alpha^*\widehat{g}_{\mathscr{F}^{\omega}}]\right),\quad \quad \delta^{(\nabla^{S^1\times S^1},\widetilde\Phi_\alpha^*{\nabla}^{\omega})}(\widetilde\Phi_\alpha^{*}\omega) =0.
\end{align*}
Therefore
\begin{align*}
    \tau(\Phi_\alpha) =-\pi_{*}\left(\Tr_{g}[\widetilde\Phi_\alpha^*(-\omega^1\cdot\omega^2\otimes D_3+\omega^1\cdot\omega^3\otimes D_2)]\right),\quad\quad \delta^{(\nabla^{S^1\times S^1},\widetilde\Phi_\alpha^*{\nabla}^{\omega})}(\widetilde\Phi_\alpha^{*}(\omega^1\otimes i^*)) =0.
    \end{align*} By \cite[Proposition 3.3.17]{baird2003harmonic}, a straightforward computation shows that $\Phi_\alpha$ is uncharged (i.e.  $\tau(\Phi_\alpha)=0$) if and only if $\alpha\in\frac{\pi}{4}\cdot\mathbb Z$; that is, if and only if $\widetilde \Phi_\alpha$ is the standard Clifford  torus.
  \end{Ex}
 \begin{Ex}
    We finish with another family based on spherical harmonic immersions from  $S^{3}\times S^{3}$ into   the total space of the $SU(2)$-Hopf fibration $S^{7}\longrightarrow S^{4}$ equipped with a Kaluza-Klein metric. We use the same notations as in the previous example.\\
    \textbf{Step 1} 
We study the differential geometry of the $SU(2)$-Hopf fibration:\\ 
The sphere $S^7$ can be understood as the unit octonions thought as the Cayley-Dickson double  of quaternions
\begin{align*}
    S^{7}=\{\bm{o}=(q_1,q_2)\in\mathbb{H}\times\mathbb{H}: g_{\text{euc}}(\bm{o},\bm{o})=\|q_1\|^2+\|q_2\|^2=1\}.
\end{align*}
\begin{Lem}
    For each $\bm{o}\in\mathbb{O}_{1}= S^7$, the metric $g_{euc}$ in $\mathbb{O}$ is invariant under the diffeomorphisms $L_{\bm{o}}, R_{\bm{o}}$, i.e., for any $\bm{o}\in S^7$ it holds 
\begin{align*}
    L_{\bm{o}}^*g_{euc}=g_{euc}, \quad R_{\bm{o}}^*g_{euc}=g_{euc}.
\end{align*}
The metric $\widehat{g}$ induced on $S^7$ by $g_{\text{euc}}$ is bi-invariant under the natural left and right actions of the group $SU(2)=\mathbb H_1$ of unit quaternions. Therefore $(S^7,R,\widehat g)$, where $R$ is the right action $R$ of $SU(2)$, is a Kaluza-Klein principal $SU(2)$-bundle. Furthermore,  $B_{S^7}=\{D_1=\bm{i}^*, D_2=\bm{j}^*, D_3=\bm{k}^*, D_4=\bm{l}^*, D_5=\bm{m}^*, D_6=\bm{n}^*, D_7=\bm{o}^*\}$ is a $\widehat{g}$-orthonormal frame  with dual frame  $B^{*}_{S^7}=\{\omega^1, \omega^2, \omega^3,\omega^4,\omega^5,\omega^6,\omega^7\}$.
\end{Lem}

\begin{thm}
    The vertical and horizontal sub-bundles of the $SU(2)$-bundle Hopf $(S^7, R)$ are given by
$$V(S^7)=\{D_1=\bm{i}^*, D_2=\bm{j}^*, D_{3}= \bm{k}^*\}, \quad H(S^7)=\{D_{4}=\bm{l}^*, D_5=\bm{m}^*, D_6=\bm{n}^*, D_7=\bm{o}^*\}.$$

The principal connection defined by the Kaluza-Klein metric $\widehat{g}$ is $\omega=\omega^1\otimes \bm{i}^*+\omega^2\otimes \bm{j}^*+\omega^3\otimes \bm{k}^*.$

The map $\beta\in\cin(S^3,\mathrm{Met}(\mathfrak{u}(2)))^{SU(2)}$ defined by $\widehat g$ is constant  and $\beta$ is $\Ad$-invariant.

    The curvature of the connection $\omega$ is
    \begin{align*}
     \Omega^\omega &=2[(-\omega^4\wedge\omega^5+\omega^6\wedge\omega^7)\otimes\bm{i}^*-(\omega^4\wedge\omega^6+\omega^5\wedge\omega^7)\otimes\bm{j}^*+(-\omega^4\wedge\omega^7+\omega^5\wedge\omega^6)\otimes\bm{k}^*].
    \end{align*}

The Lorentz endomorphism of $(S^7,\widehat{g})$ is
\begin{align*}
    \mathscr{F}^{\omega} &=2[(-\omega^4\otimes D_5+\omega^5\otimes D_4+\omega^6\otimes D_7-\omega^7\otimes D_6)\otimes\bm{i}^*-\\
    &-(\omega^4\otimes D_6-\omega^6\otimes D_4+\omega^5\otimes D_7-\omega^7\otimes D_5)\otimes\bm{j}^*+\\
    &+(-\omega^4\otimes D_7+\omega^7\otimes D_4+\omega^5\otimes D_6-\omega^6\otimes D_5)\otimes\bm{k}^*].
\end{align*}

The curvature modified metric of  $(S^7,\widehat{g})$  is
\begin{align*}
    \widehat{g}_{\mathscr{F}^\omega} &=(\omega^1\cdot\omega^5+\omega^2\cdot\omega^6+\omega^3\cdot\omega^7)\otimes D_4-(\omega^1\cdot\omega^4-\omega^2\cdot\omega^7+\omega^3\cdot\omega^6)\otimes D_5+\\
    &+(\omega^1\cdot\omega^7-\omega^2\cdot\omega^4+\omega^3\cdot\omega^5)\otimes D_6+(\omega^1\cdot\omega^6-\omega^2\cdot\omega^5-\omega^3\cdot\omega^4)\otimes D_7.
\end{align*}
\end{thm}

\textbf{Step 2} For any $\alpha\in\R$, let $\widetilde\Psi_\alpha: S^{1}\times S^{1}\longrightarrow S^{3}\subset\mathbb{R}^{4}$ be the map given by $\widetilde\Psi_\alpha(q_1,q_2)=\cos\alpha\,q_1+\sin\alpha\, q_2\cdot \boldsymbol{l}$. There is a commutative diagram 
\[\begin{tikzcd}
	{S^{3}\times S^{3}} && {S^{7}\subset\mathbb{R}^{8}} \\
	\\
	&& {S^{4}\subset\mathbb{R}^{5}}
	\arrow["\widetilde{\Psi}_\alpha", from=1-1, to=1-3]
	\arrow["\Psi_\alpha"', from=1-1, to=3-3]
	\arrow["\pi", from=1-3, to=3-3]
\end{tikzcd}\]
where $\Psi_\alpha\colon 	S^{3}\times S^{3} \to  S^{7}$ is the composition $\Psi_\alpha=\pi\circ \widetilde\Psi_\alpha$. Using once again \cite[Proposition 3.3.17]{baird2003harmonic} one straightforwardly proves that if $\alpha\neq\frac{\pi}{2}\mathbb Z$ then $ \widetilde\Psi_\alpha$ is an spherical harmonic immersion, that we call an $\alpha$-twisted spherical immersion, see \cite{dragomir1989sottovarieta} for the standard case when $\alpha=\frac{\pi}{4}$.  Therefore, $\Psi_\alpha$ is a generalized magnetic map.

\textbf{Step 3} Since $\beta$ is constant and $\ad$-invariant, by Corollary \ref{cor:difference-of-connections} we get $\bm{T}=0$, $\bm{\Theta}=0$. Hence, it follows from Corollary \ref{cor:equivalent-harmonic-equations}, that the Wong's equations satisfied by $\widetilde\Psi_\alpha$ are 
\begin{align*}
    \tau(\Psi_\alpha) &=-\pi_*(\Tr_g[\widetilde\Psi_\alpha^*((\omega^1\cdot\omega^5+\omega^2\cdot\omega^6+\omega^3\cdot\omega^7)\otimes D_4-(\omega^1\cdot\omega^4-\omega^2\cdot\omega^7+\omega^3\cdot\omega^6)\otimes D_5+\\
    &+(\omega^1\cdot\omega^7-\omega^2\cdot\omega^4+\omega^3\cdot\omega^5)\otimes D_6+(\omega^1\cdot\omega^6-\omega^2\cdot\omega^5-\omega^3\cdot\omega^4)\otimes D_7)])
\end{align*}
    and
    \begin{align*}
\delta^{(\nabla^{S^3\times S^3},\widetilde\Psi_\alpha^*{\nabla}^{\omega})}(\widetilde\Psi_\alpha^*(\omega^1\otimes \bm{i}^*+\omega^2\otimes \bm{j}^*+\omega^3\otimes \bm{k}^*)) &=0.
        \end{align*}
        By \cite[Proposition 3.3.17]{baird2003harmonic}, a somewhat lengthy computation shows that $\Psi_\alpha$ is uncharged (i.e $\tau(\Psi_\alpha)= 0$) if and only if $\alpha\in\frac{\pi}{4}\cdot\mathbb Z$; that is, if and only if  $\widetilde\Psi_\alpha$ is  the standard spherical harmonic  immersion.
\end{Ex}

\end{document}